\documentclass[a4paper,11pt]{amsart}

\usepackage[utf8]{inputenc}

\usepackage[a4paper,right=2.75cm,left=2.75cm,bottom=2.5cm, top=2.5cm]{geometry}

\usepackage{latexsym}
\usepackage{graphics}
\usepackage{amsmath}
\usepackage{amssymb}
\usepackage{bm}
\usepackage{dsfont}
\usepackage{mathrsfs}
\usepackage{enumerate,enumitem}
\usepackage{csquotes}
\usepackage{subcaption}
\usepackage{pgfplots,tikz}
\usetikzlibrary{decorations.pathreplacing}
\pgfplotsset{compat=1.10}
\usepgfplotslibrary{fillbetween}
\usetikzlibrary{patterns}
\usetikzlibrary{intersections, calc, angles}
\usepackage{caption}

\usepackage{hyperref}
\usepackage{cleveref}
\hypersetup{
	linktocpage=true,
	colorlinks=true,
	linkcolor=blue!70!black,
	citecolor=orange!40!red,
	urlcolor=magenta!80!black,
}
\usepackage[showonlyrefs]{mathtools}

\usepackage{etoolbox,kvoptions,logreq,pdftexcmds}

\numberwithin{equation}{section} 

\newtheorem{theorem}{Theorem}[section]
\newtheorem{corollary}[theorem]{Corollary}
\newtheorem{lemma}[theorem]{Lemma}
\newtheorem{proposition}[theorem]{Proposition}
\theoremstyle{definition} 
\newtheorem{definition}[theorem]{Definition}

\newtheorem{remark}[theorem]{Remark}
\newtheorem{assumption}[theorem]{Assumption}

\newcommand{\R}{\mathbb{R}}	
\newcommand{\N}{\mathbb{N}} 
\newcommand{\dx}{\,\mathrm{d}x}	
\newcommand{\ds}{\,\mathrm{d}S}	
\newcommand{\dr}{\,\mathrm{d}r}	
\renewcommand{\d}{\,\mathrm{d}}
\newcommand{\weak}{\rightharpoonup}
\newcommand{\nnu}{\bm{\nu}}  

\newcommand{\norm}[1]{\left\lVert #1 \right\lVert}
\newcommand{\abs}[1]{\left| #1 \right|}
\newcommand{\sub}{\subseteq}
\newcommand\restr[1]{\raisebox{-.5ex}{$|$}_{#1}}

\DeclareMathOperator{\supp}{\mathrm{supp}}

\DeclareMathOperator{\dist}{\mathrm{dist}}
\DeclareMathOperator{\dive}{\mathrm{div}}

\newenvironment{bvp}{\left\{\begin{aligned}  }{\end{aligned}\right.}

\begin{document}
	
	\title[Boundary regularity of the free interface]{Boundary regularity of the free interface in spectral optimal partition problems}

	\author[R. Ognibene]{Roberto Ognibene}
	\address{Roberto Ognibene
		\newline \indent Dipartimento di Matematica
		\newline \indent Universit\`a di Pisa
		\newline\indent  Largo Bruno Pontecorvo, 5, 56127 Pisa, Italy}
	\email{roberto.ognibene@dm.unipi.it}
	
	\author[B. Velichkov]{Bozhidar Velichkov}
	\address{Bozhidar Velichkov
		\newline \indent Dipartimento di Matematica
		\newline \indent Universit\`a di Pisa
		\newline\indent  Largo Bruno Pontecorvo, 5, 56127 Pisa, Italy}
	\email{roberto.ognibene@dm.unipi.it}

	\keywords{Optimal partitions, segregated systems, boundary regularity, boundary unique continuation, frequency function, monotonicity formulas, epiperimetric inequality}
	\subjclass{
		35R35,   
		35B40,   
		35J57,   
		49Q10   
	}
	
	\pagestyle{plain}

	\begin{abstract}
		We consider the problem of optimal partition of a domain with respect to the sum of the principal eigenvalues and we prove for the first time regularity results for the free interface up to fixed boundary. All our results are quantitative and, in particular, we obtain fine estimates on the continuity of the solutions and the oscillation of the free interface (in terms of the modulus of continuity of the normal vector of the fixed boundary), even in the case of domains with low (Dini-type) regularity. Our analysis is based on an Almgren-type monotonicity formula at boundary points and an epiperimetric inequality at points of low frequency, which, together, yield an explicit rate of convergence for blow-up sequences and the boundary strong unique continuation property. Exploiting our quantitative blow-up analysis, we manage to prove clean-up results near one-phase and two-phase points. We define  the notion of free interface inside the fixed boundary, and we prove that the subset of points of minimal frequency is regular and that the interior free interface is approaching the boundary orthogonally in a smooth way. 
		
	\end{abstract}
	
	\maketitle
	
	\tableofcontents
	
	\section{Introduction and state of the art}

	Let $d\in\mathbb{N}$, $d\geq 2$ and let $D\sub\R^d$ be an open, bounded set.
	For any fixed number $N\in\mathbb{N}$, with $N\geq 2$, we consider the family of $N$-partitions of $D$, i.e.
	\begin{equation*}
		\mathcal{P}_N(D):=\Big\{(\Omega_1,\dots,\Omega_N)\colon \Omega_i\sub D~\text{is open, bounded and connected}, ~\Omega_i\cap \Omega_j=\emptyset~\text{for }i\neq j\Big\}.
	\end{equation*}
	Given an $N$-partition $(\Omega_1,\dots,\Omega_N)\in\mathcal{P}_N(D)$ of $D$, we consider the first eigenvalue of the Dirichlet-Laplacian on each subdomain $\Omega_i$, $i=1,\dots,N$, that is
	\[
	\lambda_1(\Omega_i):=\inf\left\{\frac{\displaystyle \int_{\Omega_i}\abs{\nabla u}^2\dx}{\displaystyle \int_{\Omega_i}u^2\dx}\colon u\in H^1_0(\Omega_i)\setminus\{0\} \right\}.
	\]
	In this paper we study partitions which are optimal for the sum of the principal frequencies of the $N$ subdomains. Precisely, we consider the following variational problem
	\begin{equation}\label{eq:optimal_part}
		\inf\left\{ \sum_{i=1}^N\lambda_1(\Omega_i)\colon (\Omega_1,\dots,\Omega_N)\in\mathcal{P}_N(D) \right\},
	\end{equation}
	which has been studied in several frameworks such as dynamics of populations (see for instance \cite{CTV2003, CTV2005advances,CTV2005indiana, CTVFucick2005}), harmonic maps with values in singular spaces (see \cite{CL2007,CL2008,CL2010}) and shape optimization (\cite{BBH1998}). In the rest of the introduction, we fix some notation and we briefly recall the known results on the optimal partition problem, while the main results of the present paper are given in Section \ref{sec:main_results}.

	\subsection*{Existence of optimal partitions} 
	
	In \cite{BBH1998}, within a wider framework, Bucur, Buttazzo and Henrot proved existence of an optimal partition in the class of partitions made of quasi-open, pairwise disjoint subdomains, and their techniques are essentially based on direct minimization and $\gamma$-convergence methods. Nevertheless, completing such result by proving that the optimal partition is made of open sets is not a trivial task. 
	This was achieved by Conti, Terracini and Verzini in \cite{CTVFucick2005} (see also \cite{CTV2002,CTV2003,CTV2005indiana}) and by Caffarelli and Lin in \cite{CL2007}, where the authors considered a suitable relaxed formulation of problem \eqref{eq:optimal_part} (see \eqref{eq:opt_part_variat}), which is equivalent to the minimization among quasi-open, pairwise disjoint sets, and then, through PDEs methods, recovered existence of a solution to \eqref{eq:optimal_part}. 
	In order to be more precise on that, let us introduce the functional setting. As in \cite{CTVFucick2005,CL2007}, we work in the space $H^1_{0,N}(D)$ of $N$ segregated Sobolev densities which vanish on $\partial D$. Precisely, for any open $\mathcal{O}\sub\R^d$ we define\footnote{For sake of simplicity, in the present paper we do not distinguish, for what concerns the notation, between scalar-valued and vector-valued functions. }
	\[
	H^1_{0,N}(\mathcal{O}):=H^1_{s,N}(\mathcal{O})\cap (H^1_0(\mathcal{O}))^N,
	\]
	where $H^1_{s,N}$ denotes the space of $N$-vectors of segregated $H^1$-functions, that is
	\[
	H^1_{s,N}(\mathcal{O}):=\{u\in (H^1(\mathcal{O}))^N\colon u_iu_j=0,~\text{a. e. in }\mathcal{O}~\text{for all }i,j=1,\dots,N,~i\neq j\},
	\]
	and we point out that $H^1_{0,N}(D)$ coincides with the space $H^1_0(D,\Sigma_N)$ introduced in \cite{CL2007}, where
	\begin{equation*}
		\Sigma_N:=\left\{X\in\R^N\colon X_iX_j=0,~\text{for all }i,j=1,\dots, N,~i\neq j\right\}.
	\end{equation*}
	For any open $\mathcal{O}\sub \R^d$ and any $u
	\in H^1_{s,N}(\mathcal{O})$ such that $u_i\not\equiv 0$ for all $i=1,\dots,N$, we set
	\[
	J_N(u,\mathcal{O}):=\sum_{i=1}^N\frac{\displaystyle\int_\mathcal{O}\abs{\nabla u_i}^2\dx}{\displaystyle\int_\mathcal{O} u_i^2\dx},
	\]
	and we consider the minimization problem
	\begin{equation}\label{eq:opt_part_variat}
		\inf\left\{ J_N(u,D)\colon u\in H^1_{0,N}(D),~ u_i\geq 0, ~u_i\not\equiv 0~\text{for all }i=1,\dots,N \right\}.
	\end{equation}
	We fix the following notation, which will be employed throughout the whole paper. We assume
	\begin{equation*}
		u=(u_1,\dots,u_N)\in H^1_{0,N}(D)~\text{to be a minimizer of \eqref{eq:opt_part_variat}}
	\end{equation*}
	and we denote by
	\begin{equation}\label{eq:Omega_i}
		\Omega_i=\Omega_i^u:=\{x\in D\colon u_i(x)>0\},\quad\text{for }i=1,\dots,N.
	\end{equation}
	Since every positive multiple of $u$ is still a minimizer, it is not restrictive to assume that
	\[
	\int_D u_i^2\dx=\int_{\Omega_i}u_i^2\dx=1\quad\text{for all }i=1,\dots,N.
	\]
	Then problem \eqref{eq:optimal_part} turns out to be equivalent to \eqref{eq:opt_part_variat}. This is not trivial and it is, in fact, one of the contributions of \cite{CTVFucick2005,CL2007}, where the authors prove Lipschitz continuity of minimizers. Precisely, we have the following theorem, whose proof can be found in \cite[Theorem 2.2]{CTVFucick2005} and \cite[Proposition 3]{CL2007}.

	\begin{theorem}[Regularity of eigenfunctions, \cite{CTVFucick2005,CL2007}]\label{thm:int_eigen}
		Let $u=(u_1,\dots,u_N)\in H^1_{0,N}(D)$ be a minimizer of \eqref{eq:opt_part_variat} and let $\Omega_i$ be as in \eqref{eq:Omega_i}. Then, for all $i=1,\dots,N$, $u_i\in C^{0,1}_{\textup{loc}}(D)$ and there holds
		\[
		\begin{bvp}
			-\Delta u_i&=\lambda_1(\Omega_i)u_i, &&\text{in }\Omega_i, \\
			u_i &=0, &&\text{on }\partial \Omega_i,
		\end{bvp}
		\]
		in a weak sense. In particular $u_i\in C^\infty(\Omega_i)$. Moreover, if $D$ is of class $C^1$, then $u_i\in C^{0,1}(D)$ for all $i=1,\dots,N$.
		
	\end{theorem}
	We notice that a single component $u_i$ of a minimizer of \eqref{eq:opt_part_variat} does not satisfy an equation on the full domain $D$: indeed the space $H^1_{0,N}(D)$ does not allow arbitrary outer variations of the form $u+t\varphi$, for $\varphi\in C_c^\infty(D)$, $t\in\R$, since the segregation property is not preserved for the competitor. Nevertheless, inner (and some suitable outer) variations can still be  performed, and this allows to deduce a series of extremality conditions, which have been obtained in \cite{CTVFucick2005} and which we recall in \Cref{sec:prelim}. We remark that, among the class of inner variations, one can consider (localized) dilations: these play a central role, since they lead to the so called Pohozaev identity, which has a crucial role in the analysis of free boundary regularity.
	
	\subsection*{Interior regularity of the free interfaces}	
	
	Once existence of a minimizer to \eqref{eq:optimal_part} is established, one is naturally led to investigate the geometrical properties of the interfaces separating the optimal subdomains $\Omega_i$, and these turn out to be deeply entangled with the regularity properties of the eigenfunctions corresponding to $\lambda_1(\Omega_i)$, with the equations they satisfy and with the reflection properties between eigefunctions of touching subdomains.

	Let us consider the interface which separates the segregated densities $u_i$
	\begin{equation}\label{eq:free_boundary}
		\mathcal{F}(u):=\bigcup_{i=1}^N\partial \Omega_i\cap D
	\end{equation}
	and let us also take into consideration the zero set of the minimizer $u$, that is
	\begin{equation}\label{eq:zero_set}
		\mathcal{Z}(u):=\left\{ x\in D\colon u_i(x)=0,~\text{for all }i=1,\dots,N \right\}.
	\end{equation}
	Summing up the results obtained in \cite{CTV2005indiana} and \cite{CL2007}, we have the following theorem describing the regularity of the free boundaries in the interior of $D$.
	
	\begin{theorem}[Regularity of the free boundary, \cite{CTV2005indiana,CL2007}]\label{thm:int_free_bound}
		Let $u=(u_1,\dots,u_N)\in H^1_{0,N}(D)$ be a minimizer of \eqref{eq:opt_part_variat} and let $\mathcal{F}(u)$ be as in \eqref{eq:free_boundary} and $\mathcal{Z}(u)$ as in \eqref{eq:zero_set}. Then the free boundary $\mathcal{F}(u)$ coincides with the zero set $\mathcal{Z}(u)$, i.e. 
		$
		\mathcal{F}(u)=\mathcal{Z}(u),
		$
		and it can be decomposed into two disjoint sets $\mathcal{R}(u)$ and $\mathcal{S}(u)$
		\[
		\mathcal{F}(u)=\mathcal{R}(u)\cup \mathcal{S}(u),
		\]
		where $\mathcal{R}(u)$ is, locally, a $(d-1)$-dimensional manifold of class $C^{2,\alpha}$ (for some $0<\alpha<1$) and $\mathcal{S}(u):=\mathcal{F}(u)\setminus\mathcal{R}(u)$ is a closed set (in the topology of $\mathcal{F}(u)$) with Hausdorff dimension not exceeding $d-2$. Moreover, in a neighborhood of any $x_0\in \mathcal{R}(u)$ there are exactly two components of the optimal partition, that is, 
		there are $j,k\in\{1,\dots,N\}$, with $j\neq k$, and $r_0>0$ such that
		\begin{align*}
			&\Omega_j\cap B_r(x_0)\neq \emptyset,\quad\Omega_k\cap B_r(x_0)\neq \emptyset, \\
			&\Omega_i\cap B_r(x_0)=\emptyset~\text{for all }i\neq j
		\end{align*}
		for all $r\leq r_0$.
	\end{theorem}

	It is also worth mentioning that in \cite{Alper}, by exploiting the techniques introduced in \cite{NV1,NV2,dLMSV}, it was proved that the singular set $\mathcal{S}(u)$ has locally finite $(d-2)$-dimensional Hausdorff measure and is $(d-2)$-rectifiable, i.e. it can be covered by countably many $C^1$-manifolds of dimension $d-2$ up to a set of $(d-2)$-dimensional Hausdorff measure zero.

	\section{Main results}\label{sec:main_results}

	In the present paper we investigate for the first time the regularity of the free interfaces (arising in segregated problems) up to the fixed boundary $\partial D$. 
	As in \cite{CTVFucick2005,CL2007}, we choose as model problem the spectral optimal partition problem \eqref{eq:optimal_part} in its equivalent formulation \eqref{eq:opt_part_variat}. More precisely, we denote
	\begin{equation}\label{eq:free_bound_fixed}
		\mathcal{F}_{\partial D}(u):=\partial D\cap \overline{\mathcal{F}(u)},
	\end{equation}
	where $\mathcal F(u)$ is the interior interface defined in \eqref{eq:free_boundary}, and 
	we aim at understanding the local structure of $\mathcal{F}_{\partial D}(u)$ and describing how $\mathcal{F}(u)$ approaches $\mathcal{F}_{\partial D}(u)$. It is clear that, at this stage, the local structure of the fixed boundary $\partial D$ will strongly influence the behavior of $\mathcal{F}(u)$ and $\mathcal{F}_{\partial D}(u)$. In this regard, we assume $\partial D$ to be of class $C^1$, in the sense that $D$ can be locally described, near the boundary and up to a change of coordinates, as the epigraph of a $C^1$ function; moreover, we assume that the gradient of this function possesses a modulus of continuity satisfying certain integrability properties. In order to be more precise, let us introduce the following notation, which we adopt throughout the whole paper: for $x\in\R^d$ and $r>0$ we denote by $B_r(x)$ the $d$-dimensional ball of radius $r$ and center $x$ and
	\begin{equation*}
		B_r:=B_r(0),\quad B_r^+:=B_r\cap \R^d_+\quad\text{and}\quad B_r':=B_r\cap\partial\R^d_+,
	\end{equation*}
	where $\R^d_+:=\{(x_1,\dots,x_d)\colon x_d>0\}$. Moreover, for $x\in\R^d$ we may write $x=(x',x_d)$, with $x'=(x_1,\dots,x_{d-1})$. We now make the following assumption of the boundary of the domain $D$.

	\begin{assumption}\label{ass:domain}
		There exists a radius $R_{\partial D}>0$ such that the following holds.
		\begin{itemize}
			\item[(I)]\label{ass:domain_1} For any $x_0\in\partial D$ there exists an orthogonal matrix $\textbf{Q}=\textbf{Q}_{x_0}\in O(d)$ and a function $\varphi=\varphi_{x_0}\in C^1(B_{R_{\partial D}}')$ such that $\varphi(0)=|\nabla\varphi(0)|=0$ and
			\begin{align*}
				&D\cap B_R(x_0)=\{ \textbf{Q}x+x_0\colon x'\in B_R',~x_d>\varphi(x')  \}\cap B_R(x_0), \\
				&\partial D\cap B_R(x_0)=\{ \textbf{Q}x+x_0\colon x'\in B_R',~x_d=\varphi(x')  \}\cap B_R(x_0),
			\end{align*}
			for all $R\leq R_{\partial D}$.
			\item[(II)]\label{ass:domain_2} There exists a function $\sigma\colon[0,2R_{\partial D}]\to [0,+\infty)$ such that $\sigma\in C([0,2R_{\partial D}])$, $\sigma(0)=0$, $\sigma$ is non-decreasing and for all $x_0\in\partial D$ there holds
			\[
			\abs{\nabla\varphi_{x_0}(x')-\nabla\varphi_{x_0}(y')}\leq \sigma(\abs{x'-y'}) ~\text{for all }x',y'\in B_{R_{\partial D}}'.
			\]
			\item[(III)]\label{ass:domain_3}  There exists $\sigma_0\colon [0,2R_{\partial D}]\to[0,\infty)$
			\begin{align}
				& \sigma_0\in C^1(0,2R_{\partial D}),\quad(r^{-m_d}\sigma_0(r))'\leq 0 ~\text{for all }r\in (0,2R_{\partial D}), \label{eq:dini_hp_1}\\
				&\int_0^{2R_{\partial D}}\frac{\sigma_0(r)}{r}\d r<\infty\quad\text{and}\quad \int_0^{2R_{\partial D}}\frac{1}{r\sigma_0(r)}\int_0^r\frac{\sigma(t)}{t}\d t<\infty,\label{eq:dini_hp_2}
			\end{align}
			where ${m_d:=d\epsilon_{\textup{bd}} /4}$,  $\epsilon_{\textup{bd}}>0$ being the dimensional constant from \Cref{thm:epi2}.
		\end{itemize}
		Moreover, it is not restrictive (see \Cref{lemma:modulus}) to assume that
		\begin{itemize}
			\item[(IV)] $\sigma\in C^2(0,2R_{\partial D})$ and
			\[
			(\sigma(r)/r)'\leq 0,\quad\abs{\sigma'(r)}\leq \frac{2}{r}\sigma(r),\quad \abs{\sigma''(r)}\leq \frac{4}{r^2}\sigma(r)
			\]
			for all $r\in(0,2R_{\partial D})$.
		\end{itemize}
		
	\end{assumption}
	
	Essentially, \Cref{ass:domain} requires some integrability conditions on the modulus of continuity of the normal vector to $\partial D$. In particular, we remark that the second condition in \eqref{eq:dini_hp_1} says that $\sigma_0$ shall not grow faster than a power, while the second condition in \eqref{eq:dini_hp_2} implies that $\sigma$ must be of class $2$-Dini on $[0,2R_{\partial D}]$. Moreover, \Cref{ass:domain} is satisfied for some known classes of regular domains. In particular, if $\sigma\in C^{0,\alpha}(0,2R_{\partial D})$ (i.e. the domain $D$ is of class $C^{1,\alpha}$), then we can take $\sigma_0(r)=r^{\alpha_0}$ for any $0<\alpha_0<\min\{\alpha,m_d\}$. We are actually able to treat domains which are less then $C^{1,\alpha}$, by entering the class of $C^{1,\alpha\textup{-Dini}}$ domains: these are defined as domains whose boundary is locally described by $C^1$ functions whose gradient possesses a modulus of continuity which is $\alpha$-Dini. In turn, the notion of $\alpha$-Dini function, for a real $\alpha\geq 1$, is given in the following.
	
	\begin{definition}\label{def:alpha_dini}
		Let $R>0$ and $\alpha\in\R$, $\alpha\geq 1$. A function $f\colon [0,R]\to [0,+\infty)$ is said to be of \emph{class $\alpha$-Dini} in $[0,R]$ if it is continuous, non-decreasing, and satisfies
		\begin{equation*}
			\int_0^R\frac{f(r)|\log r|^{\alpha-1}}{r}\d r<\infty.
		\end{equation*}
	\end{definition}
	This definition naturally extends the notion of $j$-Dini function, with $j\in\N$, in view of \Cref{lemma:alpha-dini}. Hence, the least regular domain we are able to handle (without convexity assumptions) is a $C^{1,\alpha\textup{-Dini}}$ domain with $\alpha>3$; in this case, $\sigma$ is $\alpha$-Dini and in this case we can choose $\sigma_0$  to be $\sigma_0(r)=|\log r|^{-(1+\alpha_0)}$ for any $0<\alpha_0<\alpha-3$. We remark that, in case $D$ is convex, $C^1$-regularity could be enough in order to reach the same results as in the present paper.   \medskip
	
	Our first main result states that an asymptotic expansion of the minimizer $u$ holds true in a neighborhood of any boundary point. More precisely, we have the following (we refer to Figure \ref{fig:1} for a possible visualization of it in the two dimensional case). 
	
	\begin{theorem}[Taylor expansion]\label{thm:taylor}
		Let $u=(u_1,\dots,u_N)\in H^1_{0,N}(D)$ be a minimizer of \eqref{eq:opt_part_variat}. Let $x_0\in\partial D$ and let $\nnu(x_0)$ be the exterior normal to $\partial D$ at $x_0$. Then, exactly one of the following is satisfied.
		\begin{enumerate}
			\item[1)] There exist $j\in\{1,\dots,N\}$ and $a_{x_0,1}>0$ such that
			\begin{align*}
				&u_j(x)=a_{x_0,1}(-(x-x_0)\cdot\nnu(x_0))^++o(|x-x_0|)\quad\text{as }x\to x_0, \\
				&u_i(x)\equiv 0~\text{in a neighborhood of $x_0$ for all }i\neq j.
			\end{align*}
			In this case, we say that $x_0\in\omega_j$.
			\item[2)] There are $j\neq k\in\{1,\dots,N\}$, $a_{x_0,2}>0$ and $\bm{e}_{x_0}\in\partial B_1$ such that $\bm{e}_{x_0}\cdot\nnu(x_0)=0$ and
			\begin{align*}
				&u_j(x)=a_{x_0,2}((x-x_0)\cdot\bm{e}_{x_0})^+(-(x-x_0)\cdot\nnu(x_0))^++o(|x-x_0|^2)\quad\text{as }x\to x_0, \\
				&u_k(x)=a_{x_0,2}((x-x_0)\cdot\bm{e}_{x_0})^-(-(x-x_0)\cdot\nnu(x_0))^++o(|x-x_0|^2)\quad\text{as }x\to x_0.
			\end{align*}
			\item[3)] $u_i(x)=o(|x-x_0|^2)$ as $x\to x_0$ for all $i\in\{1,\dots,N\}$ and there exists $j,k\in\{1,\dots,N\}$, with $j\neq k$, such that $\Omega_j\cap B_r(x_0)\neq \emptyset$ and $\Omega_k\cap B_r(x_0)\neq \emptyset$ for all $r>0$.
		\end{enumerate}
	\end{theorem}

	\begin{figure}[h]
		\centering
		\begin{tikzpicture}
			\coordinate (O) at (0,0);
			\draw[fill=white, opacity=0.05] (0,0) circle [radius = 20mm];
			\draw[thick] (O) circle [very thick, radius=2cm, name path=c];
			
			\draw [thick, color=red, name path=1] plot [smooth] coordinates  {(-0.2,-0.2)  (-0.4,-0.8) (-0.2,-1.5) (-0.23,-1.97)};
			
			\draw [thick, color=red, name path=2] plot [smooth] coordinates  {(-0.2,-0.2)  (0.8,-0.1) (1.5,0.6) (1.73,1)};
			
			\draw [thick, color=red, name path=3] plot [smooth] coordinates  {(-0.2,-0.2)  (-0.8,0.5) (-1.5,0.6) (-1.73,1)};
			
			\draw [thick, color=red, name path=4] plot [smooth] coordinates  {(-1.73,1)  (-1.0,1.1) (0.2,0.9) (1.0,1.1) (1.73,1)};
			
			\draw[thick] (1.73,1) circle [very thick, radius=0.02cm];
			\draw[thick] (-1.73,1) circle [very thick, radius=0.02cm];
			\draw[thick] (-0.23,-1.98) circle [very thick, radius=0.02cm];

			\draw node at (1.9,1.2) {$Y$};
			\draw node at  (-1.9,1.2) {$Z$};
			\draw node at  (-0.0,-1.8)  {$X$};

			\draw node at (-1,-0.5) {$\Omega_1$};
			\draw node at (0.8,-1) {$\Omega_2$};
			\draw node at (0.1,1.5) {$\Omega_3$};
			\draw node at (0.3,0.4) {$\Omega_4$};
			
			
			\draw[thick] (O) circle [very thick, radius=2cm];
			\draw[densely dashed] (-0.23,-1.98) circle [radius = 4mm];
			\draw[very thick, ->] plot [smooth] coordinates {(0.17,-2.1) (1.6,-2.1) (2.8,-1.7) (3.6,-1.25)};
			
			\draw[densely dashed] (1.73,1) circle [radius = 4mm];
			\draw[very thick, ->] plot [smooth] coordinates {(2.1,1.1) (4.3,1.6) (5.8,1.7) (7.2,1.4) (8.8,0.55)};
			
			\begin{scope}[shift={(5.5,-1)}]
				\draw[densely dashed,  name path=g_arc_1] (2.2,0) arc [start angle=0, end angle = 90,x radius = 22mm, y radius = 22mm];
				\draw[densely dashed, name path=g_arc_2] (0,2.2) arc [start angle=90, end angle = 180,x radius = 22mm, y radius = 22mm];
				\draw [thick, red, name path=hor] plot [smooth] coordinates {(0,0) (0,2.2)};
				\draw [thick, name path=hor] plot [smooth] coordinates {(-2.2,0) (2.2,0)};
				\draw node at (-1.0,0.85) {$\Omega_1$};
				\draw node at (1.0,0.85) {$\Omega_2$};
				\draw [decorate,decoration={brace,mirror, amplitude=3pt}]
				(-2.17,-0.03) -- (-0.03,-0.03);
				\draw node at (-1.1,-0.3) {\small$\omega_1$};
				\draw [decorate,decoration={brace,mirror, amplitude=3pt}]
				(0.03,-0.03) -- (2.17,-0.03);
				\draw node at (1.1,-0.3) {\small$ \omega_2$};
				\draw [thick, ->] plot [smooth] coordinates {(0,0) (0,-0.6)};
				\draw node at (0.5,-0.6) {\small$\nnu(X)$};
				\draw[thick] (0,0) circle [very thick, radius=0.02cm];
			\end{scope}
			
			\begin{scope}[shift={(11,0)}, rotate=140]
				\draw[densely dashed, name path=g_arc_1] (2.2,0) arc [start angle=0, end angle = 90,x radius = 22mm, y radius = 22mm];
				\draw[densely dashed, name path=g_arc_2] (0,2.2) arc [start angle=90, end angle = 180,x radius = 22mm, y radius = 22mm];
				\draw [thick, red,name path=hor] plot [smooth] coordinates {(0,0) (1.1,1.9)};
				\draw [thick, red,name path=hor] plot [smooth] coordinates {(0,0) (-1.1,1.9)};
				\draw [thick, name path=hor] plot [smooth] coordinates {(-2.2,0) (2.2,0)};
				\draw node at (-1.3,0.65) {$\Omega_2$};
				\draw node at (1.3,0.65) {$\Omega_3$};
				\draw node at (0.05,1.4) {$\Omega_4$};
				\draw [decorate,decoration={brace,mirror, amplitude=3pt}]
				(-2.17,-0.03) -- (-0.03,-0.03);
				\draw node at (-1.1,-0.3) {\small$\omega_2$};
				\draw [decorate,decoration={brace,mirror, amplitude=3pt}]
				(0.03,-0.03) -- (2.17,-0.03);
				\draw node at (1.1,-0.3) {\small$ \omega_3$};
				\draw [thick, ->] plot [smooth] coordinates {(0,0) (0,-0.6)};
				\draw node at (-0.4,-0.6) {\small$\nnu(Y)$};
				\draw[thick] (0,0) circle [very thick, radius=0.02cm];
			\end{scope}
		\end{tikzpicture}
		\caption{A partition of $D=B_1$ in 4 domains (on the left); the boundary point $X$ is \textit{regular}, while $Y$ and $Z$ are \textit{singular}. The limit behavior (blow-up) of the free interface at $X$ and $Y$ is described on the pictures in the middle and on the right.}
		\label{fig:1}
	\end{figure}
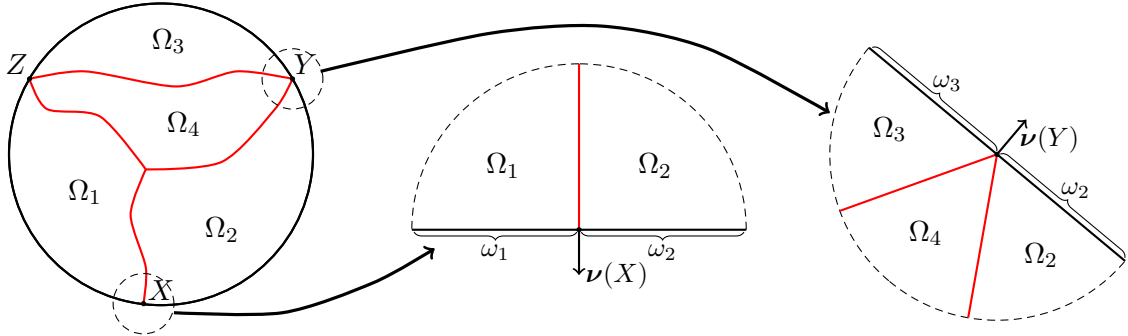
	
	\Cref{thm:taylor} allows us to identify the \enquote{trace} of the optimal partition $(\Omega_1,\dots,\Omega_N)$ on $\partial D$ and to provide a first characterization of $\mathcal{F}_{\partial D}(u)$. Namely, for any $j=1,\dots,N$, we define $\omega_j\sub\partial D$ as the set of points $x_0\in\partial D$ for which case \textit{1)} in \Cref{thm:taylor} is verified and we call $\omega_j$ the \emph{trace} of $\Omega_j$\footnote{We observe that $\omega_j$ could be empty.}. This theorem also implies that the functions $u_j$ are differentiable at all points of $\partial D$ and that  
	\begin{equation*}
		\omega_j=\left\{x\in\partial D\colon \partial_{\nnu} u_j(x)<0 \right\}.
	\end{equation*}
	Again in view of \textit{1)} of \Cref{thm:taylor} the sets $\omega_j$, $j=1,\dots,N$, are disjoint (relatively) open subsets of $\partial D$ satisfying 
	\begin{equation}\label{eq:characterizzation_omega_j-parte1}
		\omega_j\subseteq \mathrm{Int}_{\partial D}\,(\partial \Omega_j\cap \partial D).
	\end{equation}
	We point out that the previous inclusion may be strict: indeed, in a situation like the picture on the right in \Cref{fig:1}, with $(\Omega_3,\Omega_4,\Omega_2)$ being replaced by $(\Omega_2,\Omega_4,\Omega_2)$, in a neighborhood of $Y$ the set $\mathrm{Int}_{\partial D}\,(\partial \Omega_2\cap \partial D)$ is the whole segment, while $\omega_2=\mathrm{Int}_{\partial D}\,(\partial \Omega_2\cap \partial D)\setminus \{Y\}$.
	
	Furthermore, combining \Cref{thm:taylor} with some topological arguments and exploiting the unique continuation in the interior, we are able to provide a characterization of $\omega_j$ and of the free boundary $\mathcal{F}_{\partial D}(u)$, see \Cref{prop:topology}. In particular, we will show that
	\begin{equation}\label{eq:intro-omega-i-is-equal-to-Omega-i}
		\mathrm{Int}_{\partial D}(\overline{\omega_j})=\mathrm{Int}_{\partial D}(\partial \Omega_j\cap\partial D)\quad\text{for all }j=1,\dots,N,
	\end{equation}
	that $\mathcal{F}_{\partial D}(u)$ has empty interior, and that 
	\begin{equation}\label{eq:F-characterization-union-of-boundaries}
		\mathcal{F}_{\partial D}(u)=\bigcup_{j=1}^N\partial_{\partial D}\omega_j.
	\end{equation}
	We notice that \eqref{eq:intro-omega-i-is-equal-to-Omega-i} already excludes some wild behaving interfaces like the oscillating one pictured on Figure \ref{fig:2} (on the left). 
	
	\begin{figure}[h]
		\centering
		\includegraphics[scale=0.28]{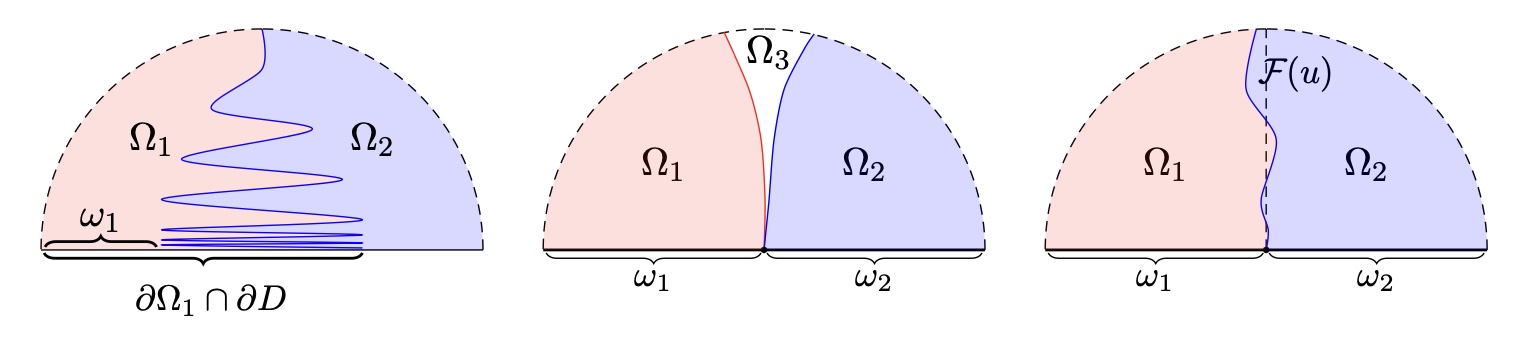}
		\caption{An oscillating free boundary (on the left), a cusp-like singularity (in the middle), and a regular free interface (on the right). We show that only the interface on the right can arise from an optimal partition.}
		\label{fig:2}
	\end{figure}
	
	We will next prove a regularity result for $\mathcal F_{\partial D}(u)$. We start by defining the \emph{regular} and \emph{singular} parts of $\mathcal{F}_{\partial D}(u)$ as follows. 
	
	\begin{definition}\label{def:reg_sing}
		If $x_0\in\mathcal{F}_{\partial D}(u)$, then in view of \eqref{eq:F-characterization-union-of-boundaries}, we have that either \textit{2)} or \textit{3)} of \Cref{thm:taylor} occurs: in the former case, we say that that $x_0$ is a \emph{regular point}, and we write $x_0\in\mathcal{R}_{\partial D}(u)$, while in the latter case we say that $x_0$ is a \emph{singular point}, and we write $x_0\in\mathcal{S}_{\partial D}(u)$. In view of \Cref{thm:taylor} the two sets $\mathcal{R}_{\partial D}(u)$ and $\mathcal{S}_{\partial D}(u)$ are disjoint. 
	\end{definition}
	
	\begin{remark}
		We notice that in the above definition whether a point $x_0\in \mathcal F_{\partial D}(u)$ is regular or singular is determined not by the smoothness of the sets $\omega_j\subset \partial D$, but by the behavior of the function $u:\overline D\to\R^N$ around $x_0$. For instance, on Figure \ref{fig:1}, the interfaces $X=\partial\omega_1\cap\partial\omega_2$ (in the middle) and $Y=\partial\omega_2\cap\partial\omega_3$ (on the right) are both isolated points, but still $X$ is regular, while $Y$ is singular. This situation is similar to the one in the thin-obstacle problem, where the boundary of a smooth set in the hyperplane can be composed of points of frequency $3/2$ (regular points), but it can also be entirely made of (singular) points of frequency $2m-1/2$ with $m\in\N$.
	\end{remark}

	The following is our main result about the free interfaces separating the segregated densities inside the fixed boundary $\partial D$.	
	\begin{theorem}\label{thm:fixed_free_bound}
		Let $u=(u_1,\dots,u_N)\in H^1_{0,N}(D)$ be a minimizer of \eqref{eq:opt_part_variat}, let $\omega_i\sub\partial D$ be as in \Cref{thm:taylor} for $i=1,\dots,N$, and let $\mathcal{F}_{\partial D}(u)$ be as in \eqref{eq:free_bound_fixed}. Then, 
	\[
	\mathcal{F}_{\partial D}(u)=\mathcal{R}_{\partial D}(u)\cup \mathcal{S}_{\partial D}(u),
	\]
	where $\mathcal{R}_{\partial D}(u)$ and $\mathcal{S}_{\partial D}(u)$ are as in \Cref{def:reg_sing} and satisfy the following: $\mathcal{S}_{\partial D}(u)$ is a relatively closed set and $\mathcal{R}_{\partial D}(u)$ is, locally, a $(d-2)$-dimensional submanifold of class $C^1$. Moreover, every regular point $x_0\in \mathcal{R}_{\partial D}(u)$ sees exactly two components of the trace of the optimal partition, that is: for any $x_0\in\mathcal{R}_{\partial D}(u)$ there exist $j,k\in\{1,\dots,N\}$, with $j\neq k$, and $r_0>0$ such that 
	\begin{align*}
		&\omega_j\cap B_r(x_0)\neq \emptyset,\quad\omega_k\cap B_r(x_0)\neq \emptyset, \\
		&\omega_i\cap B_r(x_0)=\emptyset~\text{for all }i\neq j\,,
	\end{align*}
	for all $r\leq r_0$.
\end{theorem}

\begin{remark}
	We also obtain the explicit modulus of continuity of the normal vector to $\mathcal{R}_{\partial D}(u)$. Precisely, if $\sigma_0$ is as in \Cref{ass:domain} and
	\[
	\Upsilon(r):=r^2\left(\int_0^r\frac{\sigma_0(t)}{t}\d t\right)^{\frac{1}{2}},
	\]
	which is invertible in its domain, then the modulus of continuity of the normal vector of $\mathcal{R}_{\partial D}(u)$ is
	\[
	\theta(r):=\left(\int_0^{\Upsilon^{-1}(r)}\frac{\sigma_0(t)}{t}\d t\right)^{\frac{1}{2}}.
	\]
	Moreover, thanks to \eqref{eq:sigma_0_power}, one can easily see that 
	\[
	\theta(r)\leq C_\theta\left(\int_0^{r^{\alpha_\theta}}\frac{\sigma_0(t)}{t}\d t\right)^{\frac{1}{2}},
	\]
	for some $C_\theta>0$ and $\alpha_\theta\in (0,1)$ depending only $\sigma_0(R_{\partial D})$ and $m_d$ (see \Cref{ass:domain}). Hence, we can observe that, if $\sigma(r)=r^{\alpha}$, then $\theta(r)\leq C_1r^{\alpha'}$ for some $0<\alpha'<\alpha$ and $C_1>0$, while if $\sigma$ is $\alpha$-Dini with $\alpha>3$, then $\theta(r)\leq C_2|\log r|^{-\alpha'}$ for some $0<\alpha'<\alpha-3$ and $C_2>0$.
\end{remark}

Nevertheless, even when the traces $\omega_j$ are smooth subsets of $\partial D$, some non-smooth situations may still occur, a priori: for instance, there could be a boundary point $x_0\in\mathcal{R}_{\partial D}(u)$ which sees exactly two smooth boundary components in $\partial D$ (in view of \Cref{thm:fixed_free_bound}), while a third domain is approaching it from the interior without touching $\partial D$, like on the middle picture of Figure \ref{fig:2}. We exclude this cuspidal behavior (at regular points) via clean-up results; this is contained in Theorem \ref{thm:up_to_the_bound} below, which concerns the behavior of the internal free boundary $\mathcal{F}(u)$ as it approaches the fixed boundary $\partial D$. 

\begin{theorem}\label{thm:up_to_the_bound}
	Let $u=(u_1,\dots,u_N)\in H^1_{0,N}(D)$ be a minimizer of \eqref{eq:opt_part_variat}, $\mathcal{F}(u)$ be as in \eqref{eq:free_boundary} and $\mathcal{F}_{\partial D}(u)$ be as in \eqref{eq:free_bound_fixed}. Moreover, we let $\mathcal{R}(u)$ and $\mathcal{R}_{\partial D}(u)$ be their regular parts, as in \Cref{thm:int_free_bound} and \Cref{def:reg_sing}, respectively.  Then, every regular point $x_0\in \mathcal{R}_{\partial D}(u)$ sees exactly two components of the optimal partition, that is: for any $x_0\in\mathcal{R}_{\partial D}(u)$ there exists $j,k\in\{1,\dots,N\}$, with $j\neq k$, and $r_0>0$ such that
	\begin{align*}
		&\Omega_j\cap B_r(x_0)\neq \emptyset,\quad\Omega_k\cap B_r(x_0)\neq \emptyset, \\
		&\Omega_i\cap B_r(x_0)=\emptyset~\text{for all }i\neq j
	\end{align*}
	for all $r\leq r_0$.    
	Moreover, $\overline{\mathcal{R}(u)}\cap B_{r_0}(x_0)$ is of class $C^1$ up to $\partial D$ and 
	\begin{equation*}
		\mathcal{R}_{\partial D}(u)\cap B_{r_0}(x_0)=\overline{\mathcal{R}(u)}\cap \partial D\cap B_{r_0}(x_0).
	\end{equation*}
	Furthermore, in this case, $\mathcal{R}(u)$ approaches $\partial D$ in an orthogonal way, in the sense that, if $\bm{e}_x\in\partial B_1$ denotes a unit normal vector for $\mathcal{R}(u)\cap B_{r_0}(x_0)$ at the point $x\in \mathcal{R}(u)\cap B_{r_0}(x_0)$ and $\nnu(x_0)$ denotes the unit outer normal to $\partial D$ at $x_0$, then
	\[
	\lim_{\substack{x\to x_0 \\ x \in D}}\bm{e}_x\cdot\nnu(x_0)=0.
	\]
	
\end{theorem}

Finally, as a consequence of our analysis, we obtain a complete description of the free interfaces in dimension two.

\begin{theorem}\label{thm:up_to_the_bound_2D}
	Let $d=2$ and let $u=(u_1,\dots,u_N)\in H^1_{0,N}(D)$ be a minimizer of \eqref{eq:opt_part_variat}, $\mathcal{F}(u)$ be as in \eqref{eq:free_boundary} and $\mathcal{F}_{\partial D}(u)$ be as in \eqref{eq:free_bound_fixed}. Then, the set $\mathcal{F}_{\partial D}(u)$ is finite and the free interface $\mathcal{F}(u)\cup\mathcal{F}_{\partial D}(u)$ is composed of finite number of $C^1$ arcs meeting at multiple points, at which they form equal angles (as on Figure \ref{fig:1}).
\end{theorem}

\begin{remark}
	A description of the interior interface $\mathcal F_D(u)\subset D$ was obtained in \cite{CTV2003}. We complete the proof of Theorem \ref{thm:up_to_the_bound_2D} in Section \ref{sub:proof-of-theorem-2D}, by showing that the points of $\mathcal F_{\partial D}(u)$ are isolated and by proving a regularity result for $\mathcal F_D(u)$ around these points. 
\end{remark}

\subsection*{Further remarks and possible applications} The spectral optimal partition problem \eqref{eq:opt_part_variat} 
is just an instance of a large family of problems whose main feature is the segregation of 
a fixed number of densities; we stress that our techniques are general and can be exploited also in the framework of other optimal partition problems. 
Without aiming at giving a thorough list, here we mention some of the main works on the subject.
In \cite{TerraciniTavares2012} the authors investigate regularity of the free interface for more general segregated systems which do not necessarily satisfy a minimality condition, but rather just some extremality conditions as differential inequalities and reflection laws across the free boundary. Many of the interior regularity results we mentioned have been extended to the case of higher Dirichlet eigenvalues in \cite{RTT}. It is also worth mentioning \cite{CPQ} and \cite{STTZ2018}, where the authors analyzed the case of segregation with positive distance between the components. Finally, we remark that solutions to segregated problems (hence, with emerging free boundaries) arise also as singular limits of systems with strong competition \cite{CTV2002,CTV2003,CTV2005indiana}
(see for instance \cite{CL2008,CTV2005advances,SZ2017}).

\section{Sketch of the proof}
The approach we employ in order to prove our main results \Cref{thm:taylor}, \Cref{thm:fixed_free_bound}, \Cref{thm:up_to_the_bound} and \Cref{thm:up_to_the_bound_2D} relies on the variational structure of the problem and is based on a version of the Almgren's monotonicity formula at boundary points. Let us first introduce it at interior points. We point out that what we are now going to describe is, essentially, a classical fact in the case of (scalar) elliptic equations with sufficiently smooth coefficients. For what concerns systems with segregated densities, like the ones this paper deals with, we refer among others to \cite{CL2007} and \cite{TerraciniTavares2012} for the proofs of the results described below.

For any $x_0\in D$, any $r<\dist(x_0,\partial D)$ and any minimizer $u=(u_1,\dots,u_N)\in H^1_{0,N}(D)$ of \eqref{eq:opt_part_variat}, we define the \emph{energy} function
\begin{equation*}
	E(u,r,x_0):=\frac{1}{r^{d-2}}\sum_{i=1}^N\int_{B_r(x_0)}(\abs{\nabla u_i}^2-\lambda_1(\Omega_i)u_i^2)\dx
\end{equation*}
and the \emph{height} function
\begin{equation*}
	H(u,r,x_0):=\frac{1}{r^{d-1}}\sum_{i=1}^N\int_{\partial B_r(x_0)}u_i^2\ds.
\end{equation*}
Now, whenever $H(u,r,x_0)>0$, we consider the \emph{frequency} function
\begin{equation}\label{eq:frequency_int}
	\mathcal{N}(u,r,x_0):=\frac{E(u,r,x_0)}{H(u,r,x_0)}.
\end{equation}
Resembling what happens in the scalar case, the function $\mathcal{N}(u,r,x_0)$, for small values of $r$, provides information on the local behavior of the minimizer $u$ near the point $x_0$. In particular, we have that $H(u,r,x_0)>0$ for all $r<\dist(x_0,\partial D)$ and there exists $C=C(d,D,N)>0$ such that, for any $x_0\in D$, the function
\[
r\mapsto e^{C r}\mathcal{N}(u,r,x_0)
\]
is nondecreasing, for $r<\dist(x_0,\partial D)$. It is well known that this monotonicity property carries many consequences concerning the local behavior of $u$. For instance, one obtains that the zero set
\[
\mathcal{Z}(u)=\{x\in D\colon u(x)=(0,\dots,0)\}
\]
has empty interior, which, together with continuity of minimizers, yields that
\[
\mathcal{F}(u)=\mathcal{Z}(u),
\]
with $\mathcal{F}(u)$ as in \eqref{eq:free_boundary}, and that the number
\begin{equation*}
	\gamma(u,x_0):=\lim_{r\to 0}\mathcal{N}(u,r,x_0)
\end{equation*}
is well defined for any $x_0\in D$. Moreover, one can prove that
\begin{equation*}
	\text{either}\quad\gamma(u,x_0)=1\quad\text{or}\quad\gamma(u,x_0)\geq 1+\delta_d,
\end{equation*}
for some $\delta_d>0$ depending only on the dimension. One can then define
\begin{equation*}
	\mathcal{R}(u)=\{x\in D\colon \gamma(u,x)=1\}\quad\text{and}\quad\mathcal{S}(u)=\{x\in D\colon \gamma(u,x)>1\},
\end{equation*}
with $\mathcal{R}(u)$ and $\mathcal{S}(u)$ being as in \Cref{thm:int_free_bound}, so the next step is to prove local regularity of $\mathcal{R}(u)$. We here try to explain the process in few words. For $x_0\in\mathcal{R}(u)$, one can perform a blow-up analysis of the sequence
\[
\frac{u(rx+x_0)}{\sqrt{H(u,r,x_0)}}\quad\text{as }r\to 0,
\]
and prove that it converges to a $1$-homogeneous function $P^{x_0,1}$ of the form
\[
P_j^{x_0,1}(x)=a_{x_0,1}(x\cdot\bm{e}_{x_0})^+,\quad P_k^{x_0,1}(x)=a_{x_0,1}(x\cdot\bm{e}_{x_0})^-,\quad P_i^{x_0,1}(x)=0~\text{for }i\neq j,k,
\]
for some $j,k\in\{1,\dots,N\}$, $a_{x_0,1}>0$ and $\bm{e}_{x_0}\in\partial B_1$. The final, crucial, point is then to prove that, since in the blow-up limit only two non-zero components are left, then sufficiently close to $x_0$, the minimizer $u$ possesses only two non-zero components (obviously, with the same indices as the blow-up limit). In the present paper, we might refer to this phenomenon as \enquote{clean-up}, see e.g. \cite[Lemma 5]{CL2007} or \cite[Proposition 5.4]{TerraciniTavares2012}. Once it is known that, near $x_0$ we have
\[
u_j,u_k\not\equiv 0\quad\text{and}\quad u_i\equiv 0~\text{for all }i\neq j,k,
\]
we define $u^*:=u_j-u_k$, so that the free boundary coincides with the nodal set of $u^*$. From the extremality conditions (see \Cref{lemma:extermaliti}) we derive that
\[
-\Delta u^*=f,\quad\text{in a neighborhood of }x_0,
\]
with
$$f:=\lambda_1(\Omega_j)u_j-\lambda_1(\Omega_k)u_k\in C^{0,1}.$$ 
Hence, from classical regularity theory one deduces that $u^*\in C^{2,\alpha}$ and, once established that $\nabla u^*\neq 0$ on the nodal set, from the implicit function theorem we can obtail the desired regularity of the free boundary.

Therefore, it seems natural to follow the same roadmap when looking at boundary regularity of the free interface. Indeed, one can trivially extend the definition of frequency function \eqref{eq:frequency_int} when it is centered at boundary points, just by assuming the minimizer $u$ to be extended by zero outside $D$. However, right in the very beginning of the argument, i.e. when computing $\mathcal{N}'$, one faces non-trivial troubles. In particular, when trying to prove the so called Pohozaev identity (which is known to be a key step in the proof of monotonicity of $\mathcal{N}$), one realizes that the minimizer $u$ lacks of sufficient regularity for the usual proofs to work, being not more than Lipschitz continuous. Let us be more precise on that. A possible proof of the Pohozaev identity (in case of segregated systems) entails, as a first step, performing inner variations of the type $x+t\xi(x)$ with a smooth, compactly supported vector field $\xi\colon D\to\R^d$, that is considering
\begin{equation*}\label{eq:u_t}
	u^t(x):=u(x+t\xi(x))
\end{equation*}
as a competitor for the criticality of $u$ with respect to the functional $J_N$ and computing the first variation as $t\to 0$. Then, one obtains the Pohozaev identity by letting $\xi$ approximate $(x-x_0)\,\chi_{B_r(x_0)}(x)$, if $x_0\in D$, or $(x-x_0)\,\chi_{D\cap B_r(x_0)}(x)$, if $x_0\in \partial D$. Unfortunately, passing to the limit inside the integrals when doing this approximation at boundary points requires estimating terms involving the gradient of $u$ at the boundary of $\partial D$, see e.g. \cite{Tolsa}; in our case this is a difficult task since $u$ is not a solution to a PDE inside $D$. 
Hence, we overcome this issue avoiding the derivation of gradient estimates and following a simple geometric intuition. Indeed, we observe that in order $u^t$ to be an admissible competitor, $\xi$ being compactly supported in $D$ is not a necessary condition, while the only requirement is that $u^t(x)=0$ whenever $x\in \R^d\setminus D$ and this is equivalent to ask that
\begin{equation}\label{eq:int_ext_condition}
	x+t\xi(x)\in\R^d\setminus D\quad\text{for any }x\in \R^d\setminus D
\end{equation}
for $t$ sufficiently small (not depending on $x$). This condition is fulfilled, for instance, when
\begin{equation}\label{eq:int_starshaped_xi}
	\frac{\xi(x)}{|\xi(x)|}\cdot \nnu(x)>0\quad\text{for }x\in\partial D\cap \overline{B_r(x_0)}~\text{and }t> 0~\text{sufficiently small},
\end{equation}
where $x_0\in\partial D$ is the boundary point we are centering at and $\nnu$ is the outer unit normal of $\partial D$. Now, if \eqref{eq:int_starshaped_xi} (hence \eqref{eq:int_ext_condition}) is satisfied, being $u$ a minimizer, we have that
\begin{equation}\label{eq:min_condition}
	\frac{\d}{\d t}J_N(u^t)\restr{t=0^+}\geq 0.
\end{equation}
By explicitly computing this, we obtain an integral inequality depending on $\xi$, which is essentially, a one-sided inner variation formula. We notice that condition \eqref{eq:int_starshaped_xi} can be relaxed to 
\begin{equation*}
	\xi(x)\cdot \nnu(x)\geq 0\quad\text{for }x\in\partial D\cap B_r(x_0)~\text{and }t> 0~\text{sufficiently small},
\end{equation*}
see \Cref{prop:inner_variations}. At this point, since the last step for the Pohozaev inequality is to let $\xi$ approximate $(x-x_0)\,\chi_{B_r(x_0)}(x)$ (now no integrals over $\partial D$ are involved), it is needed that
\begin{equation}\label{eq:int_sharshaped}
	(x-x_0)\cdot\nnu(x)\geq 0\quad\text{for }x\in\partial D\cap B_r(x_0),
\end{equation}
which is a restrictive geometric condition on $\partial D$ (starshapedness of $D$ with respect to $x_0$). Luckily, the issue of avoiding such geometric assumption has already been faced and smartly overcome in the literature. In particular, we adopt a successful idea introduced in the breakthrough \cite{Adolfsson1997}, which has been exploited in several works since then (see e.g. \cite{KZ}). Slightly more in detail, we introduce a diffeomorphism which locally perturbs $\partial D$ and produces the following effects:
\begin{itemize}
	\item[-] the coefficients of the differential operator driving the problem change. In particular, if one starts with the Laplacian, ends up with a second-order elliptic differential operator in divergence form, with variable coefficients;
	\item[-] a suitable geometric condition analogous to \eqref{eq:int_sharshaped} holds true.
\end{itemize}
Now, we can adjust the argument outlined above in order to obtain a Pohozaev-type \emph{inequality} for the perturbed functional, which is still sufficient for the purpose of proving almost-monotonicity of (the analogous of) the frequency function. We point out that, in order to obtain estimates from below for the derivative of the frequency in terms of integrable terms, the lightest assumption on $\partial D$ is to be $1$-Dini. We refer to \Cref{sec:equivalent} and for the details, see also \cite[Section 2]{Adolfsson1997} and \cite[Section 4]{KZ}. Summing up, in a nutshell, for any point $x_0\in\partial D$ there exists a $C^1$ diffeomorphism $\Psi_{x_0}\colon \R^d\to\R^d$ (defined in \Cref{sec:equivalent}, see \eqref{def:Psi}) such that, if $u\in H^1_{0,N}(D)$ is a minimizer of \eqref{eq:opt_part_variat}, then the quantity
\[
r\mapsto \mathcal{N}(u,r,x_0)=\frac{E(u,r,x_0)}{H(u,r,x_0)}
\]
is almost monotone near $0$, where $E$ and $H$ are defined as
\[
E(u,r,x_0):=\frac{1}{r^{d-2}}\sum_{i=1}^N\int_{\Psi_{x_0}(B_r)\cap D}(\abs{\nabla u_i}^2-\lambda_1(\Omega_i)u_i^2)\dx
\]
and
\[
H(u,r,x_0)=\frac{1}{r^{d-1}}\sum_{i=1}^N\int_{\partial\Psi_{x_0}(B_r)\cap D}u_i^2\ds.
\]
\begin{remark}
	Finding a way of justifying the validity of a Pohozaev-type inequality is a common issue when dealing with variational problems whose solutions lacks of boundary regularity, and we believe the argument we outlined above in broad terms actually applies to many of them. In fact, only two crucial conditions need to be fulfilled:
	\begin{itemize}
		\item[-] $C^1$ regularity of the boundary;
		\item[-] minimality of the solution.
	\end{itemize}
	On one hand, we observe that the former can be relaxed in case the domain already satisfies the starshapedness condition
	\begin{equation}\label{eq:rmk_star}
		(x-x_0)\cdot\nnu(x_0)\geq 0
	\end{equation}
	in a neighborhood of $x_0$. Indeed, $C^1$ regularity of $\partial D$ is needed for gaining $C^1$ regularity of the diffeomorphism introduced in \cite{Adolfsson1997}, which allows to recover the starshapedness condition. On the other hand,
	the latter condition could be recovered by solutions which are critical points of coercive functionals (hence, local minimizers). Whether critical points of \eqref{eq:opt_part_variat} are local minimizers is an open question. Finally, a more technical observation. In sufficiently regular settings, for which a true Pohozaev identity holds, a terms of the type
	\begin{equation}\label{eq:rmk_star1}
		\int_{\partial D\cap B_r(x_0)}|\partial_{\nnu} u|^2 (x-x_0)\cdot\nnu\ds
	\end{equation}
	appears and the geometric condition \eqref{eq:rmk_star} is required (or gained) in order to get rid of it when estimating $\mathcal{N}'$ from below. On the other hand, in a non-regular framework like the one in the present paper, the term \eqref{eq:rmk_star1} does not explicitly appear in the computations, and the geometric condition \eqref{eq:rmk_star} is somehow hidden in the \enquote{variational structure} of the problem and essentially expresses into \eqref{eq:min_condition}.
\end{remark}

Like in the interior case, the value of the almost-monotone quantity $\mathcal{N}(u,r,x_0)$ for small $r>0$ captures some geometric information on the minimizer $u$ near the boundary point $x_0$ and this suggests us to classify the points of $\partial D$ in terms of
\[
\gamma(u,x_0):=\lim_{r\to 0}\mathcal{N}(u,r,x_0).
\]
In the following, we may refer to $\gamma(u,x_0)$ as the \emph{frequency} of $u$ at the point $x_0$. First of all\footnote{For sake of simplicity, let us assume here that $\nnu(x_0)=-\bm{e}_d$.}, through a blow-up procedure, in view of the almost-monotonicity of the modified frequency function, we are able to show that the normalized sequence (which we might call \emph{Almgren rescaling})
\[
r\mapsto \frac{u(rx+x_0)}{\sqrt{H(u,r,x_0)}}
\]
converges, up to subsequences, to a nontrivial limit profile $(U_1^{x_0},\dots,U_N^{x_0})\in( H^1_{\textup{loc}}(\overline{\R^d_+}))^N$, which is $\gamma(u,x_0)$-homogeneous, satisfies $U_i^{x_0}U_j^{x_0}\equiv 0$ in $\R^d_+$ for all $i\neq j$ and is a local minimizer for the Dirichlet energy, thus implying that
\[
\begin{bvp}
	U_i^{x_0}&\geq 0 &&\text{in }\R^d_+, \\
	U_i^{x_0}&=0 &&\text{on }\partial \R^d_+, \\
	-\Delta U_i^{x_0}&=0,&&\text{in }\{x\in\R^d_+\colon U_i^{x_0}(x)>0\}, 
\end{bvp}
\]	
for all $i=1,\dots,N$. Thanks to these properties, we are now able to prove that one of the following happens
\[
\gamma(u,x_0)=1\quad\text{or}\quad \gamma(u,x_0)=2\quad\text{or}\quad \gamma(u,x_0)\geq 2+\delta_d,
\]
for some $\delta_d>0$ depending only on the dimension, and for any minimizer $u$ and any $x_0\in\partial D$, see \Cref{lemma:frequencies}. Hence, we are naturally led to classify the boundary points in term of their frequency, that is
\begin{equation*}
	\mathcal{Z}_\gamma^{\partial D}(u):=\{x\in\partial D\colon \gamma(u,x_0)=\gamma\}.
\end{equation*}
Moreover, since the set $\mathcal{Z}_1^{\partial D}(u)$ is expected to contain the \enquote{traces} of the positivity sets $\Omega_i$s, we define the regular part of the free boundary as the set of boundary points of minimal frequency (higher than $1$), i.e.
\begin{equation}\label{eq:def_regular}
	\mathcal{R}_{\partial D}(u):=\mathcal{Z}_2^{\partial D}(u)
\end{equation}
and the singular part as its complement
\begin{equation}\label{eq:def_singular}
	\mathcal{S}_{\partial D}(u):=\bigcup_{\gamma\geq 2+\delta_d}\mathcal{Z}_\gamma^{\partial D}(u).
\end{equation}
In view of upper semicontinuity of the function
\begin{align*}
	\partial D&\to \R \\
	x_0&\mapsto \gamma(u,x_0),
\end{align*}
as a first consequence, we obtain that
\begin{align*}
	&\mathcal{Z}_1^{\partial D}(u)\quad\text{is open in }\partial D, \\
	&\mathcal{R}_{\partial D}(u)\quad\text{is open in }\partial D\setminus \mathcal{Z}_1^{\partial D}(u), \\
	&\mathcal{S}_{\partial D}(u)\quad\text{is closed in }\partial D\setminus \mathcal{Z}_1^{\partial D}(u).
\end{align*}
Since, in wide terms, the value $\gamma(u,x_0)$ coincides with the \enquote{vanishing order} of $u$ at the point $x_0$, the next step is to examine the blow-up sequences
\begin{equation}\label{eq:int_blow_up}
	u^{r,x_0}(x):=\frac{u(rx+x_0)}{r^{\gamma(u,x_0)}}.
\end{equation}
As one may expect, it is a hard problem to prove strong convergence at any point; nevertheless, we able to do it at points belonging to $Z_1^{\partial D}(u)$ or $\mathcal{Z}_2^{\partial D}(u)$. In particular, our approach for this part is somehow opposite to the one employed so far in the literature for the problem of optimal partitions (see e.g. \cite{CL2007} or \cite{TerraciniTavares2012}). Indeed, we first prove convergence of blow-up sequences \eqref{eq:int_blow_up} to a homogeneous limit profile, for points in $\partial D\setminus\mathcal{S}_{\partial D}(u)$; more importantly, this convergence comes together with an explicit rate of convergence. As a consequence, we obtain regularity of $\mathcal{R}_{\partial D}(u)$ and clean-up lemmas, which establish a connection between the frequency of a point and the number of non-zero components in a neighborhood.

In order to obtain strong convergence of the blow-up sequence \eqref{eq:int_blow_up} and uniqueness of blow-up limits, we base ourselves on an epiperimetric inequality for the Weiss energy. In fact, the pivotal role of epiperimetric inequalities in the study of free boundary regularity is now well established and this tool has been successfully exploited in numerous situations. The idea was introduced in the pioneering work of Reifenberg \cite{Reif} in the field of minimal surfaces, and then adapted to other variational problems with emerging free boundaries, among which we find the classical obstacle problem (see \cite{weiss-obst,CSVobst}), the thin obstacle problem (see e.g. \cite{FSthin,CSVthin,GPGthin,Geraci-thin}) or the one-phase Bernoulli problem (see e.g. \cite{SVonephase,ESVonephase}). Let us briefly explain the idea in our framework. Let
\[
\widetilde{W}_\gamma(w):=\sum_{i=1}^N\left[\int_{B_1^+}|\nabla w_i|^2\dx-\gamma\int_{S_1^+}w_i^2\ds\right]
\]
be the normalized Weiss energy with homogeneity $\gamma=\gamma(u,x_0)>0$, defined for $w\in H^1_{s,N}(B_1^+)$, where $S_1^+:=\partial B_1\cap \R^d_+$. The epiperimetric inequality states that there exists $\epsilon\in (0,1)$ depending only on $d$ such that for any $\gamma$-homogeneous function $z\in H^1_{s,N}(B_1^+)$ satisfying $z=0$ on $B_1'$, there exists $w\in H^1_{s,N}(B_1^+)$ satisfying $w=z$ on $\partial B_1^+$ and
\begin{equation}\label{eq:int_epi}
	\widetilde{W}_\gamma (w)\leq (1-\epsilon)\widetilde{W}_\gamma(z).
\end{equation}
It is a standard fact that this, together with monotonicity of the Weiss energy (which comes as a consequence of the monotonicity of the Almgren frequency function) and minimality of the solution allows to deduce a Cauchy-type condition on the blow-up sequences, that is
\begin{equation}\label{eq:int_cauchy}
	\int_{S_1^+}|u^{r,x_0}-u^{s,x_0}|^2\ds\leq h(|r-s|), \quad\text{for $r,s>0$ sufficiently small}
\end{equation}
being $h$ an explicit modulus of continuity. We remark that it is only at this point that \Cref{ass:domain} is entirely needed. Indeed, so far, the $2$-Dini condition on $\sigma$, i.e.
\[
\int_0^{2R_{\partial D}}\frac{1}{r}\int_0^r\frac{\sigma(t)}{t}\d t<\infty
\]
would have been sufficient, while, in order to obtain the rate of decay of the Weiss functional and then \eqref{eq:int_cauchy}, we need \eqref{eq:dini_hp_1} and \eqref{eq:dini_hp_2}. Hence, the core consists in the proof of \eqref{eq:int_epi}, which we are able to obtain when $\gamma=1$ or $\gamma=2$. Our proof of the epiperimetric inequality is based upon building an explicit competitor $w$, which quantitatively lowers the Weiss energy with respect to the homogeneous function $z$, allowing then to obtain an explicit value of $\epsilon$ (see \Cref{sec:epi}). We also point out that, in order to apply the epiperimetric inequality, we once more need the starshapedness condition gained thanks to the local change of coordinates earlier introduced.

In order to conclude the proof of our main results, we still miss a few steps, which we now sum up. In essence, the key idea behind these last stages is the following: if one can control the norm of the difference of a blow-up sequence and its limit with an explicit rate, then the non-zero components of the blow-up sequence coincide with the non-zero components of its limit. This is what we mean by \emph{clean-up}. Hence, we proceed as follows:
\begin{itemize}
	\item[1.] clean-up at points of frequency $\gamma=1$ at the boundary $\partial D$, see \Cref{l:clean-up-boundary-1};
	\item[2.] regularity of $\mathcal{Z}_2^{\partial D}(u)=\mathcal{R}_{\partial D}(u)$ and boundary clean-up at points of frequency $\gamma=2$, see \Cref{prop:reg};
	\item[3.] quantitative clean-up at interior points of frequency $\gamma=1$, see \Cref{lemma:clean_up_interior};
	\item[4.] full clean-up at points of frequency $\gamma=2$ at $\partial D$, see \Cref{prop:clean_up_2}.
\end{itemize}
In particular, we use Step 1 to define and characterize the \enquote{traces} $\omega_j\sub\partial D$ of the domains $\Omega_j\sub D$ (see \Cref{prop:topology}).  
The combination of step 1 and the blow-up analysis is the content of \Cref{thm:taylor}. Then, through an essentially standard (once an explicit rate of convergence is available) procedure we are able to prove step 2. In order to prove step 3, we first prove an epiperimetric inequality for interior free boundary points of frequency $\gamma=1$, then derive from it an explicit rate of convergence for blow-up sequences; finally, we obtain a quantitative version of the interior clean-up (see e.g. \cite{CL2007,TerraciniTavares2012}). Now, step 4 is derived by applying the previous steps at any scale. We point out that step 4 
rules out a peculiar phenomenon that may a priori manifest at boundary points. Indeed, in step 4 we exclude that (near a point of frequency $\gamma=2$) there is a regular $(d-2)$-dimensional submanifold of $\partial D$, which locally divides $\partial D$ into two parts, say $\omega_1$ and $\omega_2$, but a third positivity set, say $\Omega_3$ is approaching such submanifold from the interior, without touching $\partial D$, see \Cref{fig:2} in the middle.
Finally, combining step 1, step 2, step 3 and step 4 we conclude the proofs of our main theorems.

\subsection{Plan of the paper}
In \Cref{sec:prelim} we collect some known facts about minimizers of $J_N$ and some preliminary results concerning regularity of $\partial D$. In \Cref{sec:equivalent} we introduce a diffeomorphism which allows us to prove a one-sided inner variation formula. We then exploit this, in \Cref{sec:almgren}, in order to prove an Almgren monotonicity formula. \Cref{sec:epi} is essentially self-sufficient, and contains the proof of the epiperimetric inequalities. In \Cref{sec:blowup} we perform a blow-up analysis at points for which the epiperimetric inequality is available and we conclude with \Cref{sec:regularity}, where we prove regularity of the free boundary and the clean-up results.
\subsection{Notation}
We collect here some notation we adopt throughout the whole paper.
\begin{itemize}
	\item $\R^d_+:=\{x=(x',x_d)\in\R^d\colon x_d>0\}$ is the upper half-space;
	\item for $x\in\R^d$ and $r>0$, we denote by $B_r(x):=\{y\in\R^d\colon |x-y|<r\}$ the ball of center $x$ and radius $r$ and
	\[
	B_r:=B_r(0),\quad B_r':=B_r\cap \partial \R^d_+,\quad B_r^+:=B_r\cap\R^d_+.
	\]
	Moreover, we denote $S_r^+:=\partial B_r\cap \R^d_+$;
	\item for sake of simplicity, we denote $\lambda_i:=\lambda_1(\Omega_i)$ for all $i=1,\dots,N$.
	
\end{itemize}

\section{Preliminaries}\label{sec:prelim}

In the present section, we collect some preliminary results, both concerning (basic) properties of the minimizer $u$, and the local geometry of the boundary of the domain.

\subsection{Properties of minimizers}

We start by recalling some known facts regarding the equations satisfied by the minimizer. We recall the following from \cite[Theorem 1.1]{CTVFucick2005}.	

\begin{lemma}[Extremality conditions]\label{lemma:extermaliti}
	Let $u\in H^1_{0,N}(D)$ be a minimizer of \eqref{eq:opt_part_variat}. Then, the following hold true in a weak sense
	\begin{gather*}
		-\Delta u_i\leq \lambda_i u_i,\quad\text{in }D, \\
		-\Delta\left(u_i-\sum_{ j\neq i} u_j\right)\geq \lambda_iu_i-\sum_{j\neq i}\lambda_j u_j,\quad\text{in }D.
	\end{gather*}
	In particular,
	\[
	-\Delta u_i=\lambda_i u_i,\quad\text{in }\Omega_i
	\]
	in a classical sense.
\end{lemma}

Second, we state here Lipschitz continuity (up to the boundary) of a minimizer. This can be found in \cite[Theorem 2.2]{CTVFucick2005} (actually based on \cite[Remark 8.1 and Theorem 8.2]{CTV2005indiana}) and \cite[Proposition 3 and Remark 2]{CL2007}.
\begin{proposition}[Lipschitz estimates]\label{prop:lipschitz}
	If $u\in H^1_{0,N}(D)$ is a minimizer of \eqref{eq:opt_part_variat},
	\[
	\sup_{i\in\{1,\dots,N\}}\sup_{x,y\in\overline{D}\cap B_r(x_0)}\frac{|u_i(x)-u_i(y)|}{|x-y|}\leq C_L\sum_{i=1}^N\norm{u_i}_{H^1(D\cap B_{2r}(x_0))}^2
	\]
	for some $C_L=C_L(d,D)>0$, for all $x_0\in\overline{D}$ and all $r>0$. In particular, $u_i\in W^{1,\infty}(D)$ and $\partial_{\nnu} u_i\in L^\infty(\partial D)$ for all $i=1,\dots,N$.
\end{proposition}

\subsection{Properties of the blow-up limits}

We now introduce the class of blow-up limits, which consists of segregated homogeneous functions which vanish on $\partial \R^d_+$ and locally minimize the Dirichlet energy.

\begin{definition}\label{def:B_gamma}
	Let $\gamma\geq 0$ and $U\in L^1_{\textup{loc}}(\R^d_+)$. We say that $U\in\mathcal{B}_\gamma$ if
	\begin{enumerate}
		\item $U\restr{B_r^+} \in H^1_{s,N}(B_r^+)$ for all $r>0$.
		\item $U(x',0)=0$ for all $x'\in\partial\R^d_+$.
		\item $U(rx)=r^\gamma U(x)$ for all $r>0$ and all $x\in\R^d_+$.
		\item $U$ is a local minimizer for the Dirichlet energy, in the sense that for all $r>0$
		\[
		\sum_{i=1}^N\int_{B_r^+}\abs{\nabla U_i}^2\dx\leq \sum_{i=1}^N\int_{B_r^+}\abs{\nabla V_i}^2
		\]
		for all $V\in H^1_{s,N}(B_r)$ such that $U=V$ on $\partial B_r^+$.			
	\end{enumerate}
\end{definition}

We now state a crucial result, which amounts to a partial classification of the admissible homogeneities of the blow-up limits. Before going one, we introduce the notation for Dirichlet eigenfunctions on the half-sphere. In the whole paper, we denote by $\{\phi_n\}_{n\geq 1}\sub H^1_0(S_1^+)$ a fixed family of eigenfunctions of the Dirichlet-Laplacian on $S_1^+$. More precisely,
\begin{equation*}
	\begin{bvp}
		-\Delta_{\partial B_1}\phi_n&=n(n+d-2)\phi_n, &&\text{in }S_1^+, \\
		\phi_n&=0, &&\text{on }\partial S_1^+, \\
		\int_{S_1^+}\phi_n\phi_m\ds&=\delta_{nm} &&\text{for all }n,m\geq 1.
	\end{bvp}
\end{equation*}
In particular, we can take
\[
\phi_1(\theta)=\frac{\theta_d^+}{\sqrt{\int_{S_1^+}(\theta_d^+)^2\ds}}\quad\text{and}\quad\phi_{i+1}(\theta)=\frac{\theta_i\theta_d^+}{\sqrt{\int_{S_1^+}(\theta_i\theta_d^+)^2\ds}}~\text{for }i=1,\dots,d-1.
\]
In the following lemma, we state some of the possible frequencies of the blow-up limits and its form in some cases. In particular, there is a gap above frequency $2$, whose explicit value is an open problem.
\begin{lemma}\label{lemma:frequencies}
	Let $U\in\mathcal{B}_\gamma\setminus\{0\}$ for some $\gamma\geq0$. Then, one of the following holds:
	\begin{enumerate}
		\item $\gamma=1$ and $U_i=\alpha x_d^+$, for some $\alpha>0$ and some $i\in\{1,\dots,N\}$, while $U_j=0$ for all $i\neq j$;\label{item:N_1}
		\item $\gamma=2$, $U_i=\alpha(x'\cdot\bm{e})^+\,x_d^+$ and $U_j=\alpha(x'\cdot\bm{e})^-\,x_d^+$, for some $\alpha>0$, $\bm{e}\in \partial S_1^+$ and $i,j\in\{1,\dots,N\}$, with $i\neq j$, while $U_k=0$ for all $k\neq i,j$; \label{item:N_2}
		\item $\gamma\geq 2+\delta_d$, for some $\delta_d>0$, and there are $i\neq j$ such that $U_i\not\equiv 0$ and $U_j\not\equiv 0$.\label{item:N_3}
	\end{enumerate}
\end{lemma}
\begin{proof}
	Since each $U_i$ is $\gamma$-homogeneous and $\Delta U_i=0$ in $\{U_i\neq 0\}$, we have: 
	\[
	\sum_{i=1}^N\int_{ S_1^+}|\nabla_{\partial B_1} U_i|^2\ds=\gamma(\gamma+d-2)\sum_{i=1}^N\int_{S_1^+} |U_i|^2\ds
	\]
	for all $U\in\mathcal{B}_\gamma$ and for all $\gamma\geq 0$.
	Let us now first assume that $U$ has only one nonzero component, i.e. $U=(U_1,0,\dots,0)$. Since $U_1\not\equiv 0$, it satisfies
	\[
	\int_{S_1^+}\abs{\nabla_{\partial B_1}U_1}^2\ds=\gamma(\gamma+d-2)\int_{S_1^+}\abs{U_1}^2\ds
	\]
	and since $U_1(\theta',0)=0$, then $\gamma\geq 1$. If $\gamma=1$, then clearly $U_1$ must be a multiple of the first eigenfunction $\phi_1$ and this proves \ref{item:N_1}. 
	
	Let us now assume that $U$ has two nonzero components, i.e. $U=(U_1,U_2,0,\dots,0)$ and let $\tilde{U}:=U_1-cU_2$, where
	\begin{equation}\label{eq:N_c}
		c:=\frac{\displaystyle \int_{S_1^+}U_1\phi_1\ds}{\displaystyle\int_{S_1^+}U_2\phi_1\ds}.
	\end{equation}
	We have that $\tilde{U}$ is $L^2(S_1^+)$-orthogonal to the first eigenfunction of the upper half-sphere and it satisfies
	\[
	\int_{S_1^+}|\nabla_{\partial B_1}\tilde{U}|^2\ds=\gamma(\gamma+d-2)\int_{S_1^+}|\tilde{U}|^2\ds.
	\]
	Therefore, we have that $\gamma\geq 2$. If $\gamma=2$, then $\tilde{U}$ is necessarily a second eigenfunction, i.e. $\tilde{U}=\alpha(\theta'\cdot\bm{e})\,\phi_1$, for some $\alpha>0$ and $\bm{e}\in\partial S_1^+$. Since $U_1$ and $U_2$ are normalized, then necessarily $c=1$ and this proves \ref{item:N_2}. 
	
	Let us finally assume that $U$ has more than two positive components. In this case, there exist two components, which we can assumed to be $U_1$ and $U_2$ (without loss of generality), such that
	\begin{equation}\label{eq:N_1}
		\mathcal{H}^{d-1}\left(\{\theta\in S_1^+\colon U_1(\theta)>0\}\right)+\mathcal{H}^{d-1}\left(\{\theta\in S_1^+\colon U_2(\theta)>0\}\right)\leq \frac{2}{3}\mathcal{H}^{d-1}\left(S_1^+\right).
	\end{equation}
	We now let
	\[
	\tilde{U}:=\frac{U_1-cU_2}{\norm{U_1-cU_2}_{L^2(S_1^+)}},
	\]
	with $c\neq 0$ as in \eqref{eq:N_c} and we consider its Fourier expansion on the upper half-sphere
	\[
	\tilde{U}=\sum_{n=1}^\infty c_n\phi_n,
	\]
	where $\{\phi_n\}_n$ is a fixed basis of $L^2(S_1^+)$ made by orthonormal eigenfunctions. We first observe that, by definition
	\begin{equation}\label{eq:N_3}
		\int_{S_1^+}|\nabla_{\partial B_1}\tilde{U}|^2\ds=\gamma(\gamma+d-2).
	\end{equation}
	On the other hand, since
	\[
	\int_{S_1^+}\abs{\nabla \phi_2}^2\ds=2d\quad\text{and}\quad\int_{S_1^+}\abs{\nabla \phi_n}^2\ds\geq 3(d+1)~\text{for }n\geq 3
	\]
	and since $c_1=0$ and
	\[
	\sum_{n=1}^\infty c_n^2=\|\tilde{U}\|_{L^2(S_1^+)}^2=1,
	\]
	we have that
	\begin{equation}
		\begin{aligned}
			\int_{S_1^+}|\nabla_{\partial B_1}\tilde{U}|^2\ds &=c_2^2\int_{S_1^+}\abs{\nabla\phi_2}^2\ds+\sum_{n=3}^\infty c_n^2\int_{S_1^+}\abs{\nabla \phi_n}^2\ds \\
			&\geq c_2^2\,2d+\left(\sum_{n=3}^\infty c_n^2\right)\,3(d+1) \\
			&=2d+(1-c_2^2)(d+3).
		\end{aligned}\label{eq:N_2}
	\end{equation}
	Furthermore, in view of \eqref{eq:N_1}
	\begin{equation}\label{eq:N_4}
		\abs{c_2}=\abs{\int_{S_1^+}\phi_2\tilde{U}\ds}\leq \max_{\substack{E\sub S_1^+ \\ \mathcal{H}^{d-1}(E)\leq \frac{2}{3}\mathcal{H}^{d-1}(S_1^+)}}\left(\int_E\phi_2^2\ds\right)^{\frac{1}{2}}=:q_d<1.
	\end{equation}
	Combining this fact with \eqref{eq:N_3} and \eqref{eq:N_2} we obtain that
	\begin{equation*}
		\gamma(d+\gamma-2)\geq 2d+(1-q_d^2)(d+3),
	\end{equation*}
	which, in turn, implies \Cref{item:N_3} with
	\begin{align*}
		\delta_d&=\frac{-(d+2)+\sqrt{(d+2)^2+4(1-q_d^2)(d+3)}}{2}.\qedhere
	\end{align*}
\end{proof}

We conclude with a remark concerning boundedness with universal constants of the eigenvalues of the partition.
\begin{remark}\label{rmk:sup_inf_lambda_i}
	One can easily observe that the eigenvalues $\lambda_i$ can be bounded from above and below by positive quantities depending only on $d$, $D$ and $N$. Indeed, on one hand
	\[
	\lambda_i\leq \sum_{i=1}^N\lambda_i=:\Lambda=\Lambda(d,D,N)\quad\text{for all }i=1,\dots,N
	\]
	and, on the other hand, from the Faber-Krahn inequality, we have that
	\[
	\lambda_i=\lambda_1(\Omega_i)\geq \lambda_1(B_1)|B_1|^{\frac{2}{d}}|\Omega_i|^{-\frac{2}{d}}\geq \lambda_1(B_1)|B_1|^{\frac{2}{d}}|D|^{-\frac{2}{d}}.
	\]
\end{remark}

\subsection{Comments on the geometry of the boundary}\label{sec:comments_bd}

In this section, we gather some results concerning the (local) geometry of the boundary $\partial D$, justifying and commenting some of the assumptions we made in \Cref{ass:domain}. We start by arguing that, when locally describing $\partial D$, considering the same modulus of continuity and the same radius for any point is not restrictive.
\begin{remark}\label{rmk:modulus_cont_1}
	When describing the boundary of the domain $D$, in principle, everything depends on the point $x_0\in\partial D$. However, assuming the radius $R_{\partial D}$ and the modulus of continuity $\sigma$ to be the same at every point does not lead to a loss of generality. In particular, for $x_0\in\partial D$, let $\textbf{Q}_{x_0}\in O(d)$, $R_{x_0}>0$ and $\varphi_{x_0}\colon B_{R_{x_0}}'\to \R$  be such that
	\begin{align*}
		&D\cap B_R(x_0)=\{ \textbf{Q}_{x_0}x+x_0\colon y'\in B_R',~x_d>\varphi_{x_0}(x')  \}\cap B_R(x_0), \\
		&\partial D\cap B_R(x_0)=\{ \textbf{Q}_{x_0}x+x_0\colon x'\in B_R',~x_d=\varphi(x')  \}\cap B_R(x_0),
	\end{align*}
	for all $R\leq R_{x_0}$ and let $\sigma_{x_0}\colon[0,2R_{x_0}]\to [0,\infty)$ be the modulus of continuity of $\nabla\varphi_{x_0}$. Since $\partial D$ is a compact set, there exists $n\in\mathbb{N}$ points $P_1,\dots,P_n\in\partial D$, $n$ radii $R_{P_1},\dots,R_{P_n}>0$ and $n$ functions $\varphi_{P_1},\dots,\varphi_{P_n}$, defined on $B_{R_{P_1}}',\dots,B_{R_{P_n}}'$ respectively, such that, up to rigid movements, $\partial D\cap B_{R_{P_i}}(P_i)$ can be represented as the graph of $\varphi_{P_i}\colon B_{R_{P_i}}'\to \R$ and
	\[
	\partial D\sub \bigcup_{i=1}^n B_{R_{P_i}}(P_i).
	\]
	We call $\sigma_i$ the modulus of continuity of $\nabla \varphi_{P_i}$, for $i=1,\dots,n$. Given this open covering of $\partial  D$, we know that there exists a radius $0<R_{\partial D}\leq \min\{R_{P_1},\dots,R_{P_n}\}$ such that for  all $x\in \partial D$, the set $B_{R_{\partial D}}(x)\cap \partial D$ is completely contained in one ball of the covering. Namely, there exists $i\in\{1,\dots,n\}$ such that $B_{R_{\partial D}}(x)\cap \partial D\sub B_{R_{P_i}}(P_i)$. Now, one can observe that, in $B_{R_{\partial D}}(x)$ the boundary $\partial D$ is described, up to rigid motions, by the function $\varphi_{P_i}$ and the corresponding modulus of continuity is still $\sigma_i$. Hence, we can just take $\sigma:=\max\{\sigma_1,\dots,\sigma_n\}$, which is defined on $[0,2R_{\partial D}]\sub[0,2\min\{R_{P_1},\dots,R_{P_n}\}]$ and non-decreasing. Moreover, $\sigma(0)=0$ and $\sigma_i\leq\sigma$ for all $i=1,\dots,n$, hence (II) in \Cref{ass:domain} still holds. Finally, the fact that
	\begin{equation*}
		\int_0^{2R_{\partial D}}\frac{1}{r\sigma_0(r)}\int_0^r\frac{\sigma(t)}{t}\d t<\infty
	\end{equation*}
	straightforwardly comes from the definition of $\sigma$ and the fact that $\sigma_1,\dots,\sigma_n$ are continuous.
\end{remark}

Now, we here justify the fact that (IV) in \Cref{ass:domain} is not restrictive.
\begin{lemma}\label{lemma:modulus}
	Let $f\colon[0,1]\to [0,+\infty)$ be such that $f\in C([0,1])$, $f(0)=0$ and $f$ non-decreasing. Then there exists $h\colon[0,1]\to [0,+\infty)$ such that $f\leq h$, $h\in C([0,1])\cap C^2(0,1)$, $h(0)=0$, $h'(r)\geq 0$, $(h(r)/r)'\leq 0$ and
	\begin{equation}\label{eq:modulus_th1}
		\abs{h'(r)}\leq \frac{2}{r}h(r)  \quad\text{and}\quad \abs{h''(r)}\leq\frac{4}{r^2}h(r).
	\end{equation}
	Moreover, if
	\begin{equation*}
		\int_0^1\frac{1}{rg(r)}\int_0^r\frac{f(t)}{t}\d t<\infty,
	\end{equation*}
	for some $g\in C([0,1])$, then
	\begin{equation*}
		\int_0^1\frac{1}{rg(r)}\int_0^r\frac{h(t)}{t}\d t<\infty.
	\end{equation*}
\end{lemma}
\begin{proof}
	These facts are essentially already known in the literature. Hence, we only sketch the proof and we refer e.g. to \cite{AN2022} and reference therein for more details. First of all, we consider
	\[
	h_1(r):=r\sup_{t\in[r,1]}\frac{f(t)}{t},
	\]
	which clearly verifies the following:
	\[
	f(r)\leq h_1(r)~\text{for all }r\in[0,1]\quad\text{and}\quad\frac{h_1(r)}{r}~\text{is monotone non-increasing.}
	\]
	Moreover, in \cite[Remark 1.2]{AN2022} it is proved that $h_1$ is monotone non-decreasing and that $h_1(r)/r\in L^1(0,1)$. This also implies that $\lim_{r\to 0}h_1(r)=0$. Analogously, one can easily prove that 
	\begin{equation*}
		\int_0^1\frac{1}{rg(r)}\int_0^r\frac{h_1(t)}{t}\d t<\infty.
	\end{equation*}
	In order to finish the proof, we need to modify $h_1$ in order to make it $C^2$ regular, still maintaining the previous properties. In order to do this, we reason iteratively. Following \cite[Remark 1.3]{AN2022}, let us first consider the smoothed version
	\[
	h_2(r):=2\int_{r/2}^r \frac{h_1(t)}{t}\d t.
	\]
	One can observe that
	\[
	h_2\in C([0,2R])\cap C^1(0,1)\quad\text{and}\quad \lim_{r\to 0}h_2(r)=0
	\]
	and
	\begin{equation}\label{eq:modulus_1}
		h_1(r)\leq h_2(r)\leq 2h_1(r/2).
	\end{equation}
	Furthermore, $h_2$ is non-decreasing and $h_2(r)/r$ is non-increasing. Indeed, on one hand, since $h_1$ is non-decreasing, we have
	\begin{equation*}
		h_2'(r)=\frac{2}{r}\left(h_1(r)-h_1(r/2)\right)\geq 0.
	\end{equation*}
	On the other hand, by simple computations and the fact that $h_1(r)/r$ is non-increasing, we get
	\begin{equation*}
		\left(\frac{h_2(r)}{r}\right)'=\frac{2}{r^2}\left(h_1(r)-h_1(r/2)-\int_{r/2}^r\frac{h_1(t)}{t}\d t\right)\leq \frac{1}{r^2}\left(h_1(r)-2h_1(r/2)\right)\leq 0.
	\end{equation*}
	Finally, from \eqref{eq:modulus_1} and the fact that $h_1$ is non-decreasing, we have that
	\begin{equation*}
		\int_0^1\frac{1}{rg(r)}\int_0^r\frac{h_2(t)}{t}\d t\leq 2\int_0^1\frac{1}{rg(r)}\int_0^r\frac{h_1(t)}{t}\d t<\infty.
	\end{equation*}
	Now, in order to gain $C^2$ regularity, we just recursively define
	\[
	h(r):=2\int_{r/2}^r \frac{h_2(t)}{t}\d t.
	\]
	By the very same argument as in the previous step, we obtain that all the properties are maintained and $h\in C^2(0,1)$. Moreover, since
	\begin{equation*}
		h'(r)=\frac{2}{r}\left( h_2(r)-h_2(r/2)\right)
	\end{equation*}
	and
	\begin{align*}
		h''(r)&=-\frac{2}{r^2}(h_2(r)-h_2(r/2))+\frac{2}{r}\left[ \frac{2}{r}(h_1(r)-h_1(r/2))-\frac{2}{r}(h_1(r/2)-h_1(r/4))\right] \\
		&=\frac{2}{r^2}\left[ -h_2(r)+h_2(r/2)+2h_1(r)-4h_1(r/2)+2h_1(r/4)\right],
	\end{align*}
	by the estimates on $h_1$ and $h_2$ we prove \eqref{eq:modulus_th1} and this completes the proof.
\end{proof}

In the following result, we justify the notion of $\alpha$-Dini function, for real $\alpha\geq 1$, by proving that the $\alpha$-Dini condition coincides with the usual $j$-Dini condition whenever $\alpha$ is integer.
\begin{lemma}\label{lemma:alpha-dini}
	Let $f\colon [0,R]\to [0,+\infty)$ be a continuous function and let $n\geq 1$ and $\epsilon\geq 0$. Let us denote, for any $j=1,\dots,n$ and $r\in[0,R]$
	\[
	\mathscr{D}_f^0(r):=f(r)\quad \text{and}\quad \mathscr{D}_f^j(r):=\int_0^r \frac{\mathscr{D}^{j-1}(r)}{r}\d r.
	\]
	Then 
	\begin{equation}\label{eq:dini_equiv_th1}
		\int_0^R\frac{f(r)|\log r|^{n+\epsilon}}{r}\d r<\infty 
	\end{equation}
	if and only if
	\begin{equation}\label{eq:dini_equiv_th2}
		\int_0^R\frac{\mathscr{D}_f^n(r)|\log r|^\epsilon}{r}\dr <\infty.
	\end{equation}
\end{lemma}
\begin{proof}
	The proof is essentially a direct consequence of the following integration by parts formula
	\begin{align*}
		\int_\rho^R \frac{f(r)|\log r|^{n+\epsilon}}{r}\d r &=|\log R|^{n+\epsilon}\int_0^R\frac{f(r)}{r}\d r-|\log \rho|^{n+\epsilon}\int_0^\rho \frac{f(r)}{r}\d r \\
		&\quad +(n+\epsilon)\int_\rho^R\frac{|\log r|^{n-1+\epsilon}}{r}\int_0^r\frac{f(t)}{t}\d t \d r \\
		&=|\log R|^{n+\epsilon}\mathscr{D}_f^1(R)- |\log \rho|^{n+\epsilon} \mathscr{D}_f^1(\rho) \\
		&\quad+(n+\epsilon)\int_\rho^R\frac{\mathscr{D}_f^1(r)|\log r|^{n-1+\epsilon}}{r}\d r
	\end{align*}
	holding for all $\rho\in (0,R)$, which, once iterated, implies that
	\begin{align*}
		\int_\rho^R \frac{f(r)|\log r|^{n+\epsilon}}{r}\d r&=\sum_{j=0}^{n-1}\left(\Pi_{k=0}^{j-1} (n-k+\epsilon)\right) \left(|\log R|^{n-j+\epsilon}\mathscr{D}_f^{j+1}(R)-|\log \rho|^{n-j+\epsilon}\mathscr{D}_f^{j+1}(\rho)\right) \\
		&\quad+\left(\Pi_{k=0}^{n-1} (n-k+\epsilon)\right)\int_\rho^R\frac{\mathscr{D}_f^n(r)|\log r|^\epsilon}{r}\d r.
	\end{align*}
	On one hand, since
	\begin{align*}
		\int_\rho^R \frac{f(r)|\log r|^{n+\epsilon}}{r}\d r&\leq \sum_{j=0}^{n-1}\left(\Pi_{k=0}^{j-1} (n-k+\epsilon)\right) |\log R|^{n-j+\epsilon}\mathscr{D}_f^{j+1}(R) \\
		&\quad+\left(\Pi_{k=0}^{n-1} (n-k+\epsilon)\right)\int_\rho^R\frac{\mathscr{D}_f^n(r)|\log r|^\epsilon}{r}\d r,
	\end{align*}
	for all $\rho\in (0,R)$, then \eqref{eq:dini_equiv_th2} immediately implies \eqref{eq:dini_equiv_th1}. In order to prove the converse, it is sufficient to prove that
	\begin{equation*}
		\lim_{\rho\to 0}|\log \rho|^{n-j+\epsilon}\mathscr{D}_f^{j+1}(\rho)=0\quad\text{for all }j=0,\dots,n-1
	\end{equation*}
	and, in view of de L'Hôpital's rule, this is equivalent to prove that
	\begin{equation*}
		\lim_{\rho\to 0}|\log \rho|^{n+1+\epsilon}f(\rho)=0.
	\end{equation*}
	Let us assume by contradiction that 
	\begin{equation*}
		|\log r|^{n+1+\epsilon}f(r)\geq C\quad\text{for }r\leq r_0,
	\end{equation*}
	for some $C>0$ and $r_0\in(0,R)$. As a consequence, we have that
	\begin{equation*}
		\int_\rho^R\frac{f(r)|\log r|^{n+\epsilon}}{r}\d r\geq \int_\rho^{r_0}\frac{f(r)|\log r|^{n+\epsilon}}{r}\d r\geq C \int_\rho^{r_0}\frac{1}{r|\log r|}\d r
	\end{equation*}
	but, since the last term diverges as $\rho \to 0$, we obtain a contradiction, thus concluding the proof.

\end{proof}

In the following lemma, we state some properties of the moduli $\sigma$ and $\sigma_0$ and some relations between them.
\begin{lemma}\label{lemma:sigma_nondeg}
	There exists $C_\sigma=C_\sigma(d,D)>0$ such that
	\begin{equation}\label{eq:sigma_nondeg_th1}
		\sigma(r)\geq C_\sigma r\quad\text{for all }r\leq 2R_{\partial D}.
	\end{equation}
	and
	\begin{equation}\label{eq:sigma_nondeg_th2}
		\frac{\sigma(r)}{r}\leq \frac{1}{r}\int_0^r\frac{\sigma(t)}{t}\d t,\quad\text{for all }r\leq 2R_{\partial D}.
	\end{equation}
	In addition,
	\begin{equation}\label{eq:sigma_nondeg_th3}
		\sigma(r)\leq C_{\sigma_0}\int_0^r\frac{\sigma_0(t)}{t}\d t,\quad\text{for all }r\leq 2R_{\partial D},
	\end{equation}
	where
	\begin{equation*}
		C_{\sigma_0}:=\frac{1}{4}\int_0^{2R_{\partial D}}\frac{\sigma(t)}{t\sigma_0(t)}\d t.
	\end{equation*}
	Finally,
	\begin{equation}\label{eq:sigma_0_power}
		\int_0^r\frac{\sigma_0(t)}{t}\d t\geq \tilde{C}_{\sigma_0}r^{m_d},\quad\text{for all }r\leq 2R_{\partial D},
	\end{equation}
	where $m_d>0$ is as in \Cref{ass:domain} and 
	\[
	\tilde{C}_{\sigma_0}:=\frac{\sigma_0(2R_{\partial D})}{m_d(2R_{\partial D})^{m_d}}.
	\]
\end{lemma}
\begin{proof}
	We have that \eqref{eq:sigma_nondeg_th1} follows from the fact that $\sigma$ is the modulus of continuity of the gradient of the graph of $\partial D$, at every point (up to rigid movements) and from the compactness of $D$. Indeed, if $\sigma(r)=o(r)$ as $r\to 0$, then $D$ must be the half-space. On the other hand, \eqref{eq:sigma_nondeg_th2} follows from the fact that $(\sigma(r)/r)'\leq 0$. Now, in order to prove \eqref{eq:sigma_nondeg_th3}, we first observe that, by monotonicity of $\sigma(r)/r$ there holds
	\begin{equation*}
		\int_0^r\sqrt{\frac{\sigma(t)}{t^2}}\d t\geq \sqrt{\frac{\sigma(r)}{r}}\int_0^r\frac{1}{\sqrt{t}}\d t\geq 2\sqrt{\sigma(r)}.
	\end{equation*}
	Combining this with Cauchy-Schwarz inequality, we obtain that
	\begin{equation*}
		\sigma(r)\leq \frac{1}{4}\left(\int_0^r\sqrt{\frac{\sigma(t)}{t^2}}\d t\right)^2\leq \frac{1}{4}\int_0^r\frac{\sigma(t)}{t\sigma_0(t)}\d t\int_0^r\frac{\sigma_0(t)}{t}\d t,
	\end{equation*}
	and this concludes the proof of \eqref{eq:sigma_nondeg_th3}. Finally, \eqref{eq:sigma_0_power} is a trivial consequence of the fact that, by assumption $(r^{-m_d}\sigma_0(r))'\leq 0$ and this finishes the proof.
\end{proof}

Finally, we here introduce an auxiliary modulus of continuity $\alpha$ and prove some properties, which will be useful in the proof of the Almgren monotonicity formula.
\begin{lemma}\label{lemma:alpha}
	Let
	\[
	\alpha(r):=3(r\sigma(r))'.
	\]
	Then, $\alpha\in C([0,2R))\cap C^1(0,2R)$,
	\begin{equation}\label{eq:alpha_th1}
		3\sigma(r)\leq \alpha(r)\leq 6\sigma(r)\quad\text{for all }r\in[0,2R_{\partial D}],
	\end{equation}
	and
	\begin{equation}\label{eq:alpha_th2}
		\abs{\alpha'(r)}\leq 24\frac{\sigma(r)}{r}\quad\text{for all }r\in(0,2R_{\partial D}). 
	\end{equation}
\end{lemma}
\begin{proof}
	\eqref{eq:alpha_th1} directly follows by observing that
	\[
	\alpha(r)=3(\sigma(r)+r\sigma'(r))=3(2\sigma(r)+r^2(\sigma(r)/r)')
	\]
	and combining it with the fact that $\sigma'(r)\geq 0$ and $(\sigma(r)/r)'\leq 0$, while \eqref{eq:alpha_th2} follows by direct computations and the fact that $\abs{\sigma'(r)}\leq 2\sigma(r)/r$ and $\abs{\sigma''(r)}\leq 4\sigma(r)/r^2$, see \Cref{ass:domain}.
\end{proof}

\subsection{Diffeomorphisms and first variations}\label{sec:diffeo_variat}

Since a crucial tool for our arguments is the local diffeomorphism introduced in \cite{Adolfsson1997} in order to prove a boundary Almgren monotonicity formula, in this section we examine how a diffeomorphism affects the functional and the equations satisfied by the minimizer. 

For any $C^1$-diffeomorphism $\Phi\colon\R^d\to\R^d$ and any $x\in \R^d$, we denote
\begin{equation*}
	p_\Phi(x):=\abs{\det D\Phi(x)}\ ,\qquad A_\Phi(x):=\abs{\det D\Phi(x)}(D\Phi(x))^{-T}(D\Phi(x))^{-1}.
\end{equation*}
Moreover, for any open, bounded $\mathcal{O}\sub\R^d$ and any $w\in H^1_{0,N}(\mathcal{O})$, we let
\begin{equation*}
	J_{N,\Phi}(w,\mathcal{O}):=\sum_{i=1}^N\frac{\displaystyle \int_{\mathcal{O}}A_\Phi\nabla w_i\cdot\nabla w_i\dx}{\displaystyle \int_{\mathcal{O}}p_\Phi w_i^2\dx}.
\end{equation*}
Now, we let $D_\Phi:=\Phi^{-1}(D)$ and, for any $w\in H^1_{0,N}(D)$, we let
\begin{equation*}
	w^{\Phi}:=w\circ \Phi\in H^1_{0,N}(D_\Phi).
\end{equation*}
We collect below some properties which can be easily seen to hold true.
\begin{lemma}\label{lemma:generic_diffeo}
	The following holds:
	\begin{enumerate}
		\item \label{it:generic_diffeo_th0} the mapping
		\begin{align*}
			H^1_{0,N}(D)&\to H^1_{0,N}(D_\Phi), \\
			w&\mapsto w^\Phi
		\end{align*}
		is a bijection.
		\item \label{it:generic_diffeo_th1} $J_N(w,D)=J_{N,\Phi}(w^\Phi,D_\Phi)$ for any $w\in H^1_{0,N}(D)$.
		\item \label{it:generic_diffeo_th2} $u^\Phi\in H^1_{0,N}(D_\Phi)$ is a minimizer for $J_{N,\Phi}(\cdot,D_\Phi)\colon H^1_{0,N}(D_\Phi)\to\R$.
	\end{enumerate}
\end{lemma}

It is now time to understand how the equation satisfied by a minimizer $u\in H^1_{0,N}(D)$ changes under the diffeomorphism. By classical suitable outer variations applied to the functional, one can easily prove the following weak formulation and integration by parts formula, still holding even in presence of segregated minimizers.

\begin{lemma}[Outer variations and integration by parts formula]\label{lemma:int_by_parts}
	Let $u\in H^1_{0,N}(D)$ be a minimizer of $J(\cdot,D)$ and let $\Phi\colon\R^d\to\R^d$ be a $C^1$-diffeomorphism.  Then, $u^\Phi\in H^1_{0,N}(D_\Phi)$ satisfies
	\begin{equation}\label{eq:weak_formulation}
		\int_{\Phi^{-1}(\Omega_i)}(A_\Phi\nabla u_i^\Phi\cdot\nabla\varphi_i-p_i^\Phi u_i^\Phi\varphi_i)\dx=0
	\end{equation}
	for all $\varphi_i\in H^1_0(\Phi^{-1}(\Omega_i))$. Moreover, for all $x_0\in \overline{D}_\Phi$ and all $r>0$, there holds
	\begin{equation}\label{eq:int_by_parts_1}
		\int_{D_\Phi\cap B_r(x_0)}(A_\Phi\nabla u_i^\Phi\cdot\nabla \varphi_i-p_i^\Phi u_i^\Phi\varphi_i)\dx=\int_{D_\Phi\cap \partial B_r(x_0)}\varphi_i A_\Phi\nabla u_i^\Phi\cdot\nnu\ds,
	\end{equation}
	for all $\varphi_i\in H^1_0(\Phi^{-1}(\Omega_i))$ and all $i=1,\dots,N$.
	In particular,
	\begin{equation}\label{eq:int_by_parts_2}
		\int_{D_\Phi\cap B_r(x_0)}(A_\Phi\nabla u_i^\Phi\cdot\nabla u_i^\Phi-p_i^\Phi |u_i^\Phi|^2)\dx=\int_{D_\Phi\cap \partial B_r(x_0)}u_i^\Phi A_\Phi\nabla u_i^\Phi\cdot\nnu\ds,
	\end{equation}
	for all $i=1,\dots,N$.
\end{lemma}
\begin{proof}
	The proof of \eqref{eq:weak_formulation} easily follows by considering variations of the type $(u^\Phi)^t_i:=u_i^\Phi+t\varphi_i$, with $\varphi_i\in C_c^\infty(\Phi^{-1}(\Omega_i))$ and $t\in(-1,1)$ and then conclude by density. In order to prove \eqref{eq:int_by_parts_1}, let $\rho_\epsilon=\rho_{\epsilon,r}\in C_c^\infty(D_\Phi)$ be a smooth approximation of the characteristic function $\chi_{D_\Phi\cap B_r(x_0)}$, i.e. $\rho_\epsilon\to \chi_{D_\Phi\cap B_r(x_0)}$ uniformly as $\epsilon\to 0$ and let $\varphi_i\in C_c^\infty(\Phi^{-1}(\Omega_i))$. Then, the proof easily follows by applying \eqref{eq:weak_formulation} with $\varphi_i\rho_\epsilon$ as a test function and letting $\epsilon\to 0$. Finally, by density of $C_c^\infty(\Phi^{-1}(\Omega_i))$ into $H^1_0(\Phi^{-1}(\Omega_i))$ we conclude the proof of \eqref{eq:int_by_parts_1}, while \eqref{eq:int_by_parts_2} follows by choosing $\varphi_i=u_i$ in \eqref{eq:int_by_parts_1}. The proof is thereby complete.
\end{proof}

We conclude the section with an inner variation formula for the perturbed functional $J_{N,\Phi}$, which is one of the main novelty of the present paper. Indeed, we consider inner variations of the type
\[
D_\Phi\ni x\mapsto x+t\xi(x)
\]
but without requiring the vector field $\xi\colon \R^d\to\R^d$ to be compactly supported in $D_\Phi$. Nevertheless, we require a geometric condition regarding how the vector field $\xi$ crosses the boundary $\partial D_\Phi$, that is
\[
\xi\cdot\nnu\geq 0 \quad\text{on }\partial D_\Phi.
\]
This essentially translates into a \enquote{variational} condition, since it is equivalent to require that
\[
x+t\xi(x)\in\R^d\setminus D_\Phi\quad\text{whenever}\quad x\in\R^d\setminus D_\Phi~\text{and }t\geq 0
\]
which is, in turn, equivalent to ask that $(u^\Phi)^t(x):=u^\Phi(x+t\xi(x))\in H^1_{0,N}(D_\Phi)$, i.e. $(u^\Phi)^t(x)$ to be an admissible competitor. Since we can only allow one-sided inner variations, that is with $t\geq 0$, from minimality of $u^\Phi$ we can only obtain the inequality
\[
\lim_{t\to 0^+}\frac{\d}{\d t} J_{N,\Phi}(v^t,D_{\Phi})\geq 0,
\]
which is still sufficient for our purposes. This result will actually lead to a Pohozaev-type inequality, see \Cref{prop:pohozaev}.	

\begin{proposition}[One-sided inner variations]\label{prop:inner_variations}
	There holds
	\begin{equation}\label{eq:innner_th1}
		\sum_{i=1}^N\int_{D_{\Phi}}\left(2D\xi A_\Phi-A_\Phi\dive \xi-\d A_\Phi[\xi]\right)\nabla u_i^\Phi\cdot\nabla u_i^\Phi\dx+\int_{D_{\Phi}}(\nabla p_i^\Phi\cdot \xi+p_i^\Phi\dive \xi)|u_i^\Phi|^2\dx\geq 0,
	\end{equation}
	for all $\xi\in C^1_c(\R^d,\R^d)$ such that
	\begin{equation}\label{eq:inner_hp1}
		\xi(x)\cdot\nnu(x)\geq 0\quad\text{for all }x\in\partial D_\Phi.
	\end{equation}
\end{proposition}
\begin{proof}
	In order to prove \eqref{eq:innner_th1}, we first fix a vector field $\xi\in C^1_c(\R^d,\R^d)$ such that \eqref{eq:inner_hp1} holds, a nonnegative $t\geq 0$ and we consider the map
	\[
	G_t(x):=x-t\xi (x).
	\]
	One can easily see that $G_t$ is a $C^1$-diffeomorphism in $\R^d$, for $t$ sufficiently small. Hence, we have that
	\[
	G_t^{-1}(x)=x+t\xi(x)+o(t),\quad\text{as }t\to 0^+,
	\]
	where the reminder $o(t)$ is uniform for $x\in\R^d$. We now let $v^t(x):=u^{\Phi}(G_t^{-1}(x))$, for $x\in\R^d$ (we extend $v$ to be zero in $\R^d\setminus D_\Phi$). Let us first assume that
	\begin{equation}\label{eq:pohoz_6}
		\frac{\xi(x)}{|\xi(x)|}\cdot\nnu(x)\geq C> 0\quad\text{for all }x\in\partial D_\Phi\cap\{\xi\neq 0\},
	\end{equation}
	for some $C>0$.	We have that $v_i^t\not\equiv 0$ for any $i=1,\dots,N$ and, in view of \eqref{eq:pohoz_6}, $v^t\in H^1_{0,N}(D_\Phi)$ for $t$ sufficiently small (depending on $C$). Therefore, $v^t$ is an admissible competitor for $J_{N,\Phi}(\cdot,D_{\Phi})$ and, being $u^{\Phi}=v^0$ a minimizer, we have that
	\begin{equation*}
		\lim_{t\to 0^+}\frac{\d}{\d t} J_{N,\Phi}(v^t,D_{\Phi})\geq 0.
	\end{equation*}
	Then, performing standard computations for first inner variations, we obtain \eqref{eq:innner_th1} in case \eqref{eq:pohoz_6} holds. At this point, we want to prove \eqref{eq:innner_th1} for $\xi\in C_c^1(\R^d,\R^d)$ only satisfying \eqref{eq:inner_hp1}. The idea is to approximate $\xi$ with a sequence $\{\xi_s\}_{s>0}\sub C_c^\infty(\R^d,\R^d)$ satisfying \eqref{eq:pohoz_6} and then pass to the limit as $s\to 0$ in \eqref{eq:innner_th1}.
	In order to to this, we first let $\eta\in C_c^\infty(\R^d)$ be such that $\eta=1$ in a neighborhood of $\supp\xi\cap\partial D_\Phi$. Then, being $D_\Phi$ of class $C^1$, one can easily see that there exists a finite number of points $p_1,\dots,p_n\in\partial D_\Phi\cap \supp\xi$ and a radius $r>0$ such that
	\[
	\supp\xi\cap \partial D_\Phi\sub\bigcup_{i=1}^n B_r(p_i)
	\]
	and such that, in each $\partial D_\Phi\cap B_r(p_i)$ is, up to a change of coordinates, the graph of a $C^1$ function. Moreover, it is possible to choose $n$ smooth cut-off functions $\eta_i\in C_c^\infty(B_r(p_i))$ such that 
	\begin{equation*}
		\sum_{i=1}^n\eta_i(x)\nnu(x)\cdot\nnu(p_i)>0\quad\text{for any }x\in\partial D_\Phi\cap \supp\xi.
	\end{equation*}
	Now, for any $s>0$, we let 
	\begin{equation*}
		\xi_s:=\xi+s\eta\sum_{i=1}^n \eta_i\nnu(p_i)
	\end{equation*}
	so that $\xi_s$ satisfies \eqref{eq:pohoz_6}; hence,  \eqref{eq:innner_th1} holds for $\xi_s$, for $s$ sufficiently small and we can then pass to the limit as $s\to 0$, thanks to dominated convergence theorem. This concludes the proof.
\end{proof}

\section{An equivalent problem on a starshaped domain}\label{sec:equivalent}

In order to derive an Almgren-type monotonicity formula at boundary points and still avoid (a posteriori) unnecessary geometric assumptions, such as convexity, we need to introduce a diffeomorphism which locally modifies the problem and ensures the right sign properties for the derivative of the corresponding frequency function. This idea was introduced in \cite{Adolfsson1997}. Hence, the aim of the present section is to introduce such diffeomorpshism, to state the main properties of the transformed problem and to derive a Pohozaev-type inequality, which is a key ingredient in the proof of the Almgren monotonicity formula.

In this section, we fix a point $x_0\in\partial D$ and we let $\textbf{Q}\in O(d)$, $R_{\partial D}>0$, $\varphi\colon B_{R_{\partial D}}'\to \R$ and $\sigma\colon[0,2R_{\partial D}]\to [0,\infty)$ be as in \Cref{ass:domain}. As a first step, we consider the nonlinear map
\begin{align*}
	\Psi_0\colon B_{R_{\partial D}}&\to \R^d \\
	x &\mapsto y=\Psi_0(x)
\end{align*}
defined as
\begin{equation*}
	\Psi_0(x):=(x',x_d+3\abs{x}\sigma(\abs{x})).
\end{equation*}
We here provide the expression and some basic properties of $D\Psi_0$.
\begin{lemma}\label{lemma:DPsi_0}
	For all $x\in B_{R_{\partial D}}\setminus\{0\}$ there holds
	\begin{equation*}
		D\Psi_0(x)=		\begin{pmatrix}
			1 &	0 &	\cdots & 0 & \frac{\alpha(\abs{x})}{\abs{x}} x_1 \\
			0 &	1 &	\cdots & 0	& \frac{\alpha(\abs{x})}{\abs{x}} x_2 \\
			\vdots & \vdots & \ddots & \vdots & \vdots \\
			0 & 0 & \cdots & 1	& \frac{\alpha(\abs{x})}{\abs{x}} x_{d-1} \\
			0 &
			0 &
			\cdots & 0 & 1+\frac{\alpha(\abs{x})}{\abs{x}} x_d
		\end{pmatrix}=\mathbf{I}+
		\begin{pmatrix}
			\bm{0}_{d\times(d-1)}& \frac{\alpha(\abs{x})}{\abs{x}}x
		\end{pmatrix}.
	\end{equation*}
	and
	\begin{equation*}
		\det D\Psi_0(x)=1+\frac{\alpha(\abs{x})}{\abs{x}} x_d.
	\end{equation*}
	Hence, by continuity, we can extend $D\Psi_0(0)=\mathbf{I}$. In particular, up to taking a smaller $R_{\partial D}$, we can assume that $\sigma(R_{\partial D})\leq 1/120$ and, hence, that $|\det D\Psi_0(x)-1|\leq 1/20$ for all $x\in B_{R_{\partial D}}$. In particular, we can assume that $\Psi_0\colon B_{R_{\partial D}}\to \Psi_0(B_{R_{\partial D}})$ is a $C^1$-diffeomorphism.
\end{lemma}
Even if we are aimed at studying only \emph{local} regularity properties, it is still convenient for our purposes to work with a diffeomorphism which is defined in the whole $\R^d$, which should coincide with $\Psi_0$ in a neighborhood of the origin. To this end, we introduce the following. Let $\eta\in C_c^\infty([0,+\infty))$ be such that
\begin{equation*}
	\eta(t)=\begin{cases}
		1, &\text{for }t\leq R_{\partial D}/2, \\
		0, &\text{for }t\geq R_{\partial D}
	\end{cases}
\end{equation*}
and $\abs{\eta'(t)}\leq 4/R_{\partial D}$ for all $t\geq 0$. We define $\Psi\colon \R^d\to \R^d$ as
\[
\Psi(x):=(x',x_d+3\eta(\abs{x})\abs{x}\sigma(\abs{x})).
\]
In particular
\[
\Psi(x)=\begin{cases}
	\Psi_0(x), &\text{for }x\in B_{R_{\partial D}/2}, \\
	x, &\text{for }x\in \R^d\setminus B_{R_{\partial D}}.
\end{cases}
\]
Moreover, we have the following.
\begin{lemma}\label{lemma:Psi}
	$\Psi$ is a $C^1$-diffeomorsphism in the whole $\R^d$.
\end{lemma}
\begin{proof}
	One easily observes that
	\[
	\det D\Psi(x)=1+\partial_{x_d}(3\eta(\abs{x})\abs{x}\sigma(\abs{x}))=1+3\eta'(\abs{x})x_d\sigma(\abs{x})+3\eta(\abs{x})\alpha(\abs{x})\frac{x_d}{\abs{x}}
	\]
	for all $x\in\R^d$. Hence, for $x\in B_{R_{\partial D}}$
	\[
	\abs{\det D\Psi-1}\leq 3(\eta'(\abs{x})\sigma(\abs{x})+\alpha(\abs{x}))\leq 60\sigma(|x|)\leq \frac{1}{2}.
	\]
	Hence the proof is concluded in view of Hadamard's Global Inversion Theorem.
\end{proof}
At this point, since for notational purposes it is always preferable to work centered at the origin, for all $x\in\R^d$, we introduce
\begin{equation}\label{def:Psi}
	\Psi_{x_0}(x):=\textbf{Q}_{x_0}\Psi(x)+x_0,
\end{equation}
which still is a $C^1$-diffeomorphism in $\R^d$ and satisfies $D\Psi_{x_0}=\textbf{Q}_{x_0}D\Psi$.	We then consider the \enquote{transformed} domain
\begin{equation}\label{eq:Omega}
	\mathcal{O}=\mathcal{O}^{x_0}:=\Psi^{-1}_{x_0}(\mathbf{Q}^{-1}_{x_0}(D-x_0)),
\end{equation}
and we let
\begin{equation*}
	\mathcal{O}_*=\mathcal{O}_*^{x_0}:=\Psi^{-1}_{x_0}(\textbf{Q}^{-1}_{x_0}(D\cap B_{R_{\partial D}}(x_0)-x_0)\cap\Psi(B_{R_{\partial D}/2}))\sub B_{R_{\partial D}/2}
\end{equation*}
while, for all $r\in (0,R_{\partial D}/2]$, we denote
\begin{equation*}
	\mathcal{O}_r=\mathcal{O}_r^{x_0}:=\mathcal{O}_*\cap B_r,\quad S_r=S_r^{x_0}:=\mathcal{O}_*\cap \partial B_r,\quad \Gamma_r=\Gamma_r^{x_0}:=\partial\mathcal{O}_*\cap B_r.
\end{equation*}	
The description of $\mathcal{O}_r^{x_0}$ can be easily given in terms of $\varphi_{x_0}$ and $\sigma$. More precisely, the following holds.
\begin{lemma}\label{lemma:Omega_Psi}
	Up to taking a smaller $R_{\partial D}$, we have that
	\begin{align*}
		\mathcal{O}_r^{x_0}&=\{ (x',x_d)\colon x'\in B_r',~\text{and } x_d>\varphi_{x_0}(x')-3\abs{x}\sigma(\abs{x})\}\cap B_r, \\
		\Gamma_r^{x_0}&=\{ (x',x_d)\colon x'\in B_r',~\text{and } x_d=\varphi_{x_0}(x')-3\abs{x}\sigma(\abs{x})\}\cap B_r,
	\end{align*}
	for all $r\leq R_{\partial D}/2$.
\end{lemma}

Summing up, for fixed $x_0\in\partial D$, we may associate to any $w\colon D\to \R$ the transformed function $w^{x_0}\colon \mathcal{O}^{x_0}\to \R$ defined as
\[
w^{x_0}(x):=w(\Psi_{x_0}(x)).
\]	
In particular, we hereafter denote
\begin{equation}\label{eq:def_v}
	v(x):=u^{x_0}(x)=u(\Psi_{x_0}(x)),\quad\text{for }x\in \mathcal{O}^{x_0},
	\	\end{equation}
with $u\in H^1_{0,N}(D)$ being a minimizer of \eqref{eq:opt_part_variat} normalized in $L^2$. In order to study obtain an almost-monotone Almgren-type frequency function centered at a point $x_0\in\partial D$, it is convenient to employ the following change of variables
\begin{equation*}
	y=\Psi_{x_0}(x)=\textbf{Q}_{x_0}(\Psi(x))+x_0
\end{equation*}
and to pass to the study of $v=u^{x_0}$ in a neighborhood of the origin. In the rest of the paper, we might drop the dependence on $x_0$ in the notation, whenever $x_0$ is fixed, but we still point out whenever the choice of the particular point comes into play.

The next step is to understand the \enquote{variational} properties of $v=u^{x_0}$ and, in order to do this, we appeal to \Cref{sec:diffeo_variat}. We let, for $x\in\mathcal{O}$,
\begin{equation}\label{eq:def_A_p}
	A(x):=(D\Psi_{x_0}(x))^{-T}(D\Psi_{x_0}(x))^{-1}\abs{\det D\Psi_{x_0}(x)}\quad\text{and}\quad p(x):=\abs{\det D\Psi_{x_0}(x)}.
\end{equation}
We observe that $A$ and $p$ do note depend on $x_0$. Indeed, since by definition $\textbf{Q}_{x_0}\in O(d)$, one immediately sees that
\begin{equation*}
	A(x)=(D\Psi(x))^{-T}(D\Psi(x))^{-1}\abs{\det D\Psi(x)}\quad\text{and}\quad p(x)=\abs{\det D\Psi(x)}.
\end{equation*}
Moreover, by definition
\begin{equation*}
	A\in C(\R^d,\R^{d\times d})\cap C^1(\R^d\setminus\{0\},\R^{d\times d})\quad\text{and}\quad p\in C(\R^d)\cap C^1(\R^d\setminus\{0\}).
\end{equation*}
We also let
\begin{equation}\label{eq:def_p_mu_beta}
	p_i(x):=\lambda_i p(x),\quad \mu(x):=\frac{A(x)x\cdot x}{\abs{x}^2},\quad\bm{\alpha}(x):=\frac{A(x)x}{\abs{x}}\quad\text{and}\quad \bm{\beta}(x):=\frac{A(x)x}{\mu(x)}
\end{equation}
More explicitly, for $x\in \mathcal{O}_*$, we have
\begin{align*}
	A_{ij}(x)&=\delta_{ij}\left(1+\frac{\alpha(\abs{x})}{\abs{x}}x_d\right)\quad\text{for }i,j=1,\dots,d-1, \\
	A_{id}(x)&=A_{di}(x)=-\frac{\alpha(\abs{x})}{\abs{x}}x_i,\quad\text{for }i=1,\dots,d-1, \\
	A_{dd}(x)&=\frac{\displaystyle 1+\frac{\alpha(\abs{x})}{\abs{x}}\abs{x'}^2}{\displaystyle1+\frac{\alpha(\abs{x})}{\abs{x}}x_d}, \\
	p(x)&=1+\frac{\alpha(\abs{x})}{\abs{x}} x_d.
\end{align*}
At this point, one can easily observe (see \Cref{lemma:generic_diffeo}) that $v\in H^1_{0,N}(\mathcal{O})$ is a minimizer of the transformed functional $J_{N,\Psi_{x_0}}(\cdot,\mathcal{O})\colon H^1_{0,N}(\mathcal{O})\to \R$ defined as
\begin{equation*}
	J_{N,\Psi_{x_0}}(w,\mathcal{O}):=\sum_{i=1}^N\frac{\displaystyle \int_\mathcal{O} A\nabla w_i\cdot\nabla w_i\dx}{\displaystyle\int_\mathcal{O} p w_i^2\dx}.
\end{equation*}
Moreover, there holds
\begin{equation*}\label{eq:lambda_i_transformed}
	\lambda_i=\frac{\displaystyle \int_\mathcal{O} A\nabla v_i\cdot\nabla v_i\dx}{\displaystyle\int_\mathcal{O} p v_i^2\dx}.
\end{equation*}
We state here the main properties of the coefficients of the transformed functional. The proof is essentially contained in \cite{Adolfsson1997}.
\begin{lemma}\label{lemma:A} Let $A$ be as in \eqref{eq:def_A_p} and $p_i$, $\mu$, $\bm{\alpha}$ and $\bm{\beta}$ be as in \eqref{eq:def_p_mu_beta}.
	There exists a constant $\kappa>0$, depending on $d$ and $D$, such that the following properties hold true for any $r\leq R_{\partial D}$ (up to reducing $R_{\partial D}$)
	\begin{enumerate}
		\item $\kappa\sigma(r)\leq \frac{1}{2}$. \label{it:diffeo_0}
		\item $A$ is symmetric and uniformly elliptic near the origin, that is
		\[
		\kappa\abs{\ell}^2\leq A(x)\ell\cdot \ell\leq \kappa^{-1}\abs{\ell}^2
		\]
		for all $\ell\in\R^d$ and all $x\in\overline{\mathcal{O}_r}$. \label{it:diffeo2}
		\item we have that $A(0)=I$,
		\[
		A(x)x\cdot\nnu(x)\geq  \abs{x}\sigma(\abs{x}),\quad\text{for all }x\in \Gamma_r
		\]
		and
		\[
		\bm{\beta}(x)\cdot\nnu(x)\geq \frac{\abs{x}\sigma(\abs{x})}{\mu(x)} \geq 0,\quad\text{for all }x\in\Gamma_r.
		\]
		
		\label{it:diffeo3}
		\item there holds
		\begin{align*}
			&\abs{A_{ij}(x)-\delta_{ij}}\leq \kappa \sigma(\abs{x}),\quad\abs{\nabla A_{ij}(x)}\leq \kappa\frac{\sigma(\abs{x})}{\abs{x}}~\text{for all }i,j=1,\dots,d, \\
			&\norm{A(x)-I}_{\mathcal{L}(\R^d)}\leq \kappa\sigma(\abs{x}), \quad \norm{\d A(x)}_{\mathcal{L}(\R^d;\mathcal{L}(\R^d))}\leq \kappa\sigma(\abs{x})
		\end{align*}
		and
		\[
		\abs{p_i(x)-\lambda_i}\leq \lambda_i\kappa\sigma(\abs{x})\quad\text{and}\quad \abs{\nabla p_i(x)}\leq\lambda_i \kappa\frac{\sigma(\abs{x})}{\abs{x}}~\text{for all }i=1,\dots,d,
		\]
		and for all $x\in\mathcal{O}_r$. \label{it:diffeo4}
		\item there holds
		\[
		\abs{\mu(x)-1}\leq \kappa \sigma(\abs{x}),\quad \abs{\nabla \mu(x)}\leq \kappa\frac{\sigma(\abs{x})}{\abs{x}},\quad \abs{\dive \bm{\alpha}(x)-(d-1)\frac{\mu(x)}{\abs{x}}}\leq \kappa\frac{\sigma(\abs{x})}{\abs{x}}
		\]
		and
		\[
		\abs{\bm{\beta}(x)-x}\leq \kappa\abs{x}\sigma(\abs{x}),\quad \abs{D\bm{\beta}(x)-I}\leq \kappa\sigma(\abs{x}),\quad \abs{\dive\bm{\beta}(x)-d}\leq \kappa\sigma(\abs{x})
		\]
		for all $x\in\mathcal{O}_r$. \label{it:diffeo5}
	\end{enumerate}
\end{lemma}
\begin{proof}
	The proof of the crucial point \textit{(3)} is contained in the proof of Theorem 0.4 in \cite{Adolfsson1997} and strongly uses \Cref{lemma:alpha} and the fact that, in view of the mean value theorem, there holds
	\begin{align*}
		&\abs{\varphi_{x_0}(x')}\leq \abs{x'}\sigma(\abs{x'}) \quad\text{for all }x'\in B_{R_{\partial D}}', \\
		&\abs{x'\cdot\nabla \varphi_{x_0} (x')-\varphi_{x_0}(x')}\leq 2\abs{x'}\sigma(\abs{x'}) \quad\text{for all }x'\in B_{R_{\partial D}}',
	\end{align*}
	up to restricting $R_{\partial D}$. The rest of the proof follows by definition, direct computations and \Cref{lemma:alpha}.
\end{proof}

As a simple corollary, we obtain the following.
\begin{corollary}\label{cor:A}
	Let $A$ be as in \eqref{eq:def_A_p} and $p_i$, $\mu$ be as in \eqref{eq:def_p_mu_beta}. Then, there holds
	\begin{equation*}
		\frac{1}{2}\leq \mu(x)\leq \frac{3}{2},\quad \norm{A(x)}_{\mathcal{L}(\R^d)}\leq \frac{3}{2},\quad \frac{\lambda_1(B_1)|B_1|^{\frac{2}{d}}}{2|D|^{\frac{2}{d}}}\leq p_i(x)\leq \frac{3\Lambda}{2}
	\end{equation*}
	as well as
	\begin{equation*}
		\frac{2}{3}\leq \mu^{-1}(x)\leq 2,\quad \norm{A^{-1}(x)}_{\mathcal{L}(\R^d)}\leq 2,\quad \frac{2}{3\Lambda}\leq p_i^{-1}(x)\leq \frac{2|D|^{\frac{2}{d}}}{\lambda_1(B_1)|B_1|^{\frac{2}{d}}},
	\end{equation*}
	for all $x\in\mathcal{O}_r$, all $r\in(0,R_{\partial D}/2)$ and all $i=1,\dots,N$, where $\Lambda$ is as in \Cref{rmk:sup_inf_lambda_i}.
\end{corollary}
\begin{proof}
	The proof easily follows from \Cref{lemma:A} and \Cref{rmk:sup_inf_lambda_i}.
\end{proof}

Finally, we can now state the main property of (and the reason why we introduced) the transformed problem. More precisely, since the transformed domain $\mathcal{O}$ satisfies the geometric property
\begin{equation*}
	A(x)x\cdot\nnu(x)\geq 0
\end{equation*}
in a neighborhood of the origin, this allows us to prove a Pohozaev-type inequality, which is, in turn, a main ingredient in the proof of the monotonicity of the Almgren functional. This result is strongly based on the crucial observation made in \Cref{prop:inner_variations}.

\begin{proposition}[Pohozaev inequality]\label{prop:pohozaev}
	Let $A$ be as in \eqref{eq:def_A_p} and $p_i$, $\mu$ and $\bm{\beta}$ be as in \eqref{eq:def_p_mu_beta}. Then, for a.e. $r\leq R_{\partial D}/2$, there holds
	\begin{multline}\label{eq:pohoz_th2}
		\sum_{i=1}^N r\int_{S_r}(A\nabla v_i\cdot \nabla v_i-p_i|v_i|^2)\ds\geq \sum_{i=1}^N\Bigg[ 2r\int_{S_r}\frac{(A\nabla v_i\cdot \nnu)^2}{\mu}\ds\\
		+\int_{\mathcal{O}_r}(A\dive\bm{\beta}
		+\d A[\bm{\beta}]-2D\bm{\beta} A)\nabla v_i\cdot\nabla v_i\dx 
		-\int_{\mathcal{O}_r}(\nabla p_i\cdot\bm{\beta}+p_i\dive\bm{\beta})|v_i|^2\dx\Bigg].
	\end{multline}
\end{proposition}
\begin{proof}
	For fixed $r\leq R_{\partial D}/2$ we let $\rho_\epsilon=\rho_{\epsilon,r}\colon\R^d\to\R$ be a smooth approximation of the characteristic function $\chi_{B_r}$, that is $\rho_\epsilon\in C_c^\infty(B_{r+\epsilon})$, $\rho_\epsilon=1$ in $B_r$ and
	\begin{equation}\label{eq:pohoz_1}
		\begin{aligned}
			&\rho_\epsilon\to \chi_{B_r}\quad\text{pointwise, as }\epsilon\to 0,  \\
			&\nabla\rho_\epsilon\overset{*}{\weak}-\nnu\,\mathcal{H}^{d-1}\text{\huge$\llcorner$} \partial B_r\quad\text{as }\epsilon\to 0. 
		\end{aligned}
	\end{equation}
	We now let $\xi_\epsilon:=\rho_\epsilon\bm{\beta}$. Thanks to \textit{(3)} in \Cref{lemma:A} we have that
	\[
	\xi_\epsilon(x)\cdot\nnu(x)=\rho_\epsilon(x)\bm{\beta}(x)\cdot\nnu(x)\geq \rho_\epsilon(x)\frac{\abs{x}\sigma(\abs{x})}{\mu(x)},\quad\text{for all }x\in \Gamma_r,
	\]
	which implies \eqref{eq:inner_hp1} (notice that $\xi_\epsilon$ is supported in $B_{R_{\partial D}/2}$ for $\epsilon$ sufficiently small); hence, since $\bm{\beta}\in C^1(\R^d,\R^d)$, we can apply \Cref{prop:inner_variations}. We explicitly compute
	\begin{align}
		& D\xi_\epsilon=\nabla\rho_\epsilon\otimes\bm{\beta}+\rho_\epsilon D\bm{\beta}, \label{eq:pohoz_2}\\
		& \dive\xi_\epsilon=\nabla\rho_\epsilon\cdot\bm{\beta}+\rho_\epsilon\dive\bm{\beta}, \label{eq:pohoz_3}\\
		& \d A[\xi_\epsilon]=\rho_\epsilon\d A[\bm{\beta}] \label{eq:pohoz_4}
	\end{align}
	and observe that
	\begin{equation}\label{eq:pohoz_5}
		\bm{\beta}\cdot\nnu=r\quad\text{on }S_r.
	\end{equation}
	Let us now consider the terms appearing in \eqref{eq:innner_th1} one by one. First, thanks to \eqref{eq:pohoz_2}, \eqref{eq:pohoz_1} and \eqref{eq:pohoz_5} we have that
	\begin{align*}
		2\int_{\mathcal{O}_*}D\xi_\epsilon A\nabla v_i\cdot\nabla v_i\dx&= 2\int_{\mathcal{O}_*}(\nabla\rho_\epsilon\otimes\bm{\beta}+\rho_\epsilon D\bm{\beta})A\nabla v_i\cdot\nabla v_i\dx \\
		&\longrightarrow 2\int_{\mathcal{O}_r}D\bm{\beta}A\nabla v_i\cdot\nabla v_i\dx-2r\int_{S_r}\frac{(A\nabla v_i\cdot\nnu)^2}{\mu}\ds,
	\end{align*}
	as $\epsilon\to 0$. Second, \eqref{eq:pohoz_3}, \eqref{eq:pohoz_1} and \eqref{eq:pohoz_5} yield
	\begin{align*}
		-\int_{\mathcal{O}_*}A\nabla v_i\cdot\nabla v_i\dive\xi_\epsilon\dx &=-\int_{\mathcal{O}_*}A\nabla v_i\cdot\nabla v_i(\nabla\rho_\epsilon\cdot\bm{\beta}+\rho_\epsilon\dive\bm{\beta})\dx \\
		&\longrightarrow r\int_{S_r}A\nabla v_i\cdot\nabla v_i\ds-\int_{\mathcal{O}_r}A\nabla v_i\cdot\nabla v_i\dive\bm{\beta}\dx,
	\end{align*}
	and
	\begin{align*}
		\int_{\mathcal{O}_*}(\nabla p_i\cdot\xi_\epsilon+p_i\dive\xi_\epsilon)v_i^2\dx &= \int_{\mathcal{O}_*}(\rho_\epsilon\nabla p_i\cdot\bm{\beta}+p_i(\nabla\rho_\epsilon\cdot\bm{\beta}+\rho_\epsilon\dive\bm{\beta}))v_i^2\dx \\
		&\longrightarrow\int_{\mathcal{O}_r}(\nabla p_i\cdot\bm{\beta}+p_i\dive\bm{\beta})v_i^2\dx-r\int_{S_r}p_i v_i^2\ds,
	\end{align*}
	as $\epsilon\to 0$. Finally, from \eqref{eq:pohoz_4}, \eqref{eq:pohoz_1} and \eqref{eq:pohoz_5} it follows that
	\begin{equation*}
		\int_{\mathcal{O}_*}\d A[\xi_\epsilon]\nabla v_i\cdot\nabla v_i\dx=\int_{\mathcal{O}_*}\rho_\epsilon\d A[\bm{\beta}]\nabla v_i\cdot\nabla v_i\dx\to \int_{\mathcal{O}_r}\d A[\bm{\beta}]\nabla v_i\cdot\nabla v_i\dx,
	\end{equation*}
	as $\epsilon\to 0$. By  rearranging the terms we obtained and summing for $i=1,\dots,N$, we may conclude the proof. 
\end{proof}

One can observe that, locally, the transformed domain $\mathcal{O}_*$ enjoys a geometric property: close to the origin, it is starshaped with respect to the it, and this is contained in the following.

\begin{corollary}\label{cor:starsh}
	Up to reducing $R_{\partial D}$,  $\Gamma_{R_{\partial D}/2}$ is starshaped with respect to the origin. More precisely,
	\[
	x\cdot\nnu(x)\geq \frac{1}{2}\abs{x}\sigma(\abs{x})\geq 0\quad\text{for all }x\in\Gamma_{R_{\partial D}/2}.
	\]
\end{corollary}
\begin{proof}
	From \textit{(3)} and \textit{(4)} in \Cref{lemma:A}, we deduce that
	\begin{equation}\label{eq:starsh_1}
		x\cdot\nnu(x)=A(x)x\cdot\nnu(x)+(I-A(x))x\cdot\nnu(x)\geq  \abs{x}\sigma(\abs{x})\left(1-\kappa\frac{\abs{x\cdot\nnu(x)}}{\abs{x}}\right).
	\end{equation}
	Now, let $\psi\colon B_{R_{\partial D}/2}'\to \R$ be the graph describing $\Gamma_{R_{\partial D}/2}$, i.e. defined implicitly by
	\[
	\psi(x')=\varphi(x')-3\abs{x}\sigma(\abs{x}),
	\]
	where $x=(x',\psi(x'))$. Since $\psi\in C^1(B_{R_{\partial D}/2}')$ and since $\psi(0)=\abs{\nabla_{x'}\psi(0)}=0$, we have that
	\[
	x\cdot\nnu(x)=x'\cdot\nabla_{x'}\psi(x')-\psi(x')=o(\abs{x'})=o(\abs{x})\quad\text{as }\abs{x}\to 0.
	\]
	Plugging this fact into \eqref{eq:starsh_1} concludes the proof, since $|x|\leq R_{\partial D}/2$, up to restricting $R_{\partial D}$.
\end{proof}

\section{Almgren monotonicity formula}\label{sec:almgren}

In order to study the decay properties of the minimizer $u\in H^1_{0,N}(D)$ near a boundary point $x_0\in\partial D$, a fundamental tool is represented by the so called \emph{Almgren frequency function}. While trying to mimic the frequency function at interior points (see e.g. \cite{CL2007} or \cite{TerraciniTavares2012}), one might be led to define the boundary version as
\[
(w,r,x_0)\longmapsto \frac{\displaystyle r\sum_{i=1}^N\int_{D\cap B_r(x_0)} \abs{\nabla w_i}^2\dx}{\displaystyle \sum_{i=1}^N\int_{D\cap\partial B_r(x_0)}w_i^2\ds},
\]
where $w\in H^1_{0,N}(D)$ and $x_0\in\partial D$. However, it turns out that, even in a sufficiently regular setting, apart from the case in which $\partial D\cap B_r(x_0)$ is starshaped with respect to $x_0$, i.e.
\[
x\cdot\nnu(x)\geq 0\quad\text{for }x\in\partial D\cap B_r(x_0),
\]
proving the (almost) monotonicity of this function is highly non-trivial. In order to overcome this issue, we make use of the diffeomorphism introduced in \Cref{sec:equivalent}, which guarantees the right geometric property in the transformed domain and the validity of the Pohozaev-type inequality \Cref{prop:pohozaev}. For any $r>0$ and $w\in H^1_{s,N}(B_r)$, $w\neq 0$, we define
\begin{equation*}
	E_i(w,r):=\frac{1}{r^{d-2}}\int_{B_r}(A\nabla w_i\cdot\nabla w_i-p_i w_i^2)\dx,\qquad H_i(w,r):=\frac{1}{r^{d-1}}\int_{\partial B_r}w_i^2\mu\ds
\end{equation*}
for any $i=1,\dots,N$, where $A$, $p_i$ and $\mu$ are as in \eqref{eq:def_A_p}-\eqref{eq:def_p_mu_beta} and
\begin{equation*}
	E(w,r):=\sum_{i=1}^N E_i(w,r),\qquad H(w,r):=\sum_{i=1}^N H_i(w,r).
\end{equation*}
We now introduce the \emph{frequency function}, defined as
\begin{equation*}
	\mathcal{N}(w,r):=\frac{E(w,r)}{H(w,r)}.
\end{equation*}
In particular, in this section we always consider the frequency function associated to a transformed minimizer $v^{x_0}\in H^1_{0,N}(\mathcal{O}^{x_0})$, for some $x_0\in\partial D$. We recall that $v=v^{x_0}$ is as in \eqref{eq:def_v}, while $\mathcal{O}=\mathcal{O}^{x_0}$ is as in \eqref{eq:Omega}. With these choices and $r\leq R_{\partial D}/2$, we restrict the integrals to $\mathcal{O}_r$, thus having that
\begin{align*}
	&E(r):=E(v^{x_0},r)=\sum_{i=1}^N\frac{1}{r^{d-2}}\int_{\mathcal{O}_r}(A\nabla v_i\cdot\nabla v_i-p_i v_i^2)\dx, \\
	&H(r):=H(v^{x_0},r)=\sum_{i=1}^N\frac{1}{r^{d-1}}\int_{S_r}v_i^2\mu\ds
\end{align*}
and
\begin{equation*}
	\mathcal{N}(r):=\frac{E(r)}{H(r)}.
\end{equation*}
If one unravels these expressions, and write them in terms of the original minimizer $u$, one obtains
\[
E(r)=\sum_{i=1}^N\frac{1}{r^{d-2}}\int_{\Psi_{x_0}(B_r)\cap D}(|\nabla u_i|^2-\lambda_i u_i^2)\dx\quad\text{and}\quad H(r):=\sum_{i=1}^N\frac{1}{r^{d-1}}\int_{\partial \Psi_{x_0}(B_r)\cap D}u_i^2\ds,
\]
that is the usual frequency function but defined on perturbed balls, rather than true balls. Following the standard path, we now want to compute the derivative of $\mathcal{N}$  with respect to $r$.	We start by analyzing the derivative of the height function $H$.
\begin{lemma}\label{lemma:H'}
	We have that $H\in W^{1,1}(0,R_{\partial D}/2)$ and the following holds a.e. in $(0,R_{\partial D}/2)$
	\begin{align}
		&\abs{H'(r)-\sum_{i=1}^N\frac{2}{r^{d-1}}\int_{S_r}v_iA\nabla v_i\cdot\nnu\ds}\leq CH(r)\frac{\sigma(r)}{r}, \label{eq:H'_th2}\\
		&\abs{H'(r)-\frac{2}{r}E(r)}\leq CH(r)\frac{\sigma(r)}{r}\label{eq:H'_th3}
	\end{align}
	for some constant $C=C(d,D)>0$.
\end{lemma}
\begin{proof}
	Standard computations (see e.g. \cite[Lemma 5.3]{FF2013}) yield
	\begin{equation}\label{eq:H'1}
		H_i'(r)=\frac{2}{r^{d-1}}\int_{S_r}v_i\partial_{\nnu} v_i \mu\ds+\frac{1}{r^{d-1}}\int_{S_r}v_i^2\nabla \mu\cdot\nnu\ds
	\end{equation}
	in a distributional sense. Therefore, since $v_i\in C^{0,1}(\mathcal{O})$ and in view of \Cref{it:diffeo5} in \Cref{lemma:A} we have that
	\[
	\abs{H'(r)}\leq Cr\sigma(r)\quad\text{a.e. in }(0,R_{\partial D}/2),
	\]
	for some constant $C>0$ depending on $\max_{i=1,\dots,N}\norm{v_i}_{C^{0,1}(\mathcal{O})}$; this implies that $H$ is in $W^{1,1}(0,R_{\partial D}/2)$. In order to prove \eqref{eq:H'_th2}, we first observe that, thanks to divergence theorem
	\begin{equation*}
		H(r)=\sum_{i=1}^N\frac{1}{r^{d-1}}\int_{S_r}v_i^2\mu\ds=\sum_{i=1}^N\frac{1}{r^{d-1}}\int_{S_r}v_i^2\bm{\alpha}\cdot\nnu\ds=\sum_{i=1}^N\frac{1}{r^{d-1}}\int_{\mathcal{O}_r}\dive(v_i^2\bm{\alpha})\ds,
	\end{equation*}
	since $\bm{\alpha}(x)=A(x)x/\abs{x}$ satisfies $\bm{\alpha}\cdot\nnu=\bm{\alpha}\cdot x/\abs{x}=\mu$ on $S_r$. Therefore, there holds
	\begin{align*}
		H'(r)&=-\frac{d-1}{r}H(r)+\sum_{i=1}^N\frac{1}{r^{d-1}}\int_{S_r}\dive(v_i^2\bm{\alpha})\ds \\
		&=-\frac{d-1}{r}H(r)+\sum_{i=1}^N\frac{1}{r^{d-1}}\int_{S_r}(2v_iA\nabla v_i\cdot\nnu+v_i\dive\bm{\alpha})\ds
	\end{align*}
	for a.e. $r\in(0,R_{\partial D}/2)$. Now, thanks to \textit{(5)} in \Cref{lemma:A}, we obtain \eqref{eq:H'_th2}. Finally, \Cref{lemma:int_by_parts} combined with \eqref{eq:H'_th2} implies \eqref{eq:H'_th3} and the proof is concluded.
\end{proof}

We recall a standard Poincaré-type inequality and a straightforward consequence of it, which consists in a coercivity property.
\begin{lemma}[Poincaré Inequality]\label{lemma:poincare}
	There holds
	\[
	\int_{B_r}w^2\ds\leq \frac{1}{d-1}\left[r^2\int_{B_r}\abs{\nabla w}^2\dx+r\int_{\partial B_r}w^2\ds\right]
	\]
	for all $w\in H^1(B_r)$ and all $r>0$.
\end{lemma}
\begin{proof}
	The proof simply follows by integration of the following identity
	\[
	\dive(w^2 x)=2w\nabla w\cdot x+dw^2
	\]
	in $B_r$ and applying divergence theorem and Young's inequality.
\end{proof}

We hereafter denote
\begin{equation}\label{def:R_0}
	R_0:=\min\left\{\frac{R_{\partial D}}{2},\sqrt{\frac{d-1}{6\Lambda}}\right\}.
\end{equation}
\begin{corollary}[Coercivity]\label{cor:coerc}
	Let $R_0>0$ be as in \eqref{def:R_0}. Then
	\begin{equation}\label{eq:coerc_th1}
		\int_{\mathcal{O}_r}(A\nabla v_i\cdot\nabla v_i-p_iv_i^2)\dx+\frac{1}{r}\int_{S_r}v_i^2\mu\ds\geq \frac{1}{2}\left(\int_{\mathcal{O}_r}A\nabla v_i\cdot\nabla v_i\dx+\frac{1}{r}\int_{S_r}v_i^2\mu\ds\right)
	\end{equation}
	for all $r\in(0,R_0)$ and all $i=1,\dots,N$. In particular
	\begin{equation}\label{eq:coerc_th2}
		\int_{\mathcal{O}_r}p_i v_i^2\dx\leq \frac{3\lambda_ir^d}{d-1}(E_i(r)+H_i(r))
	\end{equation}
	for all $r\in(0,R_0)$ and all $i=1,\dots,N$.
\end{corollary}
\begin{proof}
	From \Cref{lemma:poincare} and \Cref{cor:A} it follows that
	\begin{equation}\label{eq:coerc_1}
		\begin{aligned}
			\int_{\mathcal{O}_r}p_iv_i^2\dx \leq \frac{3\Lambda}{2}\int_{\mathcal{O}_r}v_i^2\dx 
			&\leq \frac{3\Lambda}{2(d-1)}\left[r^2\int_{\mathcal{O}_r}\abs{\nabla v_i}^2\dx+r\int_{S_r}v_i^2\ds\right] \\
			&\leq \frac{3\Lambda}{d-1}\left[r^2\int_{\mathcal{O}_r}A\nabla v_i\cdot\nabla v_i\dx+ r\int_{S_r}v_i^2\mu\ds\right],
		\end{aligned}
	\end{equation}
	which implies that
	\begin{equation*}
		\int_{\mathcal{O}_r}(A\nabla v_i\cdot\nabla v_i-p_iv_i^2)\dx+\frac{1}{r}\int_{S_r}v_i^2\mu\ds\geq \left[1-\frac{3\Lambda }{d-1}r^2\right]\left[\int_{\mathcal{O}_r}A\nabla v_i\cdot\nabla v_i\dx+\frac{1}{r}\int_{S_r}v_i^2\mu\ds\right].
	\end{equation*}
	The proof of \eqref{eq:coerc_th1} is complete by taking $r\leq R_0$. Finally, \eqref{eq:coerc_th2} follows from \Cref{cor:A} and \eqref{eq:coerc_1}. 
\end{proof}
With this result in our hand, we can prove that the Almgren frequency function is well defined at any boundary point.
\begin{lemma}\label{lemma:H>0}
	Let $R_0>0$ be as in \eqref{def:R_0}. We have that $H(r)>0$ for any $r\leq R_0$.
\end{lemma}
\begin{proof}
	Assume by contradiction that
	\[
	H(r)=\sum_{i=1}^N\frac{1}{r^{d-1}}\int_{S_r}v_i^2\mu\ds=0
	\]
	for some $r\leq R_0$. This implies that $v_i\equiv 0$ on $S_r$ for any $i=1,\dots,N$. If we combine this fact with \Cref{lemma:int_by_parts} and \eqref{eq:coerc_th1}, we find that $v_i\equiv 0$ in $\mathcal{O}_r$ for any $i=1,\dots,N$, which in contradiction with the unique continuation property at interior points, proved for instance in \cite{CL2007}. This concludes the proof.
\end{proof}

We now pass to the study of the derivative of the energy $E$. Here a crucial role is played by the Pohozaev inequality proved in \Cref{prop:pohozaev}.

\begin{lemma}\label{lemma:E'}
	Let $R_0>0$ be as in \eqref{def:R_0}. We have that $E\in W^{1,1}(0,R_0)$ and the following holds a.e. in $(0,R_0)$
	\begin{equation*}
		E'(r)\geq \sum_{i=1}^N \frac{2}{r^{d-2}}\int_{S_r}\frac{1}{\mu}(A\nabla v_i\cdot\nnu)^2\ds-C\,\frac{\sigma(r)}{r}(E(r)+H(r)),
	\end{equation*}
	for some $C=C(d,D,N)>0$.
\end{lemma}
\begin{proof}
	In view of the definition of $E_i$, the following holds in a distributional sense
	\begin{equation}\label{eq:E'1}
		E_i'(r)=-\frac{d-2}{r}E_i(r)+\frac{1}{r^{d-2}}\int_{S_r}(A\nabla v_i\cdot\nabla v_i-p_iv_i^2)\ds
	\end{equation}
	and, since $v_i\in C^{0,1}(\mathcal{O})$, then
	\begin{equation*}
		\abs{E(r)}\leq C r\quad\text{a.e. in }(0,R_0),
	\end{equation*}
	for some $C>0$ depending on $\max_{i=1,\dots,N}\norm{v_i}_{C^{0,1}(\mathcal{O})}$: this implies that $E\in W^{1,1}(0,R_0)$. Now, combining \eqref{eq:E'1} with \Cref{prop:pohozaev} yields
	\begin{multline}\label{eq:E'2}
		E'(r)\geq -\frac{d-2}{r}E(r)+\sum_{i=1}^N\frac{1}{r^{d-1}}\Bigg[2r\int_{S_r}\frac{1}{\mu}(A\nabla v_i\cdot\nnu)^2\ds \\
		+\int_{\mathcal{O}_r}(A\dive\bm{\beta}
		+\d A[\bm{\beta}]-2D\bm{\beta} A)\nabla v_i\cdot\nabla v_i\dx 
		-\int_{\mathcal{O}_r}(\nabla p_i\cdot\bm{\beta}+p_i\dive\bm{\beta})v_i^2\dx\Bigg]
	\end{multline}
	In view of the estimates on $A$, $\bm{\beta}$ and $p_i$ obtained in \Cref{lemma:A} and \Cref{cor:A} we have that
	\begin{equation*}
		\int_{\mathcal{O}_r}(A\dive\bm{\beta}
		+\d A[\bm{\beta}]-2D\bm{\beta} A)\nabla v_i\cdot\nabla v_i\dx 
		\geq (d-2-C\sigma(r))\int_{\mathcal{O}_r}A\nabla v_i\cdot\nabla v_i\dx
	\end{equation*}
	and
	\begin{equation*}
		-\int_{\mathcal{O}_r}(\nabla p_i\cdot\bm{\beta}+p_i\dive\bm{\beta})v_i^2\dx\geq (d-C\sigma(r))\int_{\mathcal{O}_r}p_iv_i^2\dx,
	\end{equation*}
	for some $C>0$ depending only on $d$, $D$ and $N$, and any $r\in(0,R_0)$. Hence, combining \eqref{eq:E'2} with these two inequalities, we obtain that
	\begin{multline*}
		E'(r)\geq \sum_{i=1}^N\frac{1}{r^{d-1}}\Bigg[2r\int_{S_r}\frac{1}{\mu}(A\nabla v_i\cdot\nnu)^2\ds \\
		-C\left((1+\sigma(r))\int_{\mathcal{O}_r}p_iv_i^2\dx+\sigma(r)\int_{\mathcal{O}_r}A\nabla v_i\cdot\nabla v_i\dx\right)\Bigg].
	\end{multline*}
	Finally, by applying \eqref{eq:coerc_th2} and rearranging the terms, we obtain the thesis.
\end{proof}

At this point we have all the ingredients needed in order to prove almost-monotonicity of the Almgren frequency function. We point out that, for this result to be true, only the $1$-Dini condition on $\sigma$ is really needed.
\begin{theorem}[Monotonicity of the Almgren function]\label{thm:N'}
	Let $R_0>0$ be as in \eqref{def:R_0}. Then $\mathcal{N}\in W^{1,1}(0,R_0)$ and for a.e. $r\in(0,R_0)$ there holds
	\begin{equation}\label{eq:N'_th1}
		\mathcal{N}'(r)\geq -C_{\textup{A}}\,\frac{\sigma(r)}{r}\left(\mathcal{N}(r)+1\right),
	\end{equation}
	for some $C_{\textup{A}}>0$ depending only on $d$, $D$ and $N$. In particular,
	\begin{equation}\label{eq:N'_th2}
		\left(e^{C_{\textup{A}}\int_0^r\frac{\sigma(t)}{t}\d t}\left(\mathcal{N}(r)+1\right)\right)'\geq 0
	\end{equation}
	and there exists $\lim_{r\to 0}\mathcal{N}(r)\in[0,\infty)$.
\end{theorem}
\begin{proof}
	First of all, $\mathcal{N}$ is well defined in $(0,R_0)$ in view of \Cref{lemma:H>0} and, since both $E$ and $H$ are in $W^{1,1}(0,R_0)$, then $\mathcal{N}\in W^{1,1}_{\textup{loc}}(0,R_0)$. We now compute its derivative. Thanks to \eqref{eq:H'_th3} we have that
	\begin{align*}
		\mathcal{N}'(r)=\frac{E'(r)H(r)-E(r)H'(r)}{H^2(r)}&\geq \frac{E'(r)H(r)-E(r)\left(\frac{2}{r}E(r)+CH(r)\frac{\sigma(r)}{r}\right)}{H^2(r)} \\
		&=\frac{E'(r)H(r)-\frac{2}{r}E^2(r)}{H^2(r)}-C\frac{\sigma(r)}{r}\mathcal{N}(r).
	\end{align*}
	Keeping in mind \Cref{lemma:sigma_nondeg}, we now combine this with \Cref{lemma:E'} and \Cref{lemma:int_by_parts} and obtain
	\begin{multline*}
		\mathcal{N}'(r)\geq \frac{1}{H^2(r)}\Bigg[H(r)\frac{2}{r^{d-2}}\sum_{i=1}^{N}\int_{S_r}\frac{1}{\mu}(A\nabla v_i\cdot\nnu)^2\ds-\frac{2}{r}\left(\sum_{i=1}^N\frac{1}{r^{d-2}}\int_{S_r}v_iA\nabla v_i\cdot\nnu\ds\right)^2 \\
		-C\frac{\sigma(r)}{r}(E(r)+H(r))H(r)\Bigg]-C\frac{\sigma(r)}{r}\mathcal{N}(r),
	\end{multline*}
	which, up to rearranging the terms, implies that
	\begin{align*}
		\mathcal{N}'(r)&\geq \frac{2r}{\sum_{i=1}^N\int_{S_r}v_i^2\mu\ds}\Bigg[\Bigg(\sum_{i=1}^N\int_{S_r}v_i^2\mu\ds\Bigg)\Bigg(\sum_{i=1}^{N}\int_{S_r}\frac{1}{\mu}(A\nabla v_i\cdot\nnu)^2\ds\Bigg) \\
		&\qquad\qquad\qquad\qquad\quad- \Bigg(\sum_{i=1}^N\int_{S_r}v_iA\nabla v_i\cdot\nnu\ds\Bigg)^2\Bigg] \\ &\quad-C\frac{\sigma(r)}{r}(1+\mathcal{N}(r))-C\frac{\sigma(r)}{r}\mathcal{N}(r).
	\end{align*}
	Now, thanks to Cauchy-Schwarz inequality, this can be bounded from below as in \eqref{eq:N'_th1}. Hence, there exists $\lim_{r\to 0}\mathcal{N}(r)$ and it is finite, which also implies that $\mathcal{N}\in W^{1,1}(0,R_0)$. The fact that the limit is non-negative is a consequence of \eqref{eq:coerc_th2} and this concludes the proof.
\end{proof}

It is well known that the almost-monotonicity property of the Almgren function immediately yields some consequences, such as boundedness of the limit as $r\to 0$ and estimates on the growth of $H$.
Before stating them, we introduce the following notation
\begin{equation}\label{def:gamma_x_0}
	\gamma(x_0):=\lim_{r\to 0}\mathcal{N}(v^{x_0},r)=\lim_{r\to 0}\mathcal{N}(r).
\end{equation}
The following result contains a trivial consequence of \Cref{thm:N'}, that is boundedness of the frequency function.
\begin{corollary}\label{cor:N_bound}
	Let $R_0>0$ be as in \eqref{def:R_0}. Then, there exists $C_{\textup{b}}>0$ depending only on $d$, $D$ and $N$ (independent of $x_0\in\partial D$) such that
	\begin{equation}\label{eq:N_bound_th1}
		\mathcal{N}(r)\leq C_{\textup{b}}\left(\mathcal{N}(R_0)+1\right)\quad\text{for a.e. }r\leq R_0
	\end{equation}
	and
	\begin{equation}\label{eq:N_bound_th2}
		\mathcal{N}(r)\leq C_{\textup{b}}\quad\text{for a.e. }r\leq R_0.
	\end{equation}
	Moreover, if $\gamma(x_0)$ is as in \eqref{def:gamma_x_0}, then $1\leq\gamma(x_0)\leq C_{\textup{b}}$  (uniformly in $x_0\in\partial D$).
\end{corollary}
\begin{proof}
	We have that \eqref{eq:N_bound_th1} is a straightforward consequence of \Cref{thm:N'}, while \eqref{eq:N_bound_th2} can be obtained through a simple contradiction argument (which involves also the minimizer $u$). Finally, the fact that $\gamma(x_0)\geq 1$ is a consequence of Lipschitz continuity of the solution. More precisely, from \Cref{prop:lipschitz}, \Cref{lemma:A} and the fact that $v(0)=0$, one easily obtains that
	\begin{equation}\label{eq:N_bound_1}
		H(r)\leq C r^2\quad\text{for all }r\leq R_0.
	\end{equation}
	Let us now assume by contradiction that $\gamma(x_0)<1$, which means that there exists $\epsilon>0$ such that, for $r$ sufficiently small $\mathcal{N}(r)\leq 1-\epsilon$. Hence, from \eqref{eq:H'_th3} we derive that
	\begin{equation*}
		\frac{H'(r)}{H(r)}\leq \frac{2(1-\epsilon)}{r}+C\frac{\sigma(r)}{r},\quad\text{for }r\leq R
	\end{equation*}
	and $R$ sufficiently small, and by integration in $(r,R)$ this in turn implies that
	\[
	C_Rr^{2(1-\epsilon)}\leq H(r)\quad\text{for }r\leq R
	\]
	and for some $C_R>0$ depending on $R$. Since this contradicts \eqref{eq:N_bound_1} we conclude the proof.
\end{proof}

Another consequence of \Cref{thm:N'}, combined with \Cref{cor:N_bound}, is an almost-minimality condition for the pure Dirichlet energy for minimizers (and their perturbations). More precisely, we have the following.

\begin{proposition}[Almost minimality]\label{prop:almost}
	There exists $C_{\textup{am}}>0$, depending on $d$, $D$ and $N$, such that
	\begin{equation}\label{eq:almost_th1}
		\sum_{i=1}^N\int_{B_r}\abs{\nabla v_i}^2\dx\leq (1+C_{\textup{am}}\sigma(r))\sum_{i=1}^N\int_{B_r}\abs{\nabla w_i}^2\dx,
	\end{equation}
	for all $r\leq R_0$ (up to reducing $R_0$) and all $w\in H^1_{s,N}(B_r)$ such that $v_i-w_i\in H^1_{0,N}(B_r)$ for all $i=1,\dots,N$.
\end{proposition}
\begin{proof}
	Let $w\in H^1_{s,N}(B_r)$ be such that $v_i-w_i\in H^1_{0,N}(B_r)$ for all $i=1,\dots,N$ and let $\varphi:=v-w$. We first point out that we may restrict to $i\in\{1,\dots,N\}$ such that
	\begin{equation}\label{eq:almost_5}
		\int_{B_r}\abs{\nabla(v_i+\varphi_i)}^2\dx\leq \int_{B_r}\abs{\nabla v_i}^2\dx,
	\end{equation}
	otherwise \eqref{eq:almost_th1} is trivial. In particular, this implies that
	\begin{equation}\label{eq:almost_1}
		\norm{\nabla\varphi_i}_{L^2(B_r)}\leq 2\norm{\nabla v_i}_{L^2(B_r)}.
	\end{equation}
	Since by assumption
	\begin{equation*}
		\int_{\mathcal{O}}p_i u_i^2=1\quad\text{for all }i=1,\dots,N,
	\end{equation*}
	we observe that, thanks to \Cref{cor:A}, H\"older inequality, Poincaré inequality and \eqref{eq:almost_1}, there holds
	\begin{equation}\label{eq:almost_2}
		\int_\mathcal{O} p_i(v_i+\varphi_i)^2\dx\geq 1-Cr\norm{v_i}_{L^2(B_r)}\norm{\nabla v_i}_{L^2(B_r)}.
	\end{equation}
	On the other hand, since $\gamma(x_0)\geq 1$ for all $x_0\in\partial D$, see \Cref{cor:N_bound}, we have that
	\begin{equation*}
		\mathcal{N}(r)\geq \frac{1}{2}\quad\text{for all }r\leq R_0,
	\end{equation*}
	up to restricting $R_0$. As a direct consequence, we have that
	\begin{equation*}
		\norm{v_i}_{L^2(B_r)}^2\leq Cr^2\sum_{i=1}^N\norm{\nabla v_i}_{L^2(B_r)}^2,
	\end{equation*}
	which implies that
	\begin{equation*}
		\norm{v_i}_{L^2(B_r)}\norm{\nabla v_i}_{L^2(B_r)}\leq Cr\sum_{i=1}^N\norm{\nabla v_i}_{L^2(B_r)}^2.
	\end{equation*}
	Combining this with \eqref{eq:almost_2}. we obtain that
	\begin{equation}\label{eq:almost_3}
		\int_\mathcal{O} p_i(v_i+\varphi_i)^2\dx\geq 1-Cr^2\sum_{i=1}^N\norm{\nabla v_i}_{L^2(B_r)}^2\quad\text{for all $i$ such that \eqref{eq:almost_5} holds,}
	\end{equation}
	Now, thanks to \Cref{cor:A}, up to reducing $R_0$ (still depending only on $d$, $D$ and $N$) in such a way that
	\begin{equation}\label{eq:almost_4}
		\frac{1}{\displaystyle 1-Cr^2\sum_{i=1}^N\norm{\nabla v_i}_{L^2(B_r)}^2}\leq 1+2Cr\sum_{i=1}^N\norm{\nabla v_i}_{L^2(B_r)}^2,
	\end{equation}
	being $r\leq R_0$. At this point, combining the minimality of $v$ with \eqref{eq:almost_3} and \eqref{eq:almost_4} (up to renaming the constant $C$) we have that
	\begin{equation*}
		\sum_{i=1}^N\int_{\mathcal{O}}A\nabla v_i\cdot\nabla v_i\dx\leq \left(1+Cr\sum_{i=1}^N\norm{\nabla v_i}_{L^2(B_r)}^2\right)\sum_{i=1}^N\int_{\mathcal{O}}A\nabla(v_i+\varphi_i)\cdot\nabla(v_i+\varphi_i)\dx,
	\end{equation*}
	which, exploiting again \Cref{cor:A}, easily gives that
	\begin{align*}
		\sum_{i=1}^N\int_{B_r}A\nabla v_i\cdot\nabla v_i\dx & \leq Cr\sum_{i=1}^N\norm{\nabla v_i}_{L^2(B_r)}^2 \\
		& \quad+\left(1+Cr\sum_{i=1}^N\norm{\nabla v_i}_{L^2(B_r)}^2\right)\sum_{i=1}^N\int_{B_r}A\nabla(v_i+\varphi_i)\cdot\nabla(v_i+\varphi_i)\dx\\
		&\leq Cr \sum_{i=1}^N\int_{B_r}A\nabla v_i\cdot\nabla v_i\dx \\
		&\quad  +\left(1+Cr\sum_{i=1}^N\norm{\nabla v_i}_{L^2(B_r)}^2\right)\sum_{i=1}^N\int_{B_r}A\nabla(v_i+\varphi_i)\cdot\nabla(v_i+\varphi_i)\dx.
	\end{align*}
	Hence, up to choosing $R_0$ small enough, we have that
	\begin{equation*}
		\sum_{i=1}^N\int_{B_r}A\nabla v_i\cdot\nabla v_i\dx\leq \left(1+Cr\right)\sum_{i=1}^N\int_{B_r}A\nabla(v_i+\varphi_i)\cdot\nabla(v_i+\varphi_i)\dx.
	\end{equation*}
	Making use a final time of \Cref{lemma:A}, rearranging the terms and choosing again $R_0$ small enough, we conclude the proof.
\end{proof}

Since every component $v_i$ is a non-negative subsolution in $\R^d$, by classical regularity results (namely, DeGiorgi-Nash-Moser estimates, see e.g. \cite[Theorem 4.1]{HanLin}), we have the following.
\begin{lemma}[Uniform $L^\infty$-bound]\label{lemma:L_infty}
	There exists a constant $C>0$, depending on $d$, $D$ and $N$ such that
	\[
	\norm{v_i}_{L^\infty(B_r)}^2\leq Cr^{-d}\norm{v_i}_{L^2(B_{2r})}^2
	\]
	for all $r>0$ and all $i=1,\dots,N$.
\end{lemma}

In the following lemma we derive growth estimates for $H$ and for the minimizer $v$.
\begin{lemma}\label{lemma:decay}
	Let $R_0>0$ be as in \eqref{def:R_0}. Then, there exists $C_{\textup{H}}>0$ depending only on $d$, $D$ and $N$ such that
	\begin{enumerate}
		\item[(i)] $H(r)\leq C_{\textup{H}}H(R_0)\, r^{2\gamma(x_0)}$ for any $r\leq R_0$; \label{it:decay_th1}
		\item[(ii)] $\abs{v_i(x)}^2\leq C_{\textup{H}}H(R_0)\abs{x}^{2\gamma(x_0)}$ for any $x\in B_{R_0/2}$ and all $i=1,\dots,N$; \label{it:decay_th2}
	\end{enumerate}
\end{lemma}
\begin{proof}
	In order to prove \textit{(i)}, we first observe that, thanks to \eqref{eq:H'_th2} and \Cref{thm:N'}, we have
	\begin{equation}\label{eq:decay_1}
		\frac{H'(r)}{H(r)}\geq \frac{2}{r}\mathcal{N}(r)-C\frac{\sigma(r)}{r}=\frac{2\gamma(x_0)}{r}+\frac{2}{r}\int_0^r\mathcal{N}'(t)\d t-C\frac{\sigma(r)}{r},
	\end{equation}
	for a.e. $r\leq R_0$ and some $C>0$ (depending on $d$, $D$ and $N$). Then, in view of monotonicity of $\mathcal{N}$ and \eqref{eq:N_bound_th1}, we get that 
	\begin{equation*}
		\mathcal{N}'(t)\geq -C\frac{\sigma(t)}{t}\left(\mathcal{N}(R_0)+1\right),
	\end{equation*}
	for a.e. $t\leq R_0$. Combining this with \eqref{eq:decay_1} and \Cref{lemma:sigma_nondeg} leads to
	\begin{equation}\label{eq:decay_2}
		\frac{H'(r)}{H(r)}\geq \frac{2\gamma(x_0)}{r}-C\frac{\mathcal{N}(R_0)+1}{r}\int_0^r\frac{\sigma(t)}{t}\d t
	\end{equation}
	for a.e. $r\leq R_0$ and some constant $C>0$ depending only on $d$, $D$ and $N$. Now, since by assumption
	\[
	\frac{1}{r}\int_0^r\frac{\sigma(t)}{t}\d t\in L^1(0,R_0),
	\]
	by integrating \eqref{eq:decay_2} in $(r,R_0)$, and using \eqref{eq:N_bound_th2}, we obtain \textit{(i)}. Let us pass to the proof of \textit{(ii)}. Let $x\in B_{R_0/2}$ and let $r=\abs{x}$. Thanks to \Cref{lemma:L_infty} we have that
	\begin{equation}\label{eq:decay_3}
		\abs{v_i(x)}^2\leq \norm{v_i}_{L^\infty(B_r)}^2\leq C r^{-d}\norm{v_i}_{L^2(B_{2r})}^2.
	\end{equation}
	Now, by integration of \textit{(i)}, \Cref{cor:A} and \eqref{eq:N_bound_th2} we obtain that
	\begin{equation*}
		r^{-d}\sum_{i=1}^N\int_{B_{2r}}v_i^2\dx\leq C H(R_0)r^{2\gamma(x_0)}
	\end{equation*}
	for some $C>0$ depending on $d$, $D$ and $N$. If we combine this last estimate with \eqref{eq:decay_3} we obtain \textit{(ii)}, thus concluding the proof.
\end{proof}

We now consider the following rescalings of the function $v\in H^1_{0,N}(\mathcal{O})$. For any $x_0\in\partial D$, we define
\begin{equation}\label{def:tilde_v_r}
	\tilde{v}^r(x)=\tilde{v}^{r,x_0}(x):=\frac{v^{x_0}(rx)}{\sqrt{H(r)}}=\frac{u(\textbf{Q}(\Psi(rx))+x_0)}{\sqrt{H(r)}}, \quad\text{for }x\in \frac{1}{r}\mathcal{O}
\end{equation}
and we assume $\tilde{v}^r$ to be trivially extended outside its domain. We call $\tilde{v}^r$ the \emph{Almgren rescaling} of $v$. Understanding the behavior of the Almgren rescalings as $r\to 0$ plays a crucial role in the study of the free boundary. The proof of their pre-compactness is quite standard once a monotonicity result such as \Cref{thm:N'} is available. Nevertheless, the lack of regularity of the solutions (such as in the case we are treating in the present paper) might cause non-trivial technical issues in the proof. In order to overcome these difficulties we employ the method introduced in \cite[Section 6]{FFT2012}, which turned out to be successful in other non-smooth situations (e.g. domains with cracks, see \cite{dLF}). Hence, we report here the main steps needed in order to prove compactness (up to subsequences) of Almgren rescalings.

\begin{lemma}[Boundedness of Almgren rescalings]\label{lemma:boundedness}
	There exists $C_{\textup{b}}'>0$ depending only on $d$, $D$ and $N$ such that
	\begin{equation*}
		\norm{\tilde{v}_i^r}_{H^1(B_1)}^2\leq C_{\textup{b}}'\quad\text{for all }r\leq R_0~\text{and all }i=1,\dots,N.
	\end{equation*}
\end{lemma}
\begin{proof}
	For $r\leq R_0$, thanks to \Cref{cor:A} and \eqref{eq:coerc_th2} we have that
	\begin{align*}
		\int_{B_1}\abs{\nabla \tilde{v}_i^r }^2\dx=\frac{r^{2-d}}{H(r)}\int_{B_r}\abs{\nabla v_i}^2\dx &\leq C\frac{r^{2-d}}{H(r)}\int_{B_r}A\nabla v_i\cdot\nabla v_i\dx \\
		&\leq C\left[\frac{E_i(r)}{H(r)}+r^2\left(\frac{E_i(r)}{H(r)}+\frac{H_i(r)}{H(r)}\right)\right].
	\end{align*}
	Hence, summing for $i=1,\dots,N$ and exploiting \Cref{cor:N_bound}, we obtain
	\begin{equation*}
		\sum_{i=1}^N\int_{B_1}\abs{\nabla \tilde{v}_i^r}^2\dx\leq C\left[\mathcal{N}(r)+r^2(\mathcal{N}(r)+1)\right]\leq C.
	\end{equation*}
	Moreover, thanks \Cref{cor:A} and \eqref{eq:coerc_th2}, we have that
	\begin{equation*}
		\sum_{i=1}^N\int_{B_1}(\tilde{v}_i^r)^2\dx=\frac{r^{-d}}{H(r)}\sum_{i=1}^N\int_{B_r}v_i^2\dx\leq C \frac{r^{-d}}{H(r)}\sum_{i=1}^N\int_{B_r}p_i v_i^2\dx \leq C(\mathcal{N}(r)+1).
	\end{equation*}
	The proof is thereby complete in view of \Cref{cor:N_bound}.
\end{proof}

The following result is essential in the proof of compactness of Almgren rescalings in case of lack of regularity, and their proof is exactly the same as in \cite[Section 6]{FFT2012}, hence we omit it.

\begin{lemma}\label{lemma:delta}
	There exists $R_0'\leq R_0$ such that, for all $r\leq R_0'$ there exists $\delta_r\in[1,2]$ such that
	\begin{equation*}
		\int_{\partial B_1}|\nabla \tilde{v}_i^{r\delta_r}|^2\ds\leq C_{\delta},
	\end{equation*}
	where $C_\delta>0$ depends only on $d$, $D$ and $N$.
\end{lemma}

At this point, we are able to prove compactness of Almgren rescalings.

\begin{proposition}[Compactness of Almgren rescalings]\label{prop:almgren_rescalings}
	For any $\{r_n\}_{n\in\N}$ such that $r_n\to 0^+$ as $n\to\infty$, there exists $U=U^{x_0}\in \mathcal{B}_{\gamma(x_0)}$ such that
	\begin{equation*}
		\sum_{i=1}^N\int_{S_1^+}|U_i|^2\ds=1
	\end{equation*}
	and
	\begin{equation*}
		\tilde{v}^{r_n}\to U\quad\text{in }H^1(B_1,\R^N)~\text{and in }C^{0,\alpha}(B_1)~\text{for all }\alpha\in(0,1),
	\end{equation*}
	up to a subsequence, as $n\to\infty$.
\end{proposition}
\begin{proof}
	Let $r_n\to 0$ as $n\to\infty$, let $\delta_{r_n}$ be as in \Cref{lemma:delta} and let $\rho_n:=r_n\delta_n$. Thanks to \Cref{lemma:boundedness} $\{\tilde{v}_i^{\rho_n}\}_{n\in\N}$ is bounded in $H^1(B_1)$, hence there exists a subsequence, still denoted by $\{\rho_n\}_{n\in\N}$, and a function $U\in H^1(B_1,\R^N)$ such that 
	\begin{align}
		&\tilde{v}_i^{\rho_n}\weak U_i\quad\text{weakly in }H^1(B_1), \label{eq:compact_1}\\
		&\tilde{v}_i^{\rho_n}\to U_i\quad\text{strongly in }L^2(B_1)~\text{and }L^2(\partial B_1),\label{eq:compact_2} \\
		&\tilde{v}_i^{\rho_n}\to U_i\quad\text{a.e. in }B_1 \label{eq:compact_3}
	\end{align}
	as $n\to\infty$, for any $i=1,\dots,N$. One can immediately observe that, by definition, since $\mu(\rho_n x)\to 1$ uniformly in $B_1$ and thanks to \eqref{eq:compact_2}, there holds
	\begin{equation*}
		\sum_{i=1}^N\int_{\partial B_1}U_i^2\ds=1,
	\end{equation*}
	thus implying that $(U_1,\dots,U_N)\not\equiv (0,\dots,0)$. Moreover, from \eqref{eq:compact_3} we deduce that $U_iU_j\equiv 0$ a.e. in $B_1$ for all $i,j=1,\dots,N$ such that $i\neq j$. In addition, since
	\[
	\tilde{v}_i^{\rho_n}=0\quad\text{in }B_1\setminus \frac{1}{\rho_n}\mathcal{O}_{\rho_n}
	\]	
	and since
	\[
	B_1\cap\frac{1}{\rho_n}\mathcal{O}_{\rho_n}=\left\{(x',x_d)\colon x_d>\frac{1}{\rho_n}\varphi(\rho_nx')-3\abs{x}\sigma(\rho_n\abs{x})\right\}\cap B_1
	\]
	converges to $B_1^+$ in the Hausdorff sense, as $n\to\infty$, then from \eqref{eq:compact_3} we deduce that $U_i(x',x_d)=0$ for $x_d\leq 0$. Let us now pass to the proof of strong convergence. First of all, in view of \Cref{lemma:delta} and \textit{(4)} in \Cref{lemma:A}, we have that for all $i=1,\dots,N$ there exists $h_i\in L^2(\partial B_1)$ such that, up to a subsequence
	\begin{equation}\label{eq:compact_4}
		A(\rho_n \cdot)\nabla \tilde{v}_i^{\rho_n}\cdot\nnu\weak h_i\quad\text{weakly in }L^2(\partial B_1)
	\end{equation}
	as $n\to\infty$. Now, let $\mathcal{U}_i:=\{x\in B_1\colon U_i(x)>0\}$ and let $\varphi_i\in C_c^\infty(\mathcal{U}_i\cap \overline{B_1})$. Then, in view of \eqref{eq:compact_3}, $\supp \varphi_i\sub \rho_n^{-1}\Psi_{x_0}^{-1}(\Omega_i)$  for $n$ sufficiently large. Hence, thanks to \Cref{lemma:int_by_parts} we have that
	\begin{equation*}
		\int_{B_1}(A(\rho_n x)\nabla \tilde{v}_i^{\rho_n}\cdot\nabla\varphi_i-\rho_n^2 p_i(\rho_n x)\tilde{v}_i^{\rho_n}\varphi_i)\dx=\int_{\partial B_1}\varphi_i A(\rho_n x)\nabla \tilde{v}_i^{\rho_n}\cdot\nnu\ds(x).
	\end{equation*}
	Hence, passing to the limit as $n\to\infty$, we obtain that
	\begin{equation*}
		\int_{B_1}\nabla U_i\cdot\nabla \varphi_i\dx=\int_{\partial B_1}\varphi_i h_i\ds,
	\end{equation*}
	for all $\varphi_i\in H^1_0(\mathcal{U}_i\cap \overline{B_1})$, which, after choosing $\varphi_i=U_i$ implies that
	\begin{equation}\label{eq:compact_5}
		\int_{B_1}|\nabla U_i|^2\dx=\int_{\partial B_1}U_i h_i\ds\quad\text{for all }i=1,\dots,N.
	\end{equation}
	At this point, in view of the properties of $A$ and $p_i$ stated in \Cref{lemma:A} and thanks to \Cref{lemma:int_by_parts}, \eqref{eq:compact_4}  and \eqref{eq:compact_5}, one can see that, as $n\to\infty$
	\begin{align*}
		\int_{B_1}|\nabla \tilde{v}_i^{\rho_n}|^2\dx &=\int_{B_1}(A(\rho_nx)\nabla \tilde{v}_i^{\rho_n}\cdot\nabla \tilde{v}_i^{\rho_n}-\rho_n^2 p_i(\rho_n x)|\tilde{v}_i^{\rho_n}|^2)\dx+o(1) \\
		&=\int_{\partial B_1}\tilde{v}_i^{\rho_n}A(\rho_n x)\nabla \tilde{v}_i^{\rho_n}\cdot\nnu\ds(x)+o(1) \\
		&=\int_{\partial B_1}U_i h_i\ds +o(1) =\int_{B_1}|\nabla U_i|^2\dx +o(1).
	\end{align*}
	Hence,
	\begin{equation*}
		\tilde{v}_i^{\rho_n}\to U_i\quad\text{strongly in }H^1(B_1)~\text{as }n\to\infty,
	\end{equation*}
	for all $i=1,\dots,N$. Now, by a standard procedure (essentially, passing to the limit as $n\to\infty$ in the Almgren functional), one can easily prove that $U=(U_1,\dots,U_N)$ is $\gamma(x_0)$-homogeneous. This, together with \Cref{prop:almost}, immediately implies that $U\in\mathcal{B}_{\gamma(x_0)}$. The last step is to prove strong $H^1$ convergence of $\tilde{v}^{r_n}$, up to subsequences (we recall that $r_n$ itself is a relabeled, suitably chosen subsequence). The proof is essentially contained in \cite[Proof of Lemma 6.5]{FFT2012} and we mainly refer to it. Basically, exploiting the previous step, we first prove that, up to subsequences
	\begin{equation*}
		\tilde{v}^{r_n}_i\to \bar{U}_i\quad\text{strongly in }H^1(B_1)~\text{as }n\to\infty,
	\end{equation*}
	for all $i=1,\dots,N$, for some $\bar{U}=(\bar{U}_1,\dots,\bar{U}_N)\in \mathcal{B}_{\gamma(x_0)}$ such that
	\[
	\sum_{i=1}^N\int_{\partial B_1}\bar{U}_i^2\ds=1.
	\]
	On the other hand
	\[
	\bar{U}(x)=\sqrt{\ell} U\left(\bar{\delta}^{-1}x\right),
	\]
	where
	\[
	\ell:=\lim_{n\to\infty}\frac{H(r_n\delta_n)}{H(r_n)}\quad\text{and}\quad \bar{\delta}:=\lim_{n\to\infty}\delta_n.
	\]
	Hence, from the normalization and scaling arguments (see \cite{FFT2012}), one can see that $U=\bar{U}$ and this concludes the proof of strong $H^1$ convergence. Finally, $C^{0,\alpha}$ convergence easily follows from \Cref{prop:lipschitz}. The proof is thereby complete.
\end{proof}

At this point, we observe that, combining \Cref{prop:almgren_rescalings} and \Cref{lemma:frequencies} we obtain the following.
\begin{corollary}
	One of the following holds true:
	\begin{itemize}
		\item $\gamma(x_0)=1$;
		\item $\gamma(x_0)=2$;
		\item $\gamma(x_0)\geq 2+\delta_d$,
	\end{itemize}
	where $\delta_d$ is as in \Cref{lemma:frequencies}.
\end{corollary}
Hence, it is natural to classify the points of $\partial D$ in terms of their frequencies. More precisely, for any $\gamma\ge 1$, we define
\begin{equation*}
	\mathcal{Z}_\gamma^{\partial D}(u):=\{x\in \partial D\colon \gamma(x)=\gamma\}.
\end{equation*}

Since in the final stages of the present paper it is convenient to work in the original domain $D$, we rewrite \Cref{prop:almgren_rescalings} in terms of the unperturbed minimizer $u$. In particular, for any $r>0$ and $x\in \frac{D-x_0}{r}$ we denote
\begin{equation*}
	\tilde{u}^{r,x_0}(x):=\frac{u(rx+x_0)}{\sqrt{H(v^{x_0},r)}}.
\end{equation*}
Then, thanks to \Cref{prop:almgren_rescalings}, \Cref{lemma:frequencies} and the  properties of the diffeomorphism $\Psi_{x_0}$, we have the following.

\begin{corollary}\label{cor:almgren_blowup}
	For any $x_0\in\mathcal{Z}_1^{\partial D}(u)\cup \mathcal{Z}_2^{\partial D}(u)$ and any $\{r_n\}_{n\in\N}$ such that $r_n\to 0$ as $n\to\infty$ there exists $U^{x_0}\in\mathcal{B}_{\gamma(x_0)}$ such that
	\[
	\sum_{i=1}^N \int_{\partial B_1}|U_i^{x_0}|^2=1
	\]
	and a subsequence $\{r_{n_k}\}_{k\in\N}$ such that
	\[
	\tilde{u}^{r_{n_k},x_0}\to U^{x_0}\quad\text{in }H^1(B_1,\R^N)~\text{and }C^{0,\alpha}(B_1)~\text{for all }\alpha\in(0,1), 
	\]
	as $k\to\infty$. Moreover, if we let
	\[
	\tilde{\kappa}_{d,1}:=\left(\int_{\partial B_1}x_d^+\ds\right)^{-\frac{1}{2}}\quad\text{and}\quad \tilde{\kappa}_{d,2}:=2\left(\int_{\partial B_1}x_{d-1}^+x_d^+\ds\right)^{-\frac{1}{2}},
	\]
	then the following holds. If $x_0\in\mathcal{Z}_1^{\partial D}(u)$, then there exists $j\in\{1,\dots,N\}$ such that
	\[
	U^{x_0}_j(x)=\tilde{\kappa}_{d,1}(-x\cdot\nnu(x_0))^+\quad\text{and}\quad U_i^{x_0}=0~\text{for all }i\neq j,
	\]
	while, if $x_0\in\mathcal{Z}_2^{\partial D}(u)$, then there exists $\bm{e}_{x_0}\in\partial B_1$ and $j,k\in\{1,\dots,N\}$, $j\neq k$, such that
	\[
	U^{x_0}_j=\tilde{\kappa}_{d,2}(x\cdot \bm{e}_{x_0})^-(-x\cdot\nnu(x_0))^+,\quad U^{x_0}_k(x)=\tilde{\kappa}_{d,2}(x\cdot \bm{e}_{x_0})^+(-x\cdot\nnu(x_0))^+
	\]
	and $U^{x_0}_i=0$ for all $i\neq j,k$.
\end{corollary}

\subsection{Monotonicity of the Weiss function}
As a consequence of the computations we made in the previous part of this section, we are able to deduce almost monotonicity of a Weiss-type functional, which will be crucial in the analysis of the free boundary regularity. We first introduce some notation. For $r>0$, $\gamma\geq 0$ and $w\in H^1_{s,N}(B_r)$, we let
\begin{align*}
	W_\gamma(w,r):&=\frac{H(w,r)}{r^{2\gamma}}\left[\mathcal{N}(w,r)-\gamma\right]=\frac{1}{r^{2\gamma}}\left[E(w,r)-\gamma H(w,r)\right] \\
	&=\frac{1}{r^{d+2\gamma-2}}\sum_{i=1}^N\int_{B_r}(A\nabla w_i\cdot\nabla w_i-p_iw_i^2)\dx-\frac{\gamma}{r^{d+2\gamma-1}}\sum_{i=1}^N\int_{\partial B_r}w_i^2\mu\ds,
\end{align*}
where $E$, $H$ and $\mathcal{N}$ are as in \Cref{sec:almgren}. Moreover, we introduce the \enquote{unperturbed} Weiss function
\begin{equation}\label{def:W_unpert}
	\widetilde{W}_\gamma(w,r):=\frac{1}{r^{d+2\gamma-2}}\sum_{i=1}^N\int_{B_r}\abs{\nabla w_i}^2\dx-\frac{\gamma}{r^{d+2\gamma-2}}\sum_{i=1}^N\int_{\partial B_r}w_i^2\ds.
\end{equation}
With a slight abuse of notation, we keep the same notation when dealing with a scalar $w\in H^1(B_r)$. 
\begin{proposition}[Monotonicity of the Weiss function]\label{prop:W'}
	Let $ \gamma \leq\gamma(x_0)$. Then, $W_\gamma(v,\cdot)\in W^{1,1}(0,R_0)$ and for all $r\leq R_0$ there holds
	\begin{multline}\label{eq:W'_th1}
		W_\gamma'(v,r)\geq (1-\kappa\sigma(r))\frac{d+2\gamma-2}{r}(\widetilde{W}_\gamma(h^{r,\gamma},1)-\widetilde{W}_\gamma(V^{r,\gamma},1)) \\
		+(1-\kappa\sigma(r))\frac{\mathcal{D}_{\gamma}(r)}{r}-CH(R_0)\frac{\sigma(r)}{r},
	\end{multline}
	where
	\begin{equation*}
		\begin{gathered}
			\mathcal{D}_{\gamma}(r):=\sum_{i=1}^N\int_{\partial B_1}(\nabla V_i^{r,\gamma}\cdot x-\gamma V_i^{r,\gamma})^2\dx, \\
			V^{r,\gamma}(x):=\frac{v(rx)}{r^{\gamma}},\qquad h^{r,\gamma}(x):=|x|^\gamma V^{r,\gamma}\left(\frac{x}{|x|}\right)
		\end{gathered}
	\end{equation*}
	and $C>0$ depends on $d$, $D$ and $N$. In addition,
	\begin{equation}\label{eq:W'_th2}
		W_\gamma'(v,r)\geq \sum_{i=1}^N\frac{2}{r^{d+2\gamma}}\int_{S_r}\left(\frac{1}{\sqrt{\mu}}A\nabla v_i\cdot x-\gamma\sqrt{\mu} v_i\right)^2\ds-C_WH(R_0)\frac{\sigma(r)}{r}
	\end{equation}
	for all $r\leq R_0$, where $C_W>0$ depends on $d$, $D$ and $N$.
\end{proposition}
\begin{proof}
	First of all, we observe that, since $H\in C^1(0,R_0)$ and $\mathcal{N}\in W^{1,1}(0,R_0)$, then $W_\gamma(v,\cdot)\in W^{1,1}_{\textup{loc}}(0,R_0)$. Thanks to \eqref{eq:H'_th2} and \eqref{eq:E'1} we have that
	\begin{align*}
		W_\gamma'(v,r)\geq\sum_{i=1}^N\Bigg[&-\frac{d+2\gamma-2}{r^{d+2\gamma-1}}\int_{B_r}(A\nabla v_i\cdot\nabla v_i-p_iv_i^2)\dx+\frac{2\gamma^2}{r^{d+2\gamma}}\int_{\partial B_r}v_i^2\mu\ds \\
		&+\frac{1}{r^{d+2\gamma-2}}\int_{\partial B_r}(A\nabla v_i\cdot\nabla v_i-p_iv_i^2)\ds-\frac{2\gamma }{r^{d+2\gamma-1}}\int_{\partial B_r}v_iA\nabla v_i\cdot\nnu\ds\Bigg] \\
		- C&\frac{H(r)}{r^{2\gamma}}\frac{\sigma(r)}{r}.
	\end{align*}
	In particular, thanks to \Cref{lemma:decay}, taking into account that $\gamma\leq \gamma(x_0)$, we have
	\begin{align*}
		W_\gamma'(v,r)\geq \sum_{i=1}^N\Bigg[&-\frac{d+2\gamma-2}{r^{d+2\gamma-1}}\int_{B_r}(A\nabla v_i\cdot\nabla v_i-p_iv_i^2)\dx+\frac{2\gamma^2}{r^{d+2\gamma}}\int_{\partial B_r}v_i^2\mu\ds \\
		&+\frac{1}{r^{d+2\gamma-2}}\int_{\partial B_r}(A\nabla v_i\cdot\nabla v_i-p_iv_i^2)\ds-\frac{2\gamma}{r^{d+2\gamma-1}}\int_{\partial B_r}v_iA\nabla v_i\cdot\nnu\ds\Bigg] \\
		&- CH(R_0)\frac{\sigma(r)}{r},
	\end{align*}
	where $C>0$ depends $d$, $D$ and $N$. Now, we observe that, in view of \Cref{lemma:A} there holds
	\[
	\int_{\partial B_r}A\nabla v_i\cdot\nabla v_i\ds\geq (1-\kappa\sigma(r))\int_{\partial B_r}|\nabla v_i|^2\ds
	\]
	and, thanks to Young's inequality, there holds
	\begin{align*}
		\int_{\partial B_r}v_iA\nabla v_i\cdot\nnu\ds&=\int_{\partial B_r}v_i\nabla v_i\cdot\nnu\ds+\int_{\partial B_r}v_i(A-I)\nabla v_i\cdot\nnu\ds \\
		&\leq \int_{\partial B_r}v_i\nabla v_i\cdot\nnu\ds+\kappa\sigma(r)\left(\int_{\partial B_r}v_i^2\ds+\int_{\partial B_r}|\nabla v_i|^2\ds\right)
	\end{align*}
	for all $r\leq R_0$ and $i=1,\dots,N$. Therefore, up to a change of variable and up to renaming the constant $\kappa$, we have that 
	\begin{equation}\label{eq:W'_1}
		\begin{aligned}
			W_\gamma'(v,r)\geq \frac{1}{r}\sum_{i=1}^N\Bigg[&-(d+2\gamma-2)\int_{B_1}|\nabla V_i^{r,\gamma}|^2\dx+(1-\kappa\sigma(r))\int_{\partial B_1}|\nabla V_i^{r,\gamma}|^2\ds \\
			&+2\gamma^2(1-\kappa\sigma(r))\int_{\partial B_1}|V_i^{r,\gamma}|^2\ds-2\gamma\int_{\partial B_1}V_i^{r,\gamma}\nabla V_i^{r,\gamma}\cdot x\ds 	\Bigg]\\
			+B(r)&-CH(R_0)\frac{\sigma(r)}{r},
		\end{aligned}
	\end{equation}
	where
	\begin{align*}
		B(r):=\frac{H(r)}{r^{2\gamma}}\sum_{i=1}^N\Bigg[&-\frac{d+2\gamma-2}{r}\int_{B_1}((A(rx)-I)\nabla \tilde{v}_i^r\cdot\nabla \tilde{v}_i^r-r^2p_i(rx)\abs{\tilde{v}_i^r}^2)\dx \\
		& +\frac{2\gamma^2}{r}\int_{\partial B_1}|\tilde{v}_i^r|^2(\mu(rx)-1)\ds -r\int_{\partial B_1}p_i(rx)|\tilde{v}_i^r|^2\ds\Bigg]
	\end{align*}
	and $\tilde{v}_i^r$ is as in \eqref{def:tilde_v_r}.	Thanks to \Cref{lemma:decay}, \Cref{lemma:A}, \Cref{lemma:int_by_parts} and \Cref{cor:N_bound} we have that
	\begin{equation}\label{eq:W'_5}
		|B(r)|\leq CH(R_0)\frac{\sigma(r)}{r},
	\end{equation}
	for some constant $C>0$ depending on $d$, $D$ and $N$. Moreover, reasoning analogously and taking into account \eqref{eq:W'_5}, from \eqref{eq:W'_1} we obtain that
	\begin{equation}\label{eq:W'_4}
		\begin{aligned}
			W_\gamma'(v,r)\geq \frac{1-\kappa\sigma(r)}{r}\sum_{i=1}^N\Bigg[&-(d+2\gamma-2)\int_{B_1}|\nabla V_i^{r,\gamma}|^2\dx \\
			&+\int_{\partial B_1}|\nabla V_i^{r,\gamma}|^2\ds +2\gamma^2\int_{\partial B_1}|V_i^{r,\gamma}|^2\ds \\
			&-2\gamma\int_{\partial B_1}V_i^{r,\gamma}\nabla V_i^{r,\gamma}\cdot x\ds 	\Bigg]-CH(R_0)\frac{\sigma(r)}{r},
		\end{aligned}
	\end{equation}
	Now, thanks to homogeneity properties, one can easily see that
	\begin{equation}\label{eq:W'_2}
		-\sum_{i=1}^N\int_{B_1}|\nabla V_i^{r,\gamma}|^2\dx=\widetilde{W}(h^{r,\gamma},1)-\widetilde{W}(V^{r,\gamma},1)-\frac{1}{d+2\gamma-2}\sum_{i=1}^N\int_{\partial B_1}|\nabla h_i^{r,\gamma}|^2\dx
	\end{equation}
	and that
	\begin{equation}\label{eq:W'_3}
		|\nabla h_i^{r,\gamma}|^2=\gamma^2|V_i^{r,\gamma}|^2+|\nabla V_i^{r,\gamma}|^2-|\nabla V_i^{r,\gamma}\cdot x|^2\quad\text{on }\partial B_1.
	\end{equation}
	At this point, plugging \eqref{eq:W'_2} and \eqref{eq:W'_3} into \eqref{eq:W'_4}, we obtain \eqref{eq:W'_th1}. We are left to prove \eqref{eq:W'_th2}. In order to do this, we compute
	\begin{equation*}
		W_\gamma'(v,r)=\frac{1}{r^{2\gamma}}\left[E'(v,r)-\gamma H'(v,r)\right]-\frac{2\gamma}{r^{2\gamma+1}}\left[E(v,r)-\gamma H(v,r)\right].
	\end{equation*}
	Now, we combine this expression with the computations we previously made for the Almgren frequency function. In particular, thanks to \Cref{lemma:E'}, \eqref{eq:H'_th2} and \Cref{lemma:int_by_parts} we obtain
	\begin{align*}
		W_\gamma'(v,r)&\geq \sum_{i=1}^N\Bigg\{\frac{1}{r^{2\gamma}}\left[\frac{2}{r^{d-2}}\int_{S_r}\frac{1}{\mu}(A\nabla v_i\cdot\nnu)^2\ds-\frac{2\gamma}{r^{d-1}}\int_{S_r}v_i A\nabla v_i\cdot\nnu\ds\right] \\
		&\quad -\frac{2\gamma}{r^{2\gamma+1}}\left[ \frac{1}{r^{d-2}}\int_{S_r}v_i A\nabla v_i\cdot\nnu-\frac{\gamma}{r^{d-1}}\int_{S_r}v_i^2\mu\ds\right]\Bigg\} \\
		&\quad-C\frac{H(v,r)}{r^{2\gamma}}\frac{\sigma(r)}{r}-C\frac{H(v,r)}{r^{2\gamma}}\frac{\sigma(r)}{r}(\mathcal{N}(r)+1).
	\end{align*}
	Rearranging the terms and applying \Cref{lemma:decay} and \Cref{cor:N_bound} to the reminder yields \eqref{eq:W'_th2}. Finally, since $W_\gamma(v,\cdot)\in W^{1,1}_{\textup{loc}}(0,R_0)$ and $W_\gamma(v,\cdot)$ admits a finite limit at $0$, we have that $W_\gamma(v,\cdot)\in W^{1,1}(0,R_0)$ and this concludes the proof.
\end{proof}

\section{Epiperimetric inequalities at points of low frequency}\label{sec:epi}

The aim of the present section is to prove an epiperimetric inequality for the Weiss energy corresponding to the low isolated frequencies of free boundary points. In particular, we are able to obtain it for both the lowest frequency of interior free boundary points and for the lowest isolated frequencies of those which are located on the boundary $\partial D$, see \Cref{thm:epi1} and \Cref{thm:epi2}, respectively. By definition, proving an epiperimetric inequality entails the construction of a competitor function which lowers the Weiss energy of the corresponding homogeneous extension by a universal multiplicative factor. We adopt a direct approach, which is devoted to build an explicit competitor and, in order to do this, we find useful to introduce the following operators. The first one is the harmonic extension of a function defined on the unit sphere. We recall that $\{\phi_n\}_n$ denotes a sequence of orthonormal eigenfunctions of the Laplacian on $\partial B_1$.
\begin{definition}[Harmonic extension]\label{def:harm_ext}
	For any $f\in H^1(\partial B_1)$ we denote by $\mathscr{H}(f)\in H^1(B_1)$ the unique function solving
	\begin{equation*}
		\begin{bvp}
			-\Delta \varphi &= 0 &&\text{in }B_1, \\
			\varphi&= f, &&\text{on }\partial B_1.
		\end{bvp}
	\end{equation*}
	We call $\mathscr{H}(f)$ the \emph{harmonic extension} of $f$ in $B_1$. In particular, if
	\begin{equation*}
		a_j:=\int_{\partial B_1}f\phi_j,\quad j\geq 0,
	\end{equation*}
	then we have
	\begin{equation*}
		\mathscr{H}(f)(r,\theta)=\sum_{j=0}^\infty a_j r^j\phi_j(\theta).
	\end{equation*}
\end{definition}
Second, we introduce the homogeneous extension operator.
\begin{definition}[Homogeneous extension]\label{def:homo_ext}
	Let $\gamma\geq 0$. For any $f\in H^1(\partial B_1)$ we denote
	\begin{equation*}
		Z_\gamma(f)(x):=|x|^\gamma f\left( \frac{x}{|x|} \right) 
	\end{equation*}
	and we call it the \emph{$\gamma$-homogeneous extension} of $f$.
\end{definition}
Third, we need the truncation operator, which homogeneously extends a function of the unit sphere, up to a certain radius.
\begin{definition}[Truncation]\label{def:truncation}
	Let $\rho\in (0,1)$ and $\tau>0$. For any $f\in H^1(\partial B_1)$ we denote
	\begin{equation*}
		T_{\rho,\tau}(f)(x):=\frac{|(|x|-\rho)^+|^\tau}{(1-\rho)^\tau}\,f\left(\frac{x}{|x|}\right)
	\end{equation*}
	and we call it the \emph{$(\rho,\tau)$ homogeneous truncation} of $f$.
\end{definition}

Finally, we introduce the rescaling operator, which shrinks a function of the unit ball into a smaller ball and fills the annulus with the homogeneous extension.	
\begin{definition}[Rescaling]\label{def:scaled}
	Let $\gamma\geq 0$, $\rho\in(0,1)$ and $Z_\gamma$ be as in \Cref{def:homo_ext}. For any $f\in H^1(\partial B_1)$ and any $w\in H^1(B_1)$ such that $w=f$ on $\partial B_1$, we denote
	\begin{equation*}
		R_{\gamma,\rho} (w)(x):=\begin{cases}
			Z_\gamma(f)(x), &\text{in }B_1\setminus B_\rho, \\
			|x|^\gamma w\left( \frac{x}{\rho} \right), &\text{in }B_\rho.
		\end{cases}
	\end{equation*}
\end{definition}

The first main tool for the proof of the epiperimetric inequality is a slicing lemma for the Weiss energy. We recall that $\widetilde{W}_\gamma(w)=\widetilde{W}_\gamma(w,1)$ is as in \eqref{def:W_unpert}. 
\begin{lemma}[Slicing Lemma]\label{lemma:slicing}
	Let $\gamma\geq 0$, let $w\in H^1(B_1)$ and let
	\[
	\varphi_r(\theta):=r^{-\gamma}w(r,\theta).
	\]
	Then, we have that
	\begin{equation*}
		\widetilde{W}_\gamma(w)=\int_0^1r^{d+2\gamma-3}\mathcal{F}_\gamma(\varphi_r)\d r+\int_0^1r^{d+2\gamma-1}\int_{\partial B_1}|\partial_r\varphi_r|^2\ds\d r,
	\end{equation*}
	where $\mathcal{F}_\gamma\colon H^1(\partial B_1)\to\R$ is defined as
	\begin{equation*}
		\mathcal{F}_\gamma(\varphi):=\int_{\partial B_1}(|\nabla_{\partial B_1}\varphi|^2-\gamma(d+\gamma-2)\varphi^2)\ds.
	\end{equation*}	
\end{lemma}
\begin{proof}
	By easy computations in polar coordinates and by definition of $\varphi_r$, one can see that
	\begin{align*}
		\int_{B_1}\abs{\nabla w}^2\dx&=\int_0^1r^{d-1}\int_{\partial B_1}(|\partial_r(r^\gamma\varphi_r)|^2+r^{2\gamma-2}\abs{\nabla_{\partial B_1}\varphi_r}^2)\ds\d r \\
		&=\int_0^1r^{d-1}\int_{\partial B_1}(\gamma^2r^{2\gamma-2}\varphi_r^2+r^{2\gamma}(\partial_r\varphi_r)^2 \\
		&\qquad\qquad\qquad\quad+\gamma r^{2\gamma-1}\partial_r(\varphi_r^2)+r^{2\gamma-2}\abs{\nabla_{\partial B_1}\varphi_r}^2)\ds\d r.
	\end{align*}
	Moreover, we have that
	\begin{equation*}
		\int_0^1r^{d+2\gamma-2}\int_{\partial B_1}\partial_r(\varphi_r^2)\ds\d r=\gamma\int_{\partial B_1}w^2\ds-\gamma(d+2\gamma-2)\int_0^1r^{d+2\gamma-3}\int_{\partial B_1}\varphi_r^2\ds\d r.
	\end{equation*}
	Combining these two identities with the definition of $\widetilde{W}_\gamma$ and rearranging the terms conclude the proof.
\end{proof}

We now compute the gain in term of Weiss energy when comparing a homogeneous function with its harmonic extension.
\begin{lemma}[Energy gain of harmonic extensions]\label{lemma:gain_harmonic}
	Let $\gamma\geq 0$, $f\in H^1(\partial B_1)$ and let
	\begin{equation*}
		a_j:=\int_{\partial B_1}f\phi_j\ds,\quad j\geq 0.
	\end{equation*}
	If $H$ and $Z_\gamma$ are as in \Cref{def:harm_ext} and \Cref{def:homo_ext}, we have that
	\begin{equation}\label{eq:epi_harm_th1}
		\widetilde{W}_\gamma(\mathscr{H}(f))-(1-\epsilon)\widetilde{W}_\gamma (Z_\gamma(f))=\sum_{j=0}^\infty a_j^2\frac{j-\gamma}{(d+2\gamma-2)(d+\gamma+j-2)}\left(\epsilon-\frac{j-\gamma}{d+\gamma+j-2}\right)
	\end{equation}
	for any $\epsilon>0$. In particular, if
	\begin{equation}\label{eq:epi_harm_hp1}
		\epsilon_1=\epsilon_1(d,\gamma):=\frac{\lfloor\gamma+1\rfloor-\gamma}{d+2\gamma-1}
	\end{equation}
	then
	\begin{equation}\label{eq:epi_harm_th2}
		\widetilde{W}_\gamma(\mathscr{H}(f))\leq (1-\epsilon_1) \widetilde{W}_\gamma(Z_\gamma(f)).
	\end{equation}
\end{lemma}
\begin{proof}
	We apply \Cref{lemma:slicing} with $w=Z_\gamma(f)$ and $w=\mathscr{H}(f)$. In the former case, we have that
	\begin{equation*}
		\varphi_r(\theta)=\sum_{j=0}^\infty a_j\phi_j(\theta).
	\end{equation*}
	Hence, since
	\begin{equation*}
		\int_{\partial B_1}\nabla_{\partial B_1}\phi_j\cdot\nabla_{\partial B_1}\phi_k\ds=j(d+j-2)\delta_{jk}\quad\text{and}\quad \int_{\partial B_1}\phi_j\phi_k\ds=\delta_{jk},
	\end{equation*}
	we have that 
	\begin{equation*}
		\mathcal{F}_\gamma(\varphi_r)=\sum_{j=0}^\infty a_j^2 [j(d+j-2)-\gamma(d+\gamma-2)].
	\end{equation*}
	Therefore,
	\begin{equation}\label{eq:epi_harm_1}
		\widetilde{W}_\gamma(Z_\gamma(f))=\sum_{j=0}^\infty a_j^2\frac{j(d+j-2)-\gamma(d+\gamma-2)}{d+2\gamma-2}=\sum_{j=0}^\infty a_j^2\left[\frac{j(d+j-2)+\gamma^2}{d+2\gamma-2}-\gamma\right].
	\end{equation}
	On the other hand, if $w=\mathscr{H}(f)$, i.e. 
	\begin{equation*}
		\varphi_r(\theta)=\sum_{j=0}^\infty a_jr^{j-\gamma}\phi_j(\theta),
	\end{equation*}
	then
	\begin{equation*}
		\mathcal{F}_\gamma(\varphi_r)=\sum_{j=0}^\infty a_j^2 r^{2(j-\gamma)}[j(d+j-2)-\gamma(d+\gamma-2)]
	\end{equation*}
	and
	\begin{equation*}
		\int_{\partial B_1}|\partial_r\varphi_r|^2\ds=\sum_{j=0}^\infty a_j^2r^{2(j-\gamma-1)}(j-\gamma)^2.
	\end{equation*}
	Given these computations, we can apply \Cref{lemma:slicing} and obtain
	\begin{equation}\label{eq:epi_harm_2}
		\widetilde{W}_\gamma(\mathscr{H}(f))=\sum_{j=0}^\infty a_j^2\left[ \frac{j(d+j-2)-\gamma(d+\gamma-2)}{d+2j-2}+\frac{(j-\gamma)^2}{d+2j-2}\right]=\sum_{j=0}^\infty a_j^2(j-\gamma).
	\end{equation}
	Now, combining \eqref{eq:epi_harm_1} and \eqref{eq:epi_harm_2} we obtain \eqref{eq:epi_harm_th1}. Finally, by a direct study of the monotonicity of the function
	\begin{equation*}
		j\mapsto \frac{j-\gamma}{d+\gamma+j-2}
	\end{equation*}
	we derive \eqref{eq:epi_harm_th2} by choosing $\epsilon_1$ as in \eqref{eq:epi_harm_hp1} and this completes the proof.
\end{proof}

When the Rayleigh quotient of a homogeneous function is sufficiently high, we can quantitatively lower its Weiss energy just by truncation, and this what the following result contains.	
\begin{lemma}[Improvement of high modes]\label{lemma:high}
	Let $\gamma\geq 1$,  $f\in H^1(\partial B_1)$ and $\ell>0$ be such that
	\begin{equation}\label{eq:high_hp1}
		(\gamma(d+\gamma-2)+\ell)\int_{\partial B_1}f^2\ds\leq \int_{\partial B_1}|\nabla_{\partial B_1}f|^2\ds
	\end{equation}
	and $\mathcal{F}_\gamma(f)>0$, where $\mathcal{F}_\gamma$ is as in \Cref{lemma:slicing}. Then, there exists $\epsilon_2=\epsilon_2(d,\gamma,\ell)>0$ such that
	\begin{equation}\label{eq:high_th1}
		\widetilde{W}_\gamma(T_{\rho,\gamma+a}(f))\leq (1-\epsilon_2)\widetilde{W}_\gamma(Z_\gamma(f)),
	\end{equation}
	where $T_{\rho,\gamma+a}$ and $Z_\gamma$ are as in \Cref{def:truncation} and \Cref{def:homo_ext}, respectively,  and $\rho,a\in (0,1/2)$ depend only on $d$, $\gamma$ and $\ell$. Moreover, $\epsilon_2$ depends continuously on $\gamma$.
\end{lemma}
\begin{proof}
	We want to apply \Cref{lemma:high} with 
	\[
	w(r,\theta)=T_{\rho,\gamma+a}(f)(r,\theta)=\displaystyle\frac{|(r-\rho)^+|^{\gamma+a}}{(1-\rho)^{\gamma+a}} f(\theta)\quad\text{and}\quad \varphi_r(\theta)=\displaystyle\frac{|(r-\rho)^+|^{\gamma+a}}{(1-\rho)^{\gamma+a}}  \frac{f(\theta)}{r^\gamma}.
	\]
	On one hand, thanks to the fact that
	\begin{equation}\label{eq:high_1}
		\mathcal{F}_\gamma(f)=(d+2\gamma-2)W_\gamma(Z_\gamma(f)),
	\end{equation}
	and being $\gamma+a\geq 0$, we have that
	\begin{multline}\label{eq:high_2}
		\mathcal{F}_\gamma(\varphi_r)=\displaystyle\frac{|(r-\rho)^+|^{2(\gamma+a)}}{(1-\rho)^{2(\gamma+a)}} \frac{\mathcal{F}_\gamma(f)}{r^{2\gamma}} \\
		=(d+2\gamma-2)\displaystyle\frac{|(r-\rho)^+|^{2(\gamma+a)}}{(1-\rho)^{2(\gamma+a)}} \frac{\widetilde{W}_\gamma(Z_\gamma(f))}{r^{2\gamma}}\leq \displaystyle\frac{(d+2\gamma-2)r^{2a}}{(1-\rho)^{2(\gamma+a)}} \widetilde{W}_\gamma(Z_\gamma(f)).
	\end{multline}
	On the other hand, for $r>\rho$ one can easily compute
	\begin{equation*}
		\partial_r\varphi_r(\theta)=\frac{(r-\rho)^{\gamma+a-1}}{r^\gamma(1-\rho)^{\gamma+a}}\left( a+\frac{\gamma\rho}{r}\right).
	\end{equation*}
	Combining this fact with \eqref{eq:high_hp1} and \eqref{eq:high_1} (and using that $\gamma+a\geq 1$), we obtain that
	\begin{align}
		\int_{\partial B_1}|\partial_r\varphi_r|^2\ds&=\displaystyle\frac{(r-\rho)^{2(\gamma+a-1)}}{r^{2\gamma}(1-\rho)^{2(\gamma+a)}}\left(a+\frac{\gamma\rho}{r}\right)^2\int_{\partial B_1}f^2\ds\notag\\
		&\leq  \displaystyle\frac{(r-\rho)^{2(\gamma+a-1)}}{r^{2\gamma}(1-\rho)^{2(\gamma+a)}}\left(a+\frac{\gamma\rho}{r}\right)^2\frac{\mathcal{F}_\gamma(f)}{\ell} \notag \\
		&\leq \frac{2r^{2(a-1)}}{(1-\rho)^{2(\gamma+a)}}\left(a^2+\frac{\gamma^2\rho^2}{r^2}\right)\frac{(d+2\gamma-2)\widetilde{W}_\gamma(Z_\gamma(f))}{ \ell} . \label{eq:high_3}
	\end{align}
	Now, since $\rho,a\in(0,1/2)$, we have that
	\[
	\frac{1}{(1-\rho)^{2(\gamma+a)}}\leq 1+2(2^{2(\gamma+a)}-1)\rho
	\]
	from which we derive that
	\begin{equation}\label{eq:high_4}
		\frac{1}{(1-\rho)^{2(\gamma+a)}}\leq 1+2^{2\gamma+2}\rho\qquad\text{and}\qquad \frac{1}{(1-\rho)^{2(\gamma+a)}}\leq 2^{2\gamma+1}.
	\end{equation}
	Now, plugging the estimates from \eqref{eq:high_4} into \eqref{eq:high_2} and \eqref{eq:high_3}, respectively, and applying \Cref{lemma:slicing}, we obtain that
	\begin{align*}
		\widetilde{W}_\gamma(T_{\rho,\gamma+a}(f))&=\int_0^1r^{d+2\gamma-3}\mathcal{F}_\gamma(\varphi_r)\d r+\int_0^1r^{d+2\gamma-1}\int_{\partial B_1}|\partial_r\varphi_r|^2\ds\d r \\
		&\begin{aligned}
			\leq (d+2\gamma-2)\widetilde{W}_\gamma(Z_\gamma(f))\Bigg[& (1+2^{2\gamma+2}\rho)\int_0^1r^{d+2\gamma+2a-3}\d r \\
			&+\frac{2^{2\gamma+2}}{\ell }\int_0^2r^{d+2\gamma +2a-3}\left(a^2+\frac{\gamma^2\rho^2}{r^2}\right)\d r  \Bigg]
		\end{aligned} \\
		&\begin{aligned}
			=(d+2\gamma-2)\widetilde{W}_\gamma(Z_\gamma(f))\Bigg[ & \frac{1+2^{2\gamma+2}\rho}{d+2\gamma+2a-2}+\frac{2^{2\gamma+2}}{\ell}\Bigg(\frac{a^2}{d+2\gamma+2a-2}  \\
			&+\frac{\gamma^2\rho^2}{d+2\gamma+2a-4}\Bigg)  \Bigg].    
		\end{aligned}
	\end{align*}
	At this point, for the first term we use that
	\begin{equation*}
		\frac{(d+2\gamma-2)(1+2^{2\gamma+2}\rho)}{d+2\gamma+2a-2}\leq 1-\frac{2a}{d+2\gamma-1}+2^{2\gamma+2}\rho,
	\end{equation*}
	while for the second
	\begin{equation*}
		\frac{a^2}{d+2\gamma+2a-2}+\frac{\gamma^2\rho^2}{d+2\gamma+2a-4}\leq \frac{1}{2}\left(a^2+\frac{\gamma^2\rho^2}{a}\right).
	\end{equation*}
	Therefore
	\begin{equation*}
		\widetilde{W}_\gamma(T_{\rho,\gamma+a}(f))\leq \widetilde{W}_\gamma(Z_\gamma(f))\left[ 1-\frac{2a}{d+2\gamma-1}+2^{2\gamma+2}\rho+\frac{(d+2\gamma-2)2^{2\gamma+1}}{\ell}\left(a^2+\frac{\gamma^2\rho^2}{a}\right) \right].
	\end{equation*}
	Hence, if we let $\rho=a^{3/2}$ and use that $a^2\leq a^{3/2}$, we derive that
	\begin{equation*}
		\widetilde{W}_\gamma(T_{\rho,\gamma+a}(f))\leq \widetilde{W}_\gamma(Z_\gamma(f))\left[ 1-\frac{2a}{d+2\gamma-1}+2^{2\gamma+1}a^{3/2}\left(2+\frac{(d+2\gamma-2)(1+\gamma^2)}{\ell}\right) \right].
	\end{equation*}
	Finally, we choose $a\leq 1/2$ in such a way that
	\[
	2^{2\gamma+1}a^{3/2}\left(2+\frac{(d+2\gamma-2)(1+\gamma^2)}{\ell}\right)\leq \frac{a}{d+2\gamma-1}
	\]
	and we obtain \eqref{eq:high_th1} with
	\[
	\epsilon_2=\frac{a}{d+2\gamma-1},
	\]
	thus concluding the proof.
\end{proof}

The following results compare the energy gain of a scaled function with respect to the original one.
\begin{lemma}[Scaling]\label{lemma:scaling}
	Let $\gamma\geq 0$, $\rho\in(0,1)$, $f\in H^1(\partial B_1)$ and $w\in H^1(B_1)$ such that $w=f$ on $\partial B_1$. If $R_{\gamma,\rho}$ is as in \Cref{def:scaled} and $Z_\gamma$ is as in \Cref{def:homo_ext}, then
	\begin{equation*}
		\widetilde{W}_\gamma(R_{\gamma,\rho}(w))-\widetilde{W}_\gamma(Z_\gamma(f))=\rho^{d+2\gamma-2}\left(\widetilde{W}_\gamma(w)-\widetilde{W}_\gamma(Z_\gamma(f))\right).
	\end{equation*}
\end{lemma}
\begin{proof}
	The proof follows from a simple change of variables. Indeed, one can easily see that
	\begin{align*}
		\int_{B_1}|\nabla R_{\gamma,\rho}(w)|^2\dx&=\int_{B_1}|\nabla Z_\gamma(f)|^2\dx-\int_{B_\rho}|\nabla Z_\gamma(f)|^2\dx+\int_{B_\rho}|\nabla(|x|^\gamma w(x/\rho)|^2\dx \notag \\
		&=\int_{B_1}|\nabla Z_\gamma(f)|^2\dx-\rho^{d+2\gamma-2}\int_{B_1}|\nabla Z_\gamma(f)|^2\dx+\rho^{d+2\gamma-2}\int_{B_1}|\nabla w|^2\dx,
	\end{align*}
	which directly imply the thesis.
\end{proof}

We now have all the ingredients needed to prove the epiperimetric inequality for segregated functions at points of low frequency. 		We start by proving the epiperimetric inequality on the half-ball with homogeneous Dirichlet boundary conditions on the lower part of the boundary and for Weiss energy with homogeneity between $1$ and $2$. 

\begin{theorem}[Epiperimetric inequality at boundary points]\label{thm:epi2}
	Let $\gamma\in[1,2]$. There exists $\epsilon_{\textup{bd}}=\epsilon_{\textup{bd}}(d)>0$ such that, for all $\gamma$-homogeneous $z\in H^1_{s,N}(B_1^+)$ such that $z=0$ on $B_1'$ there exists $w\in H^1_{s,N}(B_1^+)$ satisfying
	\begin{equation*}
		w=z~\text{on }\partial B_1^+\quad\text{and}\quad \widetilde{W}_\gamma(w)\leq (1-\epsilon_{\textup{bd}})\widetilde{W}_\gamma(z).
	\end{equation*}
\end{theorem}
\begin{proof}
	Let $q_d$ be as in \eqref{eq:N_4},
	\begin{equation*}
		\ell_0:=\frac{1}{2}(1-q_d^2)(d+3)
	\end{equation*}
	and let
	\begin{equation*}
		\epsilon_{\textup{bd}}:=\min_{\gamma\in[1,2]}\min\{\epsilon_1(d,\gamma)\rho^d,\epsilon_2(d,\gamma,\ell_0)\},
	\end{equation*}
	where $\epsilon_1$  and $\rho$ are as in \Cref{lemma:gain_harmonic} and $\epsilon_2$ as in \Cref{lemma:high}. We claim that there exists at most two functions $f_1,f_2\in H^1(S_1^+)$ such that 
	\begin{equation*}
		f_1f_2\equiv 0\quad\text{and}\quad f_1=f_2=0~\text{on }\partial S_1^+
	\end{equation*}
	satisfying 
	\begin{equation*}
		(\gamma(d+\gamma-2)+\ell_0)\int_{S_1^+}f_i^2\ds\geq \int_{S_1^+}|\nabla_{\partial B_1}f_i|^2\ds\quad i=1,2.
	\end{equation*}
	Indeed, let us assume by contradiction that there exists $f_1,f_2,f_3\in H^1_0(S_1^+)$ with disjoint support such that
	\begin{equation*}
		(\gamma(d+\gamma-2)+\ell_0)\int_{S_1^+}f_i^2\ds\geq  \int_{S_1^+}|\nabla_{\partial B_1}f_i|^2\ds\quad i=1,2,3.
	\end{equation*}
	Since $\gamma\leq 2$, in particular we have
	\begin{equation}\label{eq:epi2_1}
		(2d+\ell_0)\int_{S_1^+}f_i^2\ds\geq  \int_{S_1^+}|\nabla_{\partial B_1}f_i|^2\ds\quad i=1,2,3.
	\end{equation}
	We proceed analogously to the proof of $(iii)$ in \Cref{lemma:frequencies}. It is easy to check that there exists $i,j\in\{1,2,3\}$, $i\neq j$ such that
	\begin{equation*}
		\mathcal{H}^{d-1}(\{f_i>0\})+\mathcal{H}^{d-1}(\{f_j>0\})\leq\frac{2}{3}\mathcal{H}^{d-1}(S_1^+).
	\end{equation*}
	Without loss of generality we assume $i=1$ and $j=2$. We define
	\begin{equation*}
		\title{c}_1:=\frac{f_1}{\norm{f_1}_{L^2(S_1^+)}},\quad c_2:=\frac{f_2}{\norm{f_2}_{L^2(S_1^+)}}
	\end{equation*}
	and
	\begin{equation*}
		\tilde{c}:=\frac{c_1-tc_2}{\norm{c_1-tc_2}_{L^2(S_1^+)}},\quad\text{where}\quad t:=\frac{\displaystyle\int_{S_1^+}c_1\phi_1\ds}{\displaystyle\int_{S_1^+}c_2\phi_1\ds}.
	\end{equation*}
	Reasoning as in the proof of $(iii)$ in \Cref{lemma:frequencies}, one can prove that
	\begin{equation}\label{eq:epi2_2}
		\int_{S_1^+}|\nabla_{\partial B_1} \tilde{c}|^2\ds\geq 2d+2\ell_0.
	\end{equation}
	On the other hand, since
	\begin{equation*}
		\int_{S_1^+}|\nabla_{\partial B_1} \tilde{c}|^2\ds=\frac{1}{1+t^2}\int_{S_1^+}|\nabla_{\partial B_1}{c}_1|^2\ds+\frac{t^2}{1+t^2}\int_{S_1^+}|\nabla_{\partial B_1}{c}_2|^2\ds,
	\end{equation*} 
	combining \eqref{eq:epi2_1} and \eqref{eq:epi2_2} we find a contradiction.  Then, the claim is proved. Hence, given a $2$-homogeneous function $z\in H^1_{s,N}(B_1^+)$ such that $z=0$ on $B_1'$ and letting $f=z\restr{S_1^+}$, without loss of generality, we can assume that
	\begin{equation}
		(2d+\ell_0)\int_{S_1^+}f_i^2\ds\leq \int_{S_1^+}|\nabla_{\partial B_1}f_i|^2\ds\quad\text{for all }i=3,\dots,N.
	\end{equation}
	We now define the competitor $w\in H^1_{s,N}(B_1^+)$. If $\mathcal{F}_\gamma(f)\leq 0$ we just define let $w=z$, hence let us hereafter assume $\mathcal{F}_\gamma(f)>0$. For $i=3,\dots,N$, we let
	\begin{equation*}
		w_i:=T_{\rho,2+a}(f_i),
	\end{equation*}
	where $\rho$, $a$ are as in \Cref{lemma:high} (with $\gamma=2$) while $T_{\rho,2+a}$ is as in \Cref{def:truncation}. Moreover, we let 
	\begin{equation*}
		w_1=R_{2,\rho}(\mathscr{H}(f_1-f_2)_+)\quad\text{and}\quad w_2=R_{2,\rho}(\mathscr{H}(f_1-f_2)_-),
	\end{equation*}
	where $H$ is as in \Cref{def:harm_ext} and $R_{2,\rho}$ as in \Cref{def:scaled}. One can easily see that $w\in H^1_{s,N}(B_1^+)$ and $w=z$ on $\partial B_1^+$. Now, in view of \Cref{lemma:scaling}, \Cref{lemma:gain_harmonic} and \Cref{lemma:high} we have that
	\begin{align*}
		\widetilde{W}_\gamma(w)-\widetilde{W}_\gamma(z)&=\rho^d(\widetilde{W}_\gamma(\mathscr{H}(f_1-f_2)_+)+\widetilde{W}_\gamma(\mathscr{H}(f_1-f_2)_-)-(\widetilde{W}_\gamma(z_1)+\widetilde{W}_\gamma(z_2)) \\
		& \quad +\sum_{i=3}^N(\widetilde{W}_\gamma(T_{\rho,2+a}(f_i))-\widetilde{W}_\gamma(z_i))\\
		&=\rho^d(\widetilde{W}_\gamma(\mathscr{H}(f_1-f_2))-\widetilde{W}_\gamma(z_1-z_2))+\sum_{i=3}^N(\widetilde{W}_\gamma(T_{\rho,2+a}(f_i))-\widetilde{W}_\gamma(z_i)) \\
		&\leq -\epsilon_1\rho^d \widetilde{W}_\gamma(z_1-z_2)-\epsilon_2\sum_{i=3}^{N}\widetilde{W}_\gamma(z_i)\leq -\epsilon_{\textup{bd}}\widetilde{W}_\gamma(z).
	\end{align*}
	The proof is thereby complete.   
\end{proof}

We finally prove the epiperimetric inequality at interior points of frequency $1$.
\begin{theorem}[Epiperimetric inequality at interior points]\label{thm:epi1}
	There exists $\epsilon_{\textup{int}}=\epsilon_{\textup{int}}(d)>0$ such that for all $1$-homogeneous $z\in H^1_{s,N}(B_1)$ there exists $w\in H^1_{s,N}(B_1)$ satisfying
	\begin{equation*}
		w=z~\text{on }\partial B_1\quad\text{and}\quad \widetilde{W}_1(w)\leq (1-\epsilon_{\textup{int}})\widetilde{W}_1(z).
	\end{equation*}
\end{theorem}
\begin{proof}
	Let
	\begin{equation*}
		\epsilon_{\textup{int}}:=\min\{\epsilon_1(d,1)\rho^d,\epsilon_2(d,1,\ell_0)\}
	\end{equation*}
	where $\epsilon_1$  and $\rho$ are as in \Cref{lemma:gain_harmonic} and $\epsilon_2$ as in \Cref{lemma:high}, both for $\gamma=1$. First of all, we claim that there exists at most two functions $f_1,f_2\in H^1(\partial B_1)$ with disjoint supports such that
	\begin{equation*}
		(d-1+\ell_0)\int_{\partial B_1}f_i^2\ds\geq \int_{\partial B_1}|\nabla_{\partial B_1}f_i|^2\ds\quad i=1,2,
	\end{equation*}
	for some dimensional $\ell_0>0$. This fact can be proved analogously to \Cref{thm:epi2}.  
	Hence, given a $1$-homogeneous function $z\in H^1_{s,N}(B_1)$ and letting $f=z\restr{\partial B_1}$, without loss of generality, we can assume that
	\begin{equation}
		(d-2+\ell_0)\int_{\partial B_1}f_i^2\ds\leq \int_{\partial B_1}|\nabla_{\partial B_1}f_i|^2\ds\quad\text{for all }i=3,\dots,N.
	\end{equation}
	We now define the competitor $w\in H^1_{s,N}(B_1)$. If $\mathcal{F}_1(f)\leq 0$ we just define let $w=z$, hence let us hereafter assume $\mathcal{F}_1(f)>0$. For $i=3,\dots,N$, we let
	\begin{equation*}
		w_i:=T_{\rho,1+a}(f_i),
	\end{equation*}
	where $\rho$, $a$ are as in \Cref{lemma:high} (with $\gamma=1$) while $T_{\rho,1+a}$ is as in \Cref{def:truncation}. Moreover, we let 
	\begin{equation*}
		w_1=R_{1,\rho}(\mathscr{H}(f_1-f_2)_+)\quad\text{and}\quad w_2=R_{1,\rho}(\mathscr{H}(f_1-f_2)_-),
	\end{equation*}
	where $H$ is as in \Cref{def:harm_ext} and $R_{1,\rho}$ as in \Cref{def:scaled}. One can easily see that $w\in H^1_{s,N}(B_1)$ and $w=z$ on $\partial B_1$. Now, in view of \Cref{lemma:scaling}, \Cref{lemma:gain_harmonic} and \Cref{lemma:high} we have that
	\begin{align*}
		\widetilde{W}_1(w)-\widetilde{W}_1(z)&=\rho^d(\widetilde{W}_1(\mathscr{H}(f_1-f_2)_+)+\widetilde{W}_1(\mathscr{H}(f_1-f_2)_-)-(\widetilde{W}_1(z_1)+\widetilde{W}_1(z_2)) \\
		& \quad +\sum_{i=3}^N(\widetilde{W}_1(T_{\rho,1+a}(f_i))-\widetilde{W}_1(z_i))\\
		&=\rho^d(\widetilde{W}_1(\mathscr{H}(f_1-f_2))-\widetilde{W}_1(z_1-z_2))+\sum_{i=3}^N(\widetilde{W}_1(T_{\rho,1+a}(f_i))-\widetilde{W}_1(z_i)) \\
		&\leq -\epsilon_1\rho^d \widetilde{W}_1(z_1-z_2)-\epsilon_2\sum_{i=3}^{N}\widetilde{W}_1(z_i)\leq -\epsilon_{\textup{int}}\widetilde{W}_1(z).
	\end{align*}
	The proof is thereby complete.   
\end{proof}
\section{Quantitative blow-up analysis}\label{sec:blowup}

In the present section, we perform a blow-up analysis at points of low frequency, that is $\gamma(x_0)=1$ or $\gamma(x_0)=2$, in case of boundary points.
\begin{proposition}[Blow-up rate]\label{prop:blow_up_rate}
	There exists $C_{\textup{rate}}>0$ depending on $d$, $D$ and $N$ such that
	\begin{equation*}
		\sum_{i=1}^N\int_{\partial B_1}|V_i^{r_2,\gamma}-V_i^{r_1,\gamma}|^2\ds\leq C_{\textup{rate}}H(R_0)\int_{r_1}^{r_2}\frac{\sigma_0(t)}{t}\d t
	\end{equation*}
	for all $0\leq r_1\leq r_2\leq R_0$ and all $\gamma\in [1,\min\{\gamma(x_0),2\}]$, where
	\[
	V^{r,\gamma}(x)=\frac{v(rx)}{r^\gamma}.
	\]
	In particular, there exists $V^\gamma\in \mathcal{B}_{\gamma}$ such that
	\begin{align*}
		&\sum_{i=1}^N\int_{\partial B_1}|V^{r,\gamma}_i-V^\gamma_i|^2\dx\leq C_{\textup{rate}}H(R_0)\int_0^r\frac{\sigma_0(r)}{r}\dx, \\
		&\sum_{i=1}^N\int_{B_1}|V^{r,\gamma}_i-V^\gamma_i|^2\dx\leq \frac{C_{\textup{rate}}H(R_0)}{d+2\gamma}\int_0^r\frac{\sigma_0(r)}{r}\dx,
	\end{align*}
	for any $r\leq R_0$.		
\end{proposition}
\begin{proof}
	We denote $C_{\textup{W}}':=C_{\textup{W}}H(R_0)$ and
	\begin{equation*}
		W(r):=W_\gamma(v,r)\quad\text{and}\quad\bar{W}(r):=W_\gamma(r)+C_{\textup{W}}'\int_0^r\frac{\sigma(t)}{t}\d t,
	\end{equation*}
	where $C_{\textup{W}}'$ is as in \Cref{prop:W'}, so that, thanks to \eqref{eq:W'_th2} we have that
	\begin{equation*}
		\bar{W}(r)\geq 0\quad\text{for all }r\in (0,R_0).
	\end{equation*}
	Moreover, for sake of simplicity in this proof, we denote $V^r:=V^{r,\gamma}$, being $\gamma$ fixed.
	First of all, we claim that
	\begin{equation}\label{eq:blow_up_1}
		\bar{W}'(r)\geq \frac{\epsilon(d+2\gamma-2)}{4r}\bar{W}(r)+\frac{\mathcal{D}_\gamma(r)}{2r}-\frac{CC_{\textup{W}}'}{r}\int_0^r \frac{\sigma(t)}{t}\d t
	\end{equation}
	for all $r\in(0,R_0)$, where $\epsilon:=\epsilon_{\textup{bd}}$ is as in \Cref{thm:epi2}. In order to prove \eqref{eq:blow_up_1}, we need to apply the epiperimetric inequality \Cref{thm:epi2} with
	\[
	z(x)=h^r(x):=|x|^\gamma V^{r,\gamma}\left(\frac{x}{|x|}\right).
	\]
	This can be done since, in view of \Cref{cor:starsh}, $r^{-1}\mathcal{O}_r\cap B_1\sub B_1^+$ and so $h^r=0$ on $B_1'$. This being observed, we start from \eqref{eq:W'_1} and we apply \Cref{thm:epi2}, denoting by $w^r\in H^1_{s,N}(B_1^+)$ the competitor for $h^r$. We thus obtain that
	\begin{equation}\label{eq:blow_up_2}
		\begin{aligned}
			\bar{W}'(r)&\geq (1-\kappa\sigma(r))\left[\frac{d+2\gamma-2}{r}\left(\widetilde{W}_\gamma(h^r,1)-\widetilde{W}_\gamma(V^r,1)\right)+\frac{\mathcal{D}_\gamma(r)}{r}\right]-C_{\textup{W}}'\frac{\sigma(r)}{r} \\
			&\geq (1-\kappa\sigma(r))\left[\frac{d+2\gamma-2}{r}\left( \frac{\widetilde{W}_\gamma(w^r,1)}{1-\epsilon}-\widetilde{W}_\gamma(V^r,1)\right)+\frac{\mathcal{D}_\gamma(r)}{r}\right]-C_{\textup{W}}'\frac{\sigma(r)}{r}. 
		\end{aligned}
	\end{equation}
	Next, from the almost minimality condition \Cref{prop:almost} we easily see that
	\begin{equation*}
		\widetilde{W}_\gamma(w^r,1)\geq \widetilde{W}_\gamma(V^r,1)(1-C_{\textup{am}}\sigma(r))-C_{\textup{am}}\sigma(r),
	\end{equation*}
	which, combined with \eqref{eq:blow_up_2} implies that
	\begin{equation}\label{eq:blow_up_3}
		\bar{W}'(r)\geq (1-\kappa\sigma(r))\left[\frac{d+2\gamma-2}{r}\widetilde{W}_\gamma(V^r,1)\left(\frac{1-C_{\textup{am}}\sigma(r)}{1-\epsilon}-1\right)+\frac{\mathcal{D}_\gamma(r)}{r}\right]-CC_{\textup{W}}'\frac{\sigma(r)}{r}.
	\end{equation}
	Now, we can choose $r$ sufficiently small in such a way that
	\begin{equation*}
		\frac{1-C_{\textup{am}}\sigma(r)}{1-\epsilon}-1=\frac{\epsilon-C_{\textup{am}}\sigma(r)}{1-\epsilon}\geq \frac{\epsilon}{2}>0
	\end{equation*}
	and
	\begin{equation*}
		\abs{\widetilde{W}_\gamma(V^r,1)-W(r)}\leq C\sigma(r).
	\end{equation*}
	Combining these facts and $\Bar{W}(r)\geq 0$ with \eqref{eq:blow_up_3} and manipulating the expression, we obtain \eqref{eq:blow_up_1}, taking into account the fact that $1-\kappa\sigma(r)\geq 1/2$. Next, we observe that
	\begin{equation}\label{eq:blow_up_4}
		\left(\frac{\bar{W}(r)}{\sigma_0(r)}+CC_{\textup{W}}'\int_0^r \frac{G(t)}{\sigma_0(t)}\right)'\geq \frac{\mathcal{D}_\gamma(r)}{2r\sigma_0(r)}\geq 0,
	\end{equation}
	for some $C>0$ sufficiently large and $r$ sufficiently small, where
	\begin{equation*}
		G(r):=\frac{1}{r}\int_0^r\frac{\sigma (t)}{t}\d t.
	\end{equation*}
	Indeed, from \eqref{eq:blow_up_1}, we have that
	\begin{equation*}
		\left(\frac{\bar{W}(r)}{\sigma_0(r)}\right)'=\frac{\bar{W}'(r)}{\sigma_0(r)}-\frac{\sigma_0'(r)}{\sigma_0(r)}\frac{\bar{W}(r)}{\sigma_0(r)} \geq \frac{\bar{W}(r)}{\sigma_0(r)}\left(\frac{\epsilon(d+2\gamma-2)}{4r}-\frac{\sigma_0'(r)}{\sigma_0(r)}\right)+\frac{\mathcal{D}_\gamma(r)}{2r\sigma_0(r)}-CC_{\textup{W}}'\frac{G(r)}{\sigma_0(r)}
	\end{equation*}
	By \Cref{ass:domain} and $\gamma\geq 1$,
	\[
	\frac{\epsilon(d+2\gamma-2)}{4r}-\frac{\sigma_0'(r)}{\sigma_0(r)}\geq\frac{\epsilon d}{4r}-\frac{\sigma_0'(r)}{\sigma_0(r)} -\frac{(r^{-\frac{\epsilon d}{4}}\sigma_0(r))'}{r^{-\frac{\epsilon d}{4}}\sigma_0(r)}\geq 0
	\]
	and since $\bar{W}(r)\geq 0$, we get \eqref{eq:blow_up_4}. Now, by classical computations (see e.g. \cite[Lemma 12.14]{velichkov2023regularity}), we know that
	\begin{equation*}
		\sum_{i=1}^N\int_{\partial B_1}|V_i^{r_1}-V_i^{r_2}|^2\ds\leq\sum_{i=1}^N \int_{\partial B_1}\left(\int_{r_1}^{r_2}\frac{1}{r}|\nabla V_i^{r}(x)\cdot x- V_i^r(x)|\d r\right)^2\ds(x).
	\end{equation*}
	Moreover, by Cauchy-Schwarz inequality, we have that
	\begin{equation*}
		\sum_{i=1}^N\int_{\partial B_1}|V_i^{r_1}-V_i^{r_2}|^2\ds\leq \int_{r_1}^{r_2}\frac{\mathcal{D}_\gamma(r)}{r\sigma_0(r)}\d r \int_{r_1}^{r_2}\frac{\sigma_0(r)}{r}\d r.
	\end{equation*}
	Finally, from \eqref{eq:blow_up_4} and \Cref{cor:N_bound} we derive that 
	\begin{align*}
		\int_{r_1}^{r_2}\frac{\mathcal{D}_\gamma(r)}{r\sigma_0(r)}\d r& \leq \int_0^{R_0}\frac{\mathcal{D}_\gamma(r)}{r\sigma_0(r)}\d r \leq \frac{\bar{W}(R_0)}{\sigma_0(R_0)}+CC_{\textup{W}}'\int_0^{R_0}\frac{G(t)}{\sigma_0(t)}\d t \\
		& =\frac{H(R_0)}{R_0^{2\gamma}}(\mathcal{N}(R_0)-\gamma)+C_{\textup{W}}'\int_0^{R_0}\frac{\sigma(t)}{t}\d t+CC_{\textup{W}}'\int_0^{R_0}\frac{G(t)}{\sigma_0(t)}\d t \leq 2C_{\textup{rate}}H(R_0),
	\end{align*}
	for all $0\leq r_1\leq r_2\leq R_0$, for some $C_{\textup{rate}}>0$ depending on $d$, $D$ and $N$, and this concludes the proof of the first part. In order to conclude the proof, we observe that, by completeness of $L^2(\partial B_1)$ there exists a blow-up limit $V^\gamma$ and the fact that $V^\gamma\in \mathcal{B}_\gamma$ follows by the almost minimality conditions and the fact that
	\[
	\lim_{r\to 0}W_\gamma(v,r)=\lim_{r\to 0}W_\gamma(V^r,1)=0,
	\]
	in view of \eqref{eq:W'_th2}. Finally, the estimate in $L^2(B_1)$ follows by integrating the one in $L^2(\partial B_1)$.
\end{proof}

\begin{proposition}[Nondegeneracy]\label{prop:nondeg}
	Let $x_0\in\mathcal{Z}_1^{\partial D}(u)\cup\mathcal{Z}_2^{\partial D}(u)$. Then, there holds
	\begin{equation}\label{eq:nondeg_th1}
		H_{x_0}:=\lim_{r\to 0}\frac{H(v^{x_0},r)}{r^{2\gamma(x_0)}}\in (0,\infty).
	\end{equation}
	In particular, if $V^{\gamma(x_0)}$ is as in \Cref{prop:blow_up_rate}, then $V^{\gamma(x_0)}\not\equiv 0$. Moreover, if $r_{x_0}\in(0,R_0)$ is such that
	\[
	\gamma(x)\leq \gamma(x_0)\quad\text{for all }x\in \overline{B_{r_{x_0}}(x_0)}
	\]
	then there exists $C_{\textup{nd}}>0$ (depending on $d$, $D$, $N$ and $r_{x_0}$) such that
	\begin{equation*}
		\frac{1}{C_{\textup{nd}}}\leq H_{x}\leq C_{\textup{nd}}\quad\text{for all }x\in \overline{B_{r_{x_0}}(x_0)}\cap \{ y\in\partial D\colon \gamma(y)=\gamma(x_0)\}.
	\end{equation*}
\end{proposition}
\begin{proof}
	First of all, we show that the function
	\[
	r\mapsto \frac{H(r)}{r^{2\gamma(x_0)}}
	\]
	admits a finite limit as $r\to 0$. Thanks to \Cref{lemma:H'} and \Cref{lemma:decay} we have that
	\begin{equation*}
		\left(\frac{H(r)}{r^{2\gamma(x_0)}}\right)'\geq \frac{2}{r}W(r)-C\frac{H(r)}{r^{2\gamma(x_0)}}\frac{\sigma(r)}{r}\geq \frac{2}{r}W(r)-CH(R_0)\frac{\sigma(r)}{r},
	\end{equation*}
	for some constant $C>0$ depending on $d$, $D$ and $N$ and $r$ sufficiently small. Now, by integrating \eqref{eq:W'_th2} we can bound $W(r)$ from below and, in view of \eqref{eq:sigma_nondeg_th2}, deduce that
	\begin{equation*}
		\left(\frac{H(r)}{r^{2\gamma(x_0)}}\right)'\geq -CH(R_0)\left(\frac{1}{r}\int_0^r\frac{\sigma(t)}{t}\d t- \frac{\sigma(r)}{r}\right)\geq -\frac{CH(R_0)}{r}\int_0^r\frac{\sigma(t)}{t}\d t .
	\end{equation*}    
	By assumption, the right-hand side is integrable near $0$, so there exists
	\begin{equation*}
		H_{x_0}:=\lim_{r\to 0}\frac{H(r)}{r^{2\gamma(x_0)}}\in [0,\infty).
	\end{equation*}
	Now, let as assume by contradiction that
	\begin{equation}\label{eq:nondeg_1}
		\frac{H(r)}{r^{2\gamma(x_0)}}\to 0\quad\text{as }r\to 0.
	\end{equation}
	In view of \Cref{prop:almgren_rescalings}, we know that for any $r_n\to 0$ there exists $U=U^{x_0}\in \mathcal{B}_{\gamma(x_0)}$ such that
	\begin{equation*}
		\sum_{i=1}^N\norm{U_i}_{L^2(\partial B_1)}^2=1
	\end{equation*}
	and, up to a subsequence,
	\[
	\tilde{v}^{r_n}\to U\quad\text{strongly in }H^1(B_1,\R^N)~\text{and }L^2(\partial B_1,\R^N)~\text{as }n\to\infty,
	\]
	where we recall that
	\[
	\tilde{v}^{r_n}(x)=\frac{v^{x_0}(r_n x)}{\sqrt{H(v^{x_0},r_n)}}.
	\]
	Let us now consider the function
	\[
	w^{r,\rho}(x):=\frac{\tilde{v}^{r}(\rho x)}{\rho^{\gamma(x_0)}}=(\tilde{v}^{r})^{\rho,\gamma(x_0)}.
	\]
	In view of \Cref{prop:blow_up_rate}, we know that there exists $W^r\in \mathcal{B}_{\gamma(x_0)}$ such that
	\begin{equation*}
		w^{r,\rho}\to W^r\quad\text{strongly in }L^2(\partial B_1,\R^N),~\text{as }\rho\to 0.
	\end{equation*}
	Moreover, thanks to \eqref{eq:nondeg_1} one can easily prove that $W^r=0$ for all $r$. 
	We now would like to apply \Cref{prop:blow_up_rate} to $w^{\rho,r}$. To be precise, \Cref{prop:blow_up_rate} is state for minimizers of $J_{N,\Psi_{x_0}}$ on $\mathcal{O}$; however, one can easily see that \Cref{prop:blow_up_rate} can be applied starting from $\tilde{v}^r$ rather than $v$ once there holds
	\begin{equation}\label{eq:nondeg1}
		H(\tilde{v}^{r},R_0)\leq C,\quad\text{uniformly for $r$ sufficiently small,}
	\end{equation}
	for some $C>0$ depending on $d$, $D$ and $N$. The proof of \eqref{eq:nondeg1} is actually straightforward: indeed, from \Cref{lemma:A} and by integration of \eqref{eq:H'_th3} we have that
	\[
	H(\tilde{v}^{r},R_0)\leq C\frac{H(v,r R_0)}{H(v,r)}\leq C(\mathcal{N}(v,R_0)+1),\quad\text{for $r$ sufficiently small,}
	\]
	and from \Cref{cor:N_bound} we obtain \eqref{eq:nondeg1}. Hence, there exists a constant $\hat{C}>0$ depending on $d$, $D$ and $N$ such that
	\begin{equation*}
		\sum_{i=1}^N\int_{\partial B_1}|w^{r,\rho}_i|^2\ds=\sum_{i=1}^N\int_{\partial B_1}|w^{r,\rho}_i-W^r_i|^2\ds\leq \hat{C}\int_{0}^{\rho}\frac{\sigma_0(t)}{t}\d t.
	\end{equation*}
	At this point, thanks to homogeneity properties and the previous estimate, we have that
	\begin{align*}
		1=\sum_{i=1}^N\int_{\partial B_1}U_i^2\ds&=\sum_{i=1}^N\frac{1}{\rho^{d+2\gamma(x_0)-1}}\int_{\partial B_\rho}U_i^2\ds \\
		&\leq \sum_{i=1}^N\frac{2}{\rho^{d+2\gamma(x_0)-1}}\int_{\partial B_\rho}|U_i-\tilde{v}^{r_n}_i|^2\ds+2\int_{\partial B_1}|w^{r_n,\rho}_i|^2\ds \\
		&\leq \sum_{i=1}^N\frac{2}{\rho^{d+2\gamma(x_0)-1}}\int_{\partial B_\rho}|U_i-\tilde{v}^{r_n}_i|^2\ds+\hat{C}\int_{0}^{\rho}\frac{\sigma_0(t)}{t}\d t.
	\end{align*}
	Now, by choosing $\rho$ sufficiently small and $n=n(\rho)$ sufficiently large, we reach a contradiction. Finally, the second part of the statement simply follows from the first part and from the continuity of $H_x$ with respect to $x$ in $\{y\in\partial D\colon \gamma(y)=\gamma(x_0)\}$.
\end{proof}

Hence, we can now sum up the main result of the present section, i.e. the complete blow-up analysis at points of frequency $\gamma(x_0)=1$ or $\gamma(x_0)=2$. The following result is stated for the original minimizer $u$, rather than $v$: passing from one to the other is possible in view of the properties of $\Psi_{x_0}$ and \eqref{eq:sigma_nondeg_th2}. For any $x_0\in\partial D$, $r>0$, $\gamma\geq 0$, we denote
\[
u_i^{r,x_0,\gamma}(x):=\frac{u_i(rx+x_0)}{r^\gamma},\quad x\in\frac{D-x_0}{r},~i=1,\dots,N.
\]
\begin{corollary}[Blow-up analysis]\label{cor:blowup}
	There exists $C_{\textup{BU}}>0$, depending only on $d$, $D$ and $N$, such that the following holds.
	For any $x_0\in\partial D$ there exists $P^{x_0,1}\in \mathcal{B}_1$ of the form
	\begin{equation*}
		P_j^{x_0,1}=a_{x_0,1}(-x\cdot\nnu(x_0))^+,\quad 		P_i^{x_0,1}=0\quad\text{for all }i\neq j
	\end{equation*}
	for some $j\in\{1,\dots,N\}$ and $a_{x_0,1}\geq 0$, such that
	\begin{equation*}
		\sum_{i=1}^N\int_{B_1}|u^{r,x_0,1}_i-P^{x_0,1}_i|^2\ds\leq C_{\textup{BU}}H(R_0)\int_0^r\frac{\sigma_0(t)}{t}\d t\quad\text{for all }r\leq R_0.
	\end{equation*}
	Moreover, there exists a dimensional constant $\kappa_{d,1}>0$ such that, if  $x_0\in\mathcal{Z}_1^{\partial D}(u)$, then $a_{x_0,1}=\kappa_{d,1} \sqrt{H_{x_0}}$, where $H_{x_0}>0$ is as in \Cref{prop:nondeg}, the Almgren blow-up limit $U^{x_0}$ in \Cref{cor:almgren_blowup} is unique and there holds
	\begin{equation*}
		P^{x_0,1}=\sqrt{H_{x_0}}U^{x_0}.
	\end{equation*}
	For any $x_0\in \partial D\setminus\mathcal{Z}_1^{\partial D}(u)$, there exists $P^{x_0,2}\in \mathcal{B}_2$ of the form
	\begin{equation}\label{eq:Px02}
		\begin{aligned}
			&P_j^{x_0,2}=a_{x_0,2}(x\cdot\bm{e}_{x_0})^-(-x\cdot\nnu(x_0))^+,\\
			&P_k^{x_0,2}=a_{x_0,2}(x\cdot\bm{e}_{x_0})^+(-x\cdot\nnu(x_0))^+,\\
			&		P_i^{x_0}=0\quad\text{for all }i\neq j,k
		\end{aligned}
	\end{equation}
	for some $j,k\in\{1,\dots,N\}$, $a_{x_0,2}\geq 0$ and $\bm{e}_{x_0}\in\partial B_1$, $\bm{e}_{x_0}\cdot\nnu(x_0)=0$, such that
	\begin{equation*}
		\sum_{i=1}^N\int_{B_1}|u^{r,x_0,2}_i-P^{x_0,2}_i|^2\ds\leq C_{\textup{BU}}H(R_0)\int_0^r\frac{\sigma_0(t)}{t}\d t\quad\text{for all }r\leq R_0.
	\end{equation*}
	Moreover, there exists a dimensional constant $\kappa_{d,2}>0$ such that, if  $x_0\in\mathcal{Z}_2^{\partial D}(u)$, then $a_{x_0,1}=\kappa_{d,2} \sqrt{H_{x_0}}$, where $H_{x_0}>0$ is as in \Cref{prop:nondeg}, the Almgren blow-up limit $U^{x_0}$ in \Cref{cor:almgren_blowup} is unique and there holds
	\begin{equation*}
		P^{x_0,2}=\sqrt{H_{x_0}}U^{x_0}.
	\end{equation*}
\end{corollary}

In view of \Cref{cor:blowup} and the Lipschitz continuity of the minimizer $u$ and of its $1$-homogeneous blow-up, we have the following.	
\begin{lemma}\label{lemma:Linfty_blowup}
	Let $x_0\in\partial D$ and let $P^{x_0,1}$ be as in \Cref{cor:blowup}. Then, there exists a constant $C_{\infty}>0$ depending only on $d$, $D$ and $N$ such that
	\begin{equation*}
		\sum_{i=1}^N\norm{u^{r,x_0,1}_i-P^{x_0,1}_i}_{L^\infty(B_1)}\leq C_{\infty}H(R_0)^{\frac{1}{d+2}}\left(\int_0^r\frac{\sigma_0(t)}{t}\d t\right)^{\frac{1}{d+2}}\quad\text{for all }r\leq R_0.
	\end{equation*}
\end{lemma}
\begin{proof}
	Since $u$ is Lipschitz continuous, see \Cref{prop:lipschitz}, then
	\[
	u_i^{r,x_0,1}(x)-P_i^{x_0,1}(x)\geq M_i-C_L|x|,\quad\text{for all }x\in B_1,
	\]
	and all $i=1,\dots,N$, where
	\[
	M_i:=\norm{u_i^{r,x_0,1}-P_i^{x_0,1}}_{L^\infty(B_1)}.
	\]
	Therefore, by explicit calculations, we have
	\[
	\norm{u_i^{r,x_0,1}-P_i^{x_0,1}}_{L^2(B_1)}^2\geq \int_{B_1}|(M_i-C_L|x|)^+|^2\dx=C\frac{M_i^{d+2}}{C_L^d},
	\]
	for some $C>0$ depending only on $d$. Therefore, by \Cref{cor:blowup} we conclude.		
\end{proof}

We conclude the section by stating the analogue of \Cref{cor:blowup} at interior points. In fact, even though regularity at interior points has already been extensively investigated in the literature, a quantitative blow-up analysis is missing, up to our knowledge, but still represents a key step when examining how the regular interior free boundary approaches regular points of $\partial D$. By performing the very same argument we described so far for boundary points, with the aid of the crucial result \Cref{thm:epi1}, we have the following.

\begin{proposition}\label{prop:int_blowup}
	Let $K\sub D$ be compact and let $\mathcal{F}(u)$ be as in \eqref{eq:free_boundary}. Then, there exists $\bar{C}_{\textup{BU}}>0$, $\bar{R}_0>0$ and $\bar{\alpha}\in(0,1)$ depending only on $d$ and $K$ such that the following holds. For any $x_0\in \mathcal{F}(u)\cap K$ there exists $\bar{a}_{x_0}\geq 0$, $\bar{\bm{e}}_{x_0}\in\partial B_1$ and $j,k\in\{1,\dots,N\}$ such that, letting
	\[
	\bar{P}^{x_0}_j=\bar{a}_{x_0}(x\cdot \bar{\bm{e}}_{x_0})^+,\quad \bar{P}^{x_0}_k=\bar{a}_{x_0}(x\cdot \bar{\bm{e}}_{x_0})^-\quad\text{and}\quad \bar{P}^{x_0}_i=0~\text{for all }i\neq j,k,
	\]
	we have that
	\[
	\sum_{i=1}^N \int_{ B_1}|u_i^{r,x_0,1}-\bar{P}^{x_0}_i|^2\ds\leq \bar{C}_{\textup{BU}}H(u,\bar{R}_0,x_0)\, r^{\bar{\alpha}}\quad\text{for all }r\leq \bar{R}_0,
	\]
	for some $\bar{\alpha}\in(0,1)$. Moreover, if $x_0\in\mathcal{R}(u)$, with $\mathcal{R}(u)$ being as in \Cref{thm:int_free_bound}, then $\bar{a}_{x_0}>0$.
\end{proposition}

\section{Regularity of the free boundary and clean-up}\label{sec:regularity}

In the present section, we exploit the results obtained in \Cref{sec:blowup} in order to conclude the proof of the main theorems of the present work, that is, up to the boundary regularity and clean-up results. Them main feature of these results relies in their quantitative nature.

\subsection{The optimal partition at the boundary}
In this section we define the traces of the optimal domains $\Omega_i$, $i=1,\dots,N$, at the boundary $\partial D$. The key result is a clean-up lemma (\Cref{l:clean-up-boundary-1}), in which we show that if the solution $u=(u_1,\dots,u_N)$ is sufficiently close to a one-homogeneous solution in a (small) ball with center on $\partial D$, then all the components but one vanish in some smaller ball.

\begin{lemma}\label{l:clean-up-boundary-1}
	For any $\delta>0$, there exist $\rho_1,\epsilon_1\in (0,1)$ depending on $d$, $D$, $N$ and $\delta$ such that, if
	\[
	\sum_{i=1}^N\norm{u_i-P_i}_{L^2(B_r(x_0))}^2\leq r^{d+2}\epsilon_1,
	\]
	for some $x_0\in \partial D$ and $r\leq R_0$, where $P=(a((x_0-x)\cdot\nnu(x_0))^+,0,\dots,0)$ for some $a\geq \delta$, then $u_1\geq \dist(\cdot,\partial D)\delta/4$ and as a consequence $u_i\equiv 0$ in $B_{r\rho_1}(x_0)$ for  any $i=2,\dots,N$. 
\end{lemma}
\begin{proof}
	Let $y\in B_{r\rho_1}(x_0)$ for some $\rho_1>0$ to be specified later and let $z\in B_{2r\rho_1}(x_0)$ be the projection of $y$ onto $\partial D$. We now consider the scaled function
	\[
	w^t(x):=u^{t,z,1}(x)=\frac{u(tx+z)}{t}
	\]
	and we let $W:=P^{z,1}$ be as in  \Cref{cor:blowup}.
	We know that there exists $j\in\{1,\dots,N\}$ such that $W_i\equiv 0$ for all $i\neq j$ and that
	\[
	W_j(x)=(-a'x\cdot\nnu(z))^+\quad\text{for some }a'\geq 0.
	\]
	In addition, in view of \Cref{lemma:Linfty_blowup}
	\begin{equation}\label{eq:clean_up1_1}
		\sum_{i=1}^N\norm{w^t_i-W_i}_{L^\infty(B_1)}\leq C_\infty H(R_0)^{\frac{1}{d+2}}\left(\int_0^t\frac{\sigma_0(s)}{s}\d s\right)^{\frac{1}{d+2}},
	\end{equation}
	for all $t\leq R_0$. We claim that
	\begin{equation}\label{eq:clean_up1_2}
		j=1\quad\text{and}\quad a'\geq \frac{\delta}{2}.
	\end{equation}
	Once this is proved, we can take $t=|y-z|$ in \eqref{eq:clean_up1_1} and, since $\nnu(z)=\frac{z-y}{t}$, obtain that
	\begin{align*}
		|t^{-1}u_1(y)-a'|=\left|w_1^t(-\nnu(z))-W_1(-\nnu(z)) \right|&\leq \sum_{i=1}^N\norm{w^t_i-W_i}_{L^\infty(B_1)} \\
		&\leq C_\infty H(R_0)^{\frac{1}{d+2}}\left(\int_0^t\frac{\sigma_0(s)}{s}\d s\right)^{\frac{1}{d+2}}.
	\end{align*}
	As a consequence, since $t=|y-z|\leq R_0\rho_1$, we have that
	\begin{equation*}
		\frac{\delta}{2}-\frac{u_1(y)}{t}\leq |t^{-1}u_1(y)-a'|\leq C_\infty H(R_0)^{\frac{1}{d+2}}\left(\int_0^{R_0\rho_1}\frac{\sigma_0(s)}{s}\d s\right)^{\frac{1}{d+2}}
	\end{equation*}
	and this implies that $u_1(y)\geq t\delta/4$ by taking $\rho_1$ sufficiently small, and this concludes the proof. Let us now prove \eqref{eq:clean_up1_2}. Since $B_{r(1-2\rho_1)}(z)\sub B_r(x_0)$, we deduce that
	\begin{equation*}
		\sum_{i=1}^N\norm{u_i-P_i}_{L^2(B_{r(1-2\rho_1)}(z))}^2\leq r^{d+2}\epsilon_1
	\end{equation*}
	from which we obtain that
	\begin{equation*}
		\sum_{i=1}^N\norm{w_i^t-\frac1t P_i(tx+z)}_{L^2(B_1)}^2\leq \frac{\epsilon_1 r^{d+2}}{t^{d+2}},
	\end{equation*}
	for every $t\le r(1-2\rho_1)$. At this point, we estimate
	\begin{align*}
		\frac1{C}|a-a'|&\le \norm{W_1(x)-P_1(x+x_0)}_{L^2_x(B_1)}\\
		&\leq\norm{\frac1t P_1(tx+z)-P_1(x+x_0)}_{L^2_x(B_1)}+\norm{W_1-\frac1t P_1(tx+z)}_{L^2(B_1)} \\
		&\leq\norm{ a\Big(\frac{x_0-z}{t}-x\Big)\cdot\nnu(x_0)+ax\cdot\nnu(x_0)}_{L^2_x(B_1)}+\sum_{i=1}^N\norm{W_i-\frac1t P_i(tx+z)}_{L^2(B_1)}\\
		&\leq a |B_1|\frac{2\rho_1 r}{t}+\left(C H(R_0)\int_0^t\frac{\sigma_0(s)}{s}\d s\right)^{1/2}+\left(\frac{\epsilon_1 r^{d+2}}{t^{d+2}}\right)^{1/2}\\
		&\leq C_L |B_1|\frac{2\rho_1 R_0}{t}+\left(CH(R_0)\int_0^t\frac{\sigma_0(s)}{s}\d s\right)^{1/2}+\left(\frac{\epsilon_1 R_0^{d+2}}{t^{d+2}}\right)^{1/2},
	\end{align*}
	where $C_L>0$ is the Lipschitz constant of $u$ and $C>0$ depends only on $d$, $D$ and $N$. Choosing, first $t$, then $\rho_1$ and $\epsilon_1$, small enough (depending on $\delta$, and also on $d$, $D$, $N$),  we get the claim.
\end{proof}

\begin{remark}\label{rmk:clean_up_1}
	We observe that \Cref{l:clean-up-boundary-1} can be applied if we replace $u$ and $D$ with 
	$$w(x):=\frac1R\frac{u(y_0+sx)}{s}\qquad\text{and}\qquad D_{s,y_0}:=(-y_0+D)/s,$$ 
	for some $R>0$, $s>0$ and $y_0\in\partial D$. This is true under the condition that
	\[
	R_0^{1-d}\sum_{i=1}^N\norm{w_i}_{L^2(\partial B_{R_0})}^2=\frac1{s^2R^2}(sR_0)^{1-d}\sum_{i=1}^N\norm{u_i}_{L^2(\partial B_{sR_0}(y_0))}^2\leq C,
	\]
	for some $C>0$ depending only on $d$, $D$ and $N$. 
	Essentially, this is a consequence of the fact that $w$ is  a minimizer for \eqref{eq:opt_part_variat} in $D_{s,y_0}$ and of the fact that the universal constant $C_{\infty}$ appearing in \Cref{lemma:Linfty_blowup} is multiplied by $H(R_0)$.
\end{remark}

\subsection{Traces of the optimal domains}

We here define the traces $\omega_i$ of the optimal domains $\Omega_i$ on the boundary $\partial D$. The following is a direct consequence of \Cref{l:clean-up-boundary-1}.
\begin{lemma}\label{lemma:definition-omega-j}
	Let $(\Omega_1,\dots,\Omega_N)$ be the optimal partition in $D$. Let $x_0\in \partial D$ and let $j\in\{1,\dots,N\}$ be fixed. Then, the following are equivalent:
	\begin{enumerate}
		\item[\rm (1)] $\gamma(x_0)=1$ and $u^{r,x_0,1}(x):=\frac1ru(x_0+rx)$ converges to a function of the form 
		$$P(x):=\Big(0,\dots,a\big(-x\cdot\nnu(x_0)\big)^+,\dots,0\Big),$$
		where the $j$th is the only non-zero component of $P$ and $a>0$.
		\item[\rm (2)] The following Taylor expansion holds for points $x_0+x\in \overline D$:
		$$u_j(x)=a(-(x-x_0)\cdot\nnu(x_0))^++o(|x-x_0|)~\text{as }x\to x_0 $$
		for some $a>0$ and
		\[
		u_i(x)=o(|x-x_0|)~\text{as }x\to x_0\quad\text{for every}\quad i\neq j.
		\]
		\item[\rm(3)] There is a ball $B_r(x_0)$ such that
		$$B_r(x_0)\cap D=\Omega_j\cap D.$$
	\end{enumerate}
\end{lemma}
\begin{proof}
	We first notice that clearly (1) and (2) are equivalent. Next, if assume (1), then the clean-up lemma (Lemma \ref{l:clean-up-boundary-1}) implies that  
	$$B_r(x_0)\cap \Omega_i=\emptyset\qquad\text{for every}\quad i\neq j,$$
	for some $r>0$. Then, by the interior unique continuation (see \cite{CL2007}), we get (3). Conversely, if we assume (3), then by the Hopf maximum principle at $x_0$, we get (2).
\end{proof}
In view of \Cref{lemma:definition-omega-j}, we can define the partition $\omega_j$, $j=1,\dots,N$ of $\partial D$ as follows:
$$\omega_j:=\Big\{x_0\in\partial D\ :\ B_r(x_0)\cap D=\Omega_j\cap D~\text{for some }r>0\Big\}.$$
We are now ready to prove our first main result, that is \Cref{thm:taylor}.
\begin{proof}[\bf Proof of \Cref{thm:taylor}]
	We preliminarily observe that cases \textit{1)}, \textit{2)} and \textit{3)} occur, respectively, when $\gamma(x_0)=1$, $\gamma(x_0)=2$ or $\gamma(x_0)>2$. Hence, point \textit{1)} is a direct consequence of \Cref{cor:blowup} and \Cref{lemma:definition-omega-j}, while point \textit{2)} simply follows from \Cref{cor:blowup}. Finally, if $\gamma(x_0)>2$, then combining \Cref{cor:blowup} with \Cref{cor:almgren_blowup} and \Cref{lemma:decay} point \textit{(i)}, one can easily prove that $P^{x_0,2}\equiv 0$ (with $P^{x_0,2}$ being as in \Cref{cor:blowup}), and this concludes the proof of point \textit{3)}.
\end{proof}
At this point, we define
\begin{equation*}\label{eq:def_A_i}
	A_j:=\mathrm{Int}_{\partial D}(\overline{\Omega_j}\cap\partial D)
\end{equation*}
and we prove the following structure result.
\begin{proposition}\label{prop:topology}
	The following hold true:
	\begin{enumerate}
		\item [\rm (i)] $\bigcup_{i=1}^N\omega_i=\mathcal{Z}_1^{\partial D}(u)$;
		\item [\rm (ii)] $\mathcal F_{\partial D}(u)=\partial D\setminus \mathcal Z_1^{\partial D}(u)$, where we recall that $\mathcal{F}_{\partial D}(u):=\overline{\mathcal{F}(u)}\cap \partial D$;
		\item [\rm (iii)] $\mathcal{F}_{\partial D}(u)=\bigcup_{i=1}^N\partial_{\partial D} \omega_i$ or, equivalently, $\mathrm{Int}_{\partial D}(\mathcal{F}_{\partial D}(u))=\emptyset$;
		\item[\rm (iv)] $A_i\cap \overline{\omega_j}=\emptyset$ for all $i\neq j$;
		\item[\rm (v)] $A_i=\mathrm{Int}_{\partial D}(\overline{\omega_i})$ for all $i=1,\dots,N$.
	\end{enumerate}
\end{proposition}
\begin{proof}
	The proof of (i) is an immediate consequence of \Cref{lemma:definition-omega-j}, so we first prove (ii).\\ If $x_0\in\mathcal{F}_{\partial D}(u)$ then by definition, there is a sequence of points $x_0^{(n)}\in \mathcal{F}(u)\subset D$ such that $x_0^{(n)}\to x_0$ as $n\to\infty$. Moreover, by the known results in the interior (see \Cref{thm:int_free_bound}), in any neighborhood of $x_0^{(n)}$ there are at least two non-zero components; hence the same holds true for $x_0$, which implies that $x_0\in\partial D\setminus(\cup_i\omega_i)$, which proves that $\mathcal F_{\partial D}(u)\subset\partial D\setminus \mathcal Z_1^{\partial D}(u)$.
	On the other hand, if $x_0\in\partial D\setminus\mathcal{Z}_1^{\partial D}(u)$, then by \Cref{cor:almgren_blowup} and \Cref{lemma:frequencies} its blow-up limit $U$ must have at least two different non-zero components $U_i$ and $U_j$, thus implying that $\Omega_i\cap B_r(x_0)\neq \emptyset$ and $\Omega_j\cap B_r(x_0)\neq \emptyset$ for every $r>0$, which in turn implies that $\partial\Omega_i\cap B_r(x_0)\neq \emptyset$ for every $r>0$ and so $\mathcal F(u)\cap B_r(x_0)\neq \emptyset$ for every $r>0$. This concludes the proof of (ii).
	Let us prove (iii). Since 
	\[
	\partial_{\partial D} \omega_i\sub \partial D\setminus \mathcal{Z}_1^{\partial D}(u)\quad\text{for every }i,
	\]
	in view of (ii), we have that $\bigcup_{i=1}^N\partial_{\partial D} \omega_i\sub \mathcal{F}_{\partial D}(u)$. Let now $x_0\in \mathcal{F}_{\partial D}(u)$. Assume by contradiction that there exists $x_0\in\partial D$ and $r>0$ such that
	\[
	B_r(x_0)\cap \partial D\sub \partial D\setminus\Big(\bigcup_{i=1}^N\overline{\omega_i}\Big).
	\]
	Then, in view of \Cref{cor:blowup}, we have that $u_i$ is differentiable at any point of $B_r(x_0)\cap \partial D$ and there holds
	\begin{equation}\label{eq:main2_1}
		u_i=\partial_{\nnu} u_i=0\quad\text{on }B_r(x_0)\cap \partial D.
	\end{equation}
	Furthermore, from \Cref{lemma:extermaliti} we have that
	\begin{equation*}
		\begin{cases}
			-\Delta u_i\leq \lambda_i u_i, &\text{in }B_r(x_0)\cap D, \\
			-\Delta\left(u_i-\sum_{j\neq i}u_j\right)\geq \lambda_i u_i-\sum_{j\neq i}\lambda_j u_j,&\text{in }B_r(x_0)\cap D
		\end{cases}
	\end{equation*}
	in a distributional sense, and this, together with \eqref{eq:main2_1}, implies that this holds in the whole $B_r(x_0)$, up to extending by $0$ all the components $u_i$. Namely,
	\begin{equation*}
		\begin{cases}
			-\Delta u_i\leq \lambda_i u_i, &\text{in }B_r(x_0), \\
			-\Delta\left(u_i-\sum_{j\neq i}u_j\right)\geq \lambda_i u_i-\sum_{j\neq i}\lambda_j u_j,&\text{in }B_r(x_0)
		\end{cases}
	\end{equation*}
	in a distributional sense. In other words, $u=(u_1,\dots,u_N)$ belongs to the class $\mathcal{S}(B_r(x_0))$, which was introduced in \cite{CTV2003} and which we here recall. Given an open set $\Omega\subset\R^d$, we say that $u=(u_1,\dots,u_N)\in (H^1(\Omega))^N$ belongs to the class $\mathcal S(\Omega)$ if: 
	\begin{itemize}
		\item $u_j\ge 0$ for every $j=1,\dots,N$;
		\item $u_iu_j=0$ for every $i\neq j$; 
		\item there are $\lambda_j>0$, $j=1,\dots,N$ such that, for every $j\ge 1$, $\Delta u_j\ge -\lambda_j u_j$ in sense of distributions in $\Omega$; 
		\item for every $j=1,\dots,N$, 
		$$\Delta\Big(u_j-\sum_{i\neq j}u_i\Big)\le -\lambda_j u_j+\sum_{i\neq j}\lambda_iu_i\,$$
		in sense of distributions in $\Omega$.
	\end{itemize}
	Hence, since $u_i\equiv 0$ in $B_r(x_0)\cap (\R^d\setminus D)$, this contradicts the unique continuation theorem for this class of functions \cite[Theorem 1.1]{TerraciniTavares2012}, thus proving (iii). In order to prove (iv), we just observe that, by definition of $\omega_j$, we have that $A_i\cap \omega_j=\emptyset$ for all $i\neq j$, which implies that $A_i\cap \overline{\omega_j}=\emptyset$. Finally, let us prove (v). In order to do this, we first prove that $A_i\sub\overline{\omega_i}$ for all $i$. If $x_0\in A_i$, then from (iv) we deduce that $x_0\not \in \overline{\omega_j}$ for all $j\neq i$ and, in view of (iii) we obtain that $x_0\in \overline{\omega_i}$. Finally, since $A_i$ is open we conclude that (v) holds.
\end{proof}

\subsection{\texorpdfstring{Regularity of $\mathcal{Z}_2^{\partial D}(u)$}{Regularity near the points of frequency 2}}
We now pass, in the present section, to the proof of the regularity of the set $\mathcal{Z}_2^{\partial D}(u)$ and of a boundary clean-up result. We first introduce
\[
\Upsilon(r):=r^2\left(\int_0^r\frac{\sigma_0(t)}{t}\d t\right)^{\frac{1}{2}}.
\]
One can immediately observe that $\Upsilon$ is invertible in $[0,R_0]$, hence we can define
\[
\theta(r):=\left(\int_0^{\Upsilon^{-1}(r)}\frac{\sigma_0(t)}{t}\d t\right)^{\frac{1}{2}}\quad\text{for }r\leq R_{\theta},
\]
for some $R_{\theta}\leq R_0$. We now prove how the vector $\bm{e}_{x_0}\in\partial B_1$ as in \Cref{cor:blowup} oscillates with respect to $x_0\in\partial D$.
\begin{lemma}\label{lemma:osc}
	For any $x_0\in  \mathcal{Z}_2^{\partial D}(u)$, let $\bm{e}_{x_0}\in\partial B_1$ be as in \Cref{cor:blowup} and let $r_{x_0}\in (0,R_\theta/2)$ be such that
	\[
	\overline{B_{r_{x_0}}(x_0)}\cap \left( \partial D\setminus \mathcal{Z}_1^{\partial D}(u)\right)= \overline{B_{r_{x_0}}(x_0)}\cap \mathcal{Z}_2^{\partial D}(u).
	\]
	Then there exists $C_{\textup{osc}}>0$ depending on $d$, $D$, $N$ and $r_{x_0}$ such that
	\begin{equation}
		|\bm{e}_y-\bm{e}_z|\leq C_{\textup{osc}} \theta(|y-z|),\quad\text{for all }y,z\in \overline{B_{r_{x_0}}(x_0)}\cap \mathcal{Z}_2^{\partial D}(u).
	\end{equation}
\end{lemma}
\begin{proof}
	For any $x_0\in\mathcal{Z}_2^{\partial D}(u)\cap\overline{B_{r_{x_0}}(x_0)}$, we let $P^{x_0,2}$ be as in \Cref{cor:blowup}.\\ Since $\sqrt{H_{x_0}}=\norm{P^{x_0,2}}_{L^2(\partial B_1,\R^N)}>0$, there is a dimensional constant $C>0$ such that
	\begin{align*}
		|\bm{e}_y-\bm{e}_z|&\leq C\norm{\frac{P^{y,2}}{\norm{P^{y,2}}_{L^2(\partial B_1,\R^N)}}-\frac{P^{z,2}}{\norm{P^{z,2}}_{L^2(\partial B_1,\R^N)}}}_{L^2(\partial B_1,\R^N)} \\
		&=\frac{C}{\sqrt{H_y H_z}}\norm{\sqrt{H_z}P^{y,2}-\sqrt{H_y}P^{z,2}}_{L^2(\partial B_1,\R^N)}
	\end{align*}
	for all $y,z\in\mathcal{Z}_2^{\partial D}(u)$.	At this point, \Cref{prop:nondeg} yields
	\begin{equation}\label{eq:osc1}
		|\bm{e}_y-\bm{e}_z|\leq C\norm{P^{y,2}-P^{z,2}}_{L^2(\partial B_1,\R^N)}\quad\text{for all }y,z\in\mathcal{Z}_2^{\partial D}(u)\cap\overline{B_{r_{x_0}}(x_0)},
	\end{equation}
	for some other constant $C>0$ depending on $d$, $D$, $N$ and $r_{x_0}$. By the triangular inequality
	\begin{multline}\label{eq:osc2}
		\norm{P^{y,2}-P^{z,2}}_{L^2(\partial B_1,\R^N)}\leq \norm{u^{r,y,2}-P^{y,2}}_{L^2(\partial B_1,\R^N)} \\
		+\norm{u^{r,z,2}-P^{z,2}}_{L^2(\partial B_1,\R^N)}+\norm{u^{r,y,2}-u^{r,z,2}}_{L^2(\partial B_1,\R^N)}.
	\end{multline}
	We next estimate the three terms in the right-hand side of the previous inequality. Concerning the last one, we can see that for any $y,z\in \mathcal{Z}_2^{\partial D}(u)$ and any $r\leq R_0$, there holds
	\begin{equation*}
		\sum_{i=1}^N|u^{r,y,2}_i(x)-u^{r,z,2}_i(x)|=\frac{1}{r^2}\sum_{i=1}^N\abs{u_i(rx+y)-u_i(rx+z) }\leq \frac{C_L|y-z|}{r^2}
	\end{equation*}
	for all $x\in B_1$, where $C_L>0$ is the Lipschitz constant of $u$, which implies that
	\begin{equation}\label{eq:osc3}
		\norm{u^{r,y,2}-u^{r,z,2}}_{L^2(\partial B_1,\R^N)}\leq \frac{C|y-z|}{r^2}
	\end{equation}
	for all $y,z\in\mathcal{Z}_2^{\partial D}(u)\cap\overline{B_{r_{x_0}}(x_0)}$ and all $r>0$. For what concerns the first two terms in the right-hand side of \eqref{eq:osc2}, thanks to \Cref{cor:blowup} we obtain that
	\begin{equation*}
		\norm{u^{r,y,2}-P^{y,2}}_{L^2(\partial B_1,\R^N)} \\
		+\norm{u^{r,z,2}-P^{z,2}}_{L^2(\partial B_1,\R^N)}\leq C\left(\int_0^r\frac{\sigma_0(t)}{t}\d t\right)^{\frac{1}{2}},
	\end{equation*}
	for all $y,z\in\mathcal{Z}_2^{\partial D}(u)\cap\overline{B_{r_{x_0}}(x_0)}$ and all $r\leq R_0$. Hence, by combining this last inequality with \eqref{eq:osc3}, \eqref{eq:osc2} and \eqref{eq:osc1} we obtain that
	\begin{equation}\label{eq:osc4}
		|\bm{e}_y-\bm{e}_z|\leq C\left[ \left(\int_0^r\frac{\sigma_0(t)}{t}\d t\right)^{\frac{1}{2}}+\frac{|y-z|}{r^2} \right],
	\end{equation}
	which concludes the proof by choosing $r=\Upsilon^{-1}(|y-z|)$.
\end{proof}

Next, exploiting \Cref{l:clean-up-boundary-1} and \Cref{rmk:clean_up_1}, we obtain a flatness result for $\mathcal{Z}_2^{\partial D}(u)$.

\begin{lemma}[Flatness condition]\label{lemma:flatness}
	For any $\eta,\delta,\rho\in(0,1)$ there exists $\epsilon_{\textup{flat}},R_{\textup{flat}},\rho_{\textup{flat}}\in(0,1)$ depending on $d$, $D$, $N$, $\eta$, $\delta$ and $\rho$ such that, if
	\begin{equation}\label{e:lemma-flatness-ipo}
		\sum_{i=1}^N\norm{u_i-P_i}_{L^2(B_r(x_0))}^2\leq r^{d+4}\epsilon_{\textup{flat}},
	\end{equation}
	for some $x_0\in \partial D$ and $r\leq R_{\textup{flat}}$, where
	\[P=(a((x-x_0)\cdot\bm{e}_{x_0})^+((x_0-x)\cdot\nnu(x_0))^+,a((x-x_0)\cdot\bm{e}_{x_0})^-((x_0-x)\cdot\nnu(x_0))^+,0,\dots,0)\]
	for some $a\in( \delta,\frac1{\delta})$ and $\bm{e}_{x_0}\in\partial B_1$, then
	\[
	u_1\geq \frac{1}{4}\dist(\cdot,\partial D)\delta\eta\rho r\quad\text{in }B_{r\rho_{\textup{flat}}}(z)
	\]
	for all $z\in \partial D\cap B_{\rho r}(x_0)$ such that $z\cdot\bm{e}_{x_0}\geq \eta\rho r$ and 
	\[
	u_2\geq \frac{1}{4}\dist(\cdot,\partial D)\delta\eta\rho r\quad\text{in }B_{r\rho_{\textup{flat}}}(z)
	\]
	for all $z\in \partial D\cap B_{\rho r}(x_0)$ such that $z\cdot\bm{e}_{x_0}\leq- \eta\rho r$. In particular,
	\begin{align*}
		&B_{\rho r}(x_0)\cap\{x\in\partial D\colon (x-x_0)\cdot\bm{e}_{x_0}> \eta\rho r\}\sub \omega_1 \\
		&B_{\rho r}(x_0)\cap\{x\in\partial D\colon (x-x_0)\cdot\bm{e}_{x_0}<- \eta\rho r\}\sub \omega_2
	\end{align*}
	and
	\begin{align*}
		&B_{\rho r}(x_0)\cap\{x\in D\colon (x-x_0)\cdot\bm{e}_{x_0}> \eta\rho r,~ \dist(x,\partial D)<r\rho_{\textup{flat}}\}\sub \Omega_1 \\
		&B_{\rho r}(x_0)\cap\{x\in D\colon (x-x_0)\cdot\bm{e}_{x_0}<- \eta\rho r,~ \dist(x,\partial D)<r\rho_{\textup{flat}}\}\sub \Omega_2.
	\end{align*}
\end{lemma}
\begin{proof}
	Without loss of generality, we may assume that $x_0=0$, $\nnu(x_0)=-\bm{e}_d$ and $\bm{e}_{x_0}=\bm{e}_{d-1}$. Let 
	\[
	z\in B_{\rho r}\cap\{x\in\partial D\colon x_{d-1}> \eta\rho r\}
	\]
	and $s=r\rho_{\textup{flat}}\rho_1^{-1}$, where $\rho_1=\rho_1(d,D,N,\eta\delta\rho)$ is as in \Cref{l:clean-up-boundary-1}, be such that $B_s(z)\sub B_r$ and let
	\[
	\widetilde P_1(x):=a\frac{z_{d-1}}{r}(-(x-z)\cdot\nnu(z))^+,\quad \widetilde P_i\equiv 0~\text{for all }i\geq 2.
	\]
	We want to apply \Cref{l:clean-up-boundary-1} with $\frac{u(z+sx)}{rs}$ in place of $u$ (as explained in \Cref{rmk:clean_up_1}), $0$ in place of $x_0$, $\frac1s\widetilde P(sx+z)$ in place of $P(x)$, $az_{d-1}/r$ in place of $a$, $\delta \eta\rho$ in place of $\delta$, $1$ in place of $r$. In view of \Cref{rmk:clean_up_1}, we need to verify that
	\[
	\int_{\partial B_{R_0}}\abs{\frac{u(sx+z)}{rs}}^2\ds\leq C
	\]
	for some $C>0$ depending only on $d$, $D$ and $N$, and all $r,s>0$ sufficiently small. But this is easily verified in view of \eqref{e:lemma-flatness-ipo} (by choosing $\rho_{\textup{flat}}$ sufficiently small, since $s=\rho_{\textup{flat}}\rho_1^{-1}r$), and thanks to the fact that 
	\[
	-\Delta u_i^2\leq 2\lambda_i u_i^2\quad\text{in }\R^d.
	\]
	Now, since 
	\[\sum_{i=1}^Ns^{d+2} \int_{B_1}\abs{\frac{u(z+sx)}{rs}-\frac1s\widetilde P(sx+z)}^2\dx =\sum_{i=1}^N\norm{\frac{u_i}{r}-\widetilde P_i}_{L^2(B_s(z))}^2,\]
	we need to estimate
	\begin{equation}\label{eq:flatness_1}
		\begin{aligned}
			\sum_{i=1}^N\norm{\frac{u_i}{r}-\widetilde P_i}_{L^2(B_s(z))}^2
			&\le 2\sum_{i=1}^N\norm{\frac{u_i}{r}-\frac{P_i}{r}}_{L^2(B_s(z))}^2+2\sum_{i=1}^N\norm{\frac{P_i}{r}-\widetilde P_i}_{L^2(B_s(z))}^2\\
			&= \frac{2}{r^2}\sum_{i=1}^N\norm{u_i-P_i}_{L^2(B_s(z))}^2+2\norm{\frac{P_1}{r}-\widetilde P_1}_{L^2(B_s(z))}^2,
		\end{aligned}
	\end{equation}
	where in the last inequality we used that $P_i=\widetilde P_i\equiv 0$ on $B_s(z)$ for every $i\ge 2$. For what concerns the first term, in view of \eqref{e:lemma-flatness-ipo} we have
	\[
	\sum_{i=1}^N\norm{u_i-P_i}_{L^2(B_s(z))}^2\leq \sum_{i=1}^N\norm{u_i-P_i}_{L^2(B_r)}^2\leq \epsilon_{\textup{flat}}r^{d+4},
	\]
	while for the second one, since $P_1(x)=ax_{d-1}^+x_d^+$, we have that
	\begin{align*}
		\abs{\frac{P_1}{r}-\widetilde P_1}^2&\leq \frac{2a^2}{r^2}\Big((z_{d-1}-x_{d-1})^2((x-z)\cdot\nnu(z))^2+x_{d-1}^2(x_d+(x-z)\cdot\nnu(z))^2\Big) \\
		&\leq \frac{4a^2}{r^2}\Big((z_{d-1}-x_{d-1})^2((x-z)\cdot\nnu(z))^2+x_{d-1}^2((x-z)\cdot(\nnu(z)-\nnu(0))^2 +x_{d-1}^2z_d^2 \Big),
	\end{align*}
	which implies that 
	\[
	\norm{\frac{P_1}{r}-\widetilde P_1}_{L^2(B_s(z))}^2\leq \frac{Ca^2}{r^2}(s^{d+4}+s^{d+2}r^2\sigma^2(r)+s^d r^4\sigma^2(r)),
	\]
	for some constant $C>0$ depending only on $d$ and $D$, where we used that, by assumptions on $D$, $|\nnu(0)-\nnu(z)|$ and $|z_d|$ are bounded by a universal constant times $\sigma(|z|)$. Hence, plugging these estimates into \eqref{eq:flatness_1}, we obtain that
	\[\sum_{i=1}^N\norm{\frac{u_i}{r}-\widetilde P_i}_{L^2(B_s(z))}^2\leq 2\epsilon_{\textup{flat}}r^{d+2}+ 2\frac{Ca^2}{r^2}\Big(s^{d+4}+s^{d+2}r^2\sigma^2(r)+s^d r^4\sigma^2(r)\Big).\]
	Being $s=r\rho_{\textup{flat}}\rho_1^{-1}\le r\le R_{\textup{flat}}$ and $a^2\le \delta^{-2}$, this translates into
	\begin{align*}
		\sum_{i=1}^N\norm{\frac{u_i}{r}-\widetilde P_i}_{L^2(B_s(z))}^2
		&\leq Cs^{d+2}\bigg(\frac{\epsilon_{\textup{flat}}\rho_1^{d+2}}{\rho_{\textup{flat}}^{d+2}}+\frac{\rho_{\textup{flat}}^2}{\delta^2\rho_1^2}+\frac{\sigma^2(R_{\textup{flat}})\rho_1^2}{\delta^2\rho_{\textup{flat}}^2}\bigg),
	\end{align*}
	where, as above, $C=C(d,D,\eta\delta,\rho)$. Finally, choosing first $\rho_{\textup{flat}}$, then $R_{\textup{flat}}$ and $\epsilon_{\textup{flat}}$ small enough, in such a way that 
	\[C\bigg(\frac{\epsilon_{\textup{flat}}\rho_1^{d+2}}{\rho_{\textup{flat}}^{d+2}}+\frac{\rho_{\textup{flat}}^2}{\delta^2\rho_1^2}+\frac{\sigma^2(R_{\textup{flat}})\rho_1^2}{\delta^2\rho_{\textup{flat}}^2}\bigg)\le \epsilon_1,\]
	where $\epsilon_1=\epsilon_1(d,D,N,\eta\rho\delta)$ is as in \Cref{l:clean-up-boundary-1}, by \Cref{l:clean-up-boundary-1}, we conclude the proof.
\end{proof}

At this point, we are able to prove regularity of the regular part of the free boundary. Since $\mathcal{Z}^{\partial D}_2(u)$ coincides with $\mathcal{R}_{\partial D}(u)$, the following result contains the proof of \Cref{thm:fixed_free_bound}.
\begin{proposition}[Regularity of $\mathcal{Z}_2^{\partial D}(u)$]\label{prop:reg}
	For all $x_0\in\mathcal{Z}_2^{\partial D}(u)$, there exists $R=R_{x_0}>0$ (depending on $d$, $D$, $N$ and $x_0$) such that $(\partial D\setminus\mathcal{Z}_1^{\partial D}(u))\cap B_R(x_0)=\mathcal{Z}_2^{\partial D}(u)\cap B_R(x_0)\sub \partial D$ is a $(d-2)$-dimensional submanifold of class $C^1$ and there exists $j,k\in\{1,\dots,N\}$ such that
	\begin{equation*}
		\mathcal{Z}^{\partial D}_2(u)\cap B_R(x_0)=\partial \omega_j\cap\partial \omega_k\cap B_R(x_0).
	\end{equation*}
\end{proposition}
\begin{proof}
	The proof closely follows classical arguments, which can be found e.g. in \cite[Theorem 5]{weiss-obst}, see also \cite[Section 8.2]{velichkov2023regularity}. Let $R\leq R_{\textup{osc}}\leq R_0$ (where $R_{\textup{osc}}>0$ is as in \Cref{lemma:osc}) and let $x_0\in\mathcal{Z}_2^{\partial D}(u)$. Moreover, by upper-semicontinuity of $\gamma(\cdot)$ we can take $R$ sufficiently small (depending on $x_0$) in such a way that $(\partial D\setminus\mathcal{Z}_1^{\partial D}(u))\cap B_R(x_0)=\mathcal{Z}_2^{\partial D}(u)\cap B_R(x_0)$. It is not restrictive to assume that $x_0=0$ and $\nnu(0)=-\bm{e}_d$. We now apply \Cref{cor:blowup} and assume, without loss of generality, that $j=1$, $k=2$ and $\bm{e}_{x_0}=\bm{e}_{0}=\bm{e}_{d-1}$. Hence, if we denote $P=P^{0,2}$ and $a=a_{0,2}>0$, then we have that
	\begin{equation}\label{eq:reg1}
		\sum_{i=1}^N\norm{u_i-P_i}_{L^2(B_r)}^2\leq \frac{C_{\textup{BU}}}{d+4}r^{d+4}\int_0^r\frac{\sigma_0(t)}{t}\d t\quad\text{for all }r\leq R,
	\end{equation}
	where
	\begin{equation*}
		P_1(x)=ax_{d-1}^+x_d^+,\quad P_2(x)=ax_{d-1}^-x_d^+\quad\text{and}\quad P_i(x)\equiv 0~\text{for all }i\geq 2.
	\end{equation*}
	From this fact and \Cref{lemma:flatness}, we have that, up to restricting $R$, 
	$
	0\in\partial\omega_1\cap\partial\omega_2.
	$\\
	For any $\epsilon>0$, any $\bm{e}\in\partial B_1$ and any $y\in \mathcal{Z}^{\partial D}(u)\cap B_R$, we denote
	\[
	\mathcal{C}^{\pm}_\epsilon(y,\bm{e}):=\left\{x\in\R^d\colon \pm\frac{x-y}{\abs{x-y}}\cdot \bm{e}> \epsilon\right\}
	\]
	Let us first assume that $\partial D$ is flat in a neighborhood of $0$, that is 
	\begin{align*}
		D\cap B_R=\{x\in B_R\colon x_d>0\}\qquad\text{and}\qquad
		\partial D\cap B_R=\{x\in B_R\colon x_d=0\}.
	\end{align*}
	We observe that, for any $\epsilon>0$ there exists $R_\epsilon\leq R$ such that
	\begin{equation}\label{eq:reg2}
		\mathcal{C}_\epsilon^+(y,\bm{e}_y)\cap B_{R_\epsilon}(y)\cap\partial D\sub\omega_1\quad\text{and}\quad \mathcal{C}_\epsilon^-(y,\bm{e}_y)\cap B_{R_\epsilon}(y)\cap \partial D\sub \omega_2
	\end{equation}
	for all $y\in \mathcal{Z}_2^{\partial D}(u)\cap B_{R_\epsilon}$.
	Indeed, this easily follows from \eqref{eq:reg1} and \Cref{lemma:flatness}. For $r\leq R_\epsilon$, we denote $B_r'':=\{x\in B_r\colon x_d=x_{d-1}=0\}$ and we define, for any $x''\in B_r''$
	\begin{equation*}
		S_{x''}^+:=\{(x'',t,0)\colon t\in\R\}\cap B_{R_\epsilon}'\cap \omega_1\quad\text{and}\quad         S_{x''}^-:=\{(x'',t,0)\colon t\in\R\}\cap B_{R_\epsilon}'\cap \omega_2.
	\end{equation*}
	In view of \eqref{eq:reg2} one can easily see that $S_{x''}^+$ contains the segment $\{(x'',t,0)\colon t>\epsilon R_\epsilon\}\cap B_{R_\epsilon}$ and, respectively, $S_{x''}^-$ contains the segment $\{(x'',t,0)\colon t<-\epsilon R_\epsilon\}\cap B_{R_\epsilon}$, for any $x''\in B_r''$; this, in turn, implies that the function 
	\begin{align*}
		g\colon B_r'' &\to \R, \qquad g(x''):=\inf\{t\in\R\colon (x'',T,0)\in \omega_1~\text{for all }T\in(t,r)\}
	\end{align*}
	is well defined. We consider $y''\in B_r''$ and denote $y:=(y'',g(y''),0)$. By construction, we have that $y\in \partial\omega_1\cap B_{R_\epsilon}$ and that
	\[
	-\epsilon|y''|\leq g(y'')\leq \epsilon|y''|,
	\]
	which directly implies that $|y|\leq r\sqrt{1+\epsilon^2}\leq \sqrt{2}r$. We now claim that, for $r\leq R_\epsilon$ small enough, we have that
	\begin{equation}\label{eq:reg3}
		\mathcal{C}_{2\epsilon}^+(y,\bm{e}_{d-1})\cap B_{R_\epsilon}(y)\cap\partial D\sub\omega_1\quad\text{and}\quad \mathcal{C}_{2\epsilon}^-(y,\bm{e}_{d-1})\cap B_{R_\epsilon}(y)\cap\partial D\sub \omega_2,
	\end{equation}
	which is a uniform cone condition. Since from \eqref{eq:reg2} there holds
	\begin{equation*}
		\mathcal{C}_{\epsilon}^+(y,\bm{e}_y)\cap B_{R_\epsilon}(y)\cap\partial D\sub\omega_1\quad\text{and}\quad \mathcal{C}_{\epsilon}^-(y,\bm{e}_y)\cap B_{R_\epsilon}(y)\cap\partial D\sub \omega_2
	\end{equation*}
	then \eqref{eq:reg3} is a trivial consequence of the fact that
	\begin{equation*}
		\mathcal{C}_{2\epsilon}^+(y,\bm{e}_{d-1})\sub\mathcal{C}_{\epsilon}^+(y,\bm{e}_y)\quad\text{and}\quad         \mathcal{C}_{2\epsilon}^-(y,\bm{e}_{d-1})\sub\mathcal{C}_{\epsilon}^-(y,\bm{e}_y)
	\end{equation*}
	which, in turn, reduces to prove that
	\begin{equation}\label{eq:reg4}
		\pm(x-y)\cdot\bm{e}_y> \epsilon|x-y|\quad\text{for all }x\in \mathcal{C}_{2\epsilon}^{\pm}(y,\bm{e}_{d-1}).
	\end{equation}
	Now, let $C_{\textup{osc}}>0$ and $\theta$ be as in \Cref{lemma:osc} and let $r$ be such that
	\begin{equation*}
		C_{\textup{osc}} \theta(\sqrt{2}r)\leq \epsilon.
	\end{equation*}
	Then, in view of \Cref{lemma:osc} and the inequality above we have that
	\[
	\pm(x-y)\cdot\bm{e}_y=\pm (x-y)\cdot\bm{e}_{d-1}\pm (x-y)\cdot(\bm{e}_y-\bm{e}_{d-1})>2\epsilon|x-y|-C_{\textup{osc}} \theta(\sqrt{2}r)|x-y|\geq \epsilon|x-y|,
	\]
	for all $x\in\mathcal{C}_{2\epsilon}^{\pm}(y,\bm{e}_{d-1})$, which proves \eqref{eq:reg4}. As a consequence of \eqref{eq:reg3}, we have that the sets $S^{\pm}_{x''}$ are segments for any $x''\in B_r''$ and, in particular,
	\begin{align*}
		& B_r''\times(-r,r)\cap \omega_1=\{ (x'',t,0)\colon x''\in B_r''~\text{and }g(x'')<t<r\} \\
		& B_r''\times(-r,r)\cap \omega_2=\{ (x'',t,0)\colon x''\in B_r'',~\text{and }-r<t<g(x'')\}.
	\end{align*}
	Moreover,
	\begin{align*}
		B_r''\times(-r,r)\cap\mathcal{Z}_2^{\partial D}(u)&=B_r''\times(-r,r)\cap \partial\omega_1\cap\partial\omega_2 \\
		&=\{ (x'',t,0)\colon x''\in B_r'',~t\in(-r,r)~\text{and }t=g(x'')\}
	\end{align*}
	and $g$ is Lipschitz continuous on $B_r''$ (this is a consequence of \eqref{eq:reg3}). At this point, regularity of $g$ trivially follows from the fact that, if $x=(x'',g(x''),0)\in\mathcal{Z}_2^{\partial D}(u)$, then the normal vector to the graph of $g$ at the point $x$ is exactly $\bm{e}_x$ and this is continuous with respect to $x$, with modulus of continuity $\theta$, in view of \Cref{lemma:osc}. This concludes the proof when $\partial D$ is a plane near $0$. If $\partial D$ is not flat, we consider the diffeomorphism $\Phi(x):=(x',x_d-\varphi(x'))$ which maps $\partial D$ onto $B_R'$ in a neighborhood of the origin and we consider the transformed solution $u^\Phi$. Being $\sigma$ the modulus of continuity of $\nabla \varphi$, from \Cref{cor:blowup} one can easily obtain that
	\[
	\sum_{i=1}^N\norm{u_i^\Phi-P_i}_{L^2(B_r)}^2\leq C\left( \int_0^r\frac{\sigma_0(t)}{t}\d t+\sigma(r)\right)
	\]
	for $r$ sufficiently small and some $C>0$ depending only on $d$, $D$ and $N$. Now, thanks to \eqref{eq:sigma_nondeg_th3} the estimate above implies \eqref{eq:reg1} and we can repeat the same argument as in the flat case. 
\end{proof}

The following clean-up result for interior regular points was already known in its qualitative version, see e.g. in \cite{CL2007}. However, we need a quantitative version: the proof can be obtained by following the very same argument used to prove \Cref{l:clean-up-boundary-1}, by making use of \Cref{prop:int_blowup}.
\begin{lemma}\label{lemma:clean_up_interior}
	For any $\delta>0$ there exists $\rho_2,\epsilon_2\in(0,1/3)$ depending on $d$, $\dist(x_0,\partial D)$, $N$ and $\delta$ such that, if
	\[
	\sum_{i=1}^N\norm{u_i-P_i}_{L^2(B_r(x_0))}^2\leq r^{d+2}\epsilon_2,
	\]
	for some $x_0\in D$ and $r\leq \dist(x_0,\partial D)$, where
	\[P=(a((x-x_0)\cdot\bm{e}_{x_0})^+,a((x-x_0)\cdot\bm{e}_{x_0})^-,0,\dots,0)\]
	for some $a\in( \delta,\frac1{\delta})$ and $\bm{e}_{x_0}\in\partial B_1$, then $u_i\equiv 0$ in $B_{r\rho_2}(x_0)$ for all $i=3,\dots,N$. Moreover, the interface $\mathcal F(u)$ in the ball $B_{r\rho_2}(x_0)$ (with $\mathcal{F}(u)$ being as in \eqref{eq:free_boundary}) is a $C^{1,\alpha}$ manifold
	$$\mathcal M=\mathcal F(u)\cap B_{r\rho_2}(x_0)=\mathcal R(u)\cap B_{r\rho_2}(x_0)=\partial\Omega_1\cap B_{r\rho_2}(x_0)=\partial\Omega_2\cap B_{r\rho_2}(x_0)$$
	whose normal $\nnu_{\mathcal M}$	is a $C^{0,\alpha}$ vector such that $|\nnu_{\mathcal M}-\bm e_{x_0}|\le C\epsilon_2$ for some constant $C>0$ depending only on $\delta,d,\dist(x_0,\partial D),N$. \end{lemma}
\begin{proof}
	By the rate of convergence of the $1$-homogeneous rescalings to the blow-up limit (\Cref{prop:int_blowup}) we have that for any point of the interior nodal set $y_0\in \mathcal F(u)$, there are indices $i\neq j$ and $\bar{P}=\bar{P}^{y_0}=(\bar{P}_1,\dots,\bar{P}_N)$ such that 
	\[
	\bar{P}_i(x)=\bar{a}_{y_0}((y_0-x)\cdot\bm{e}_{y_0})^+\ ;\quad 
	\bar{P}_j(x)=\bar{a}_{y_0}((y_0-x)\cdot\bm{e}_{y_0})^-\ ;\quad 
	\bar{P}_k(x)\equiv 0\quad\text{when}\quad k\neq i,j,
	\]
	for some $\bm{e}_{y_0}\in\partial B_1$ and some $\bar{a}_{y_0}\ge 0$, and we have 
	\[
	\sum_{i=1}^N\norm{u_i-\bar{P}_i}_{L^2(B_r(y_0))}^2\leq \bar{C}_{\textup{BU}}H(u,\bar{R}_0,y_0) r^{d+2+\bar{\alpha}},
	\]
	for every $r$ such that $r<\text{\rm dist}(y_0,\partial D)$ and $r\le \bar{R}_0$, $\bar{C}_{\textup{BU}}$ being a constant depending on $\text{\rm dist}(y_0,\partial D)$. 
	
	We now proceed with the proof of the lemma. Suppose that there is a point $y_0$ of $\Omega_k$, $k\ge 3$, in the ball $B_{r\rho}(x_0)$. Let $z_0$ be the projection of $y_0$ on the boundary of $\Omega_k$. Then, $z_0\in B_{2r\rho_2}(x_0)$ and by the Hopf maximum principle, $z_0$ is a point of frequency $1$ and the one-homogeneous blow-up at $z_0$ has a non-zero $k$th component. In particular, at least one between the first and the second component is identically vanishing. Let us assume it is the second one and let $\rho\in(\rho_2,1/3)$. Then 
	\begin{align*}
		\int_{B_{r\rho}(x_0)}|a((x_0-x)\cdot\bm{e}_{x_0})^+|^2&\le 4 \int_{B_{r}(x_0)}|a((x_0-x)\cdot\bm{e}_{x_0})^+-u_2(x)|^2\dx\\
		&+ 4\int_{B_{r\rho}(z_0)}|\bar{a}_{z_0}((z_0-x)\cdot\bm{e}_{z_0})^+-u_2(x)|^2\dx\\
		&+ 4\int_{B_{r\rho}}|u_2(x+z_0)-u_2(x+x_0)|^2\dx\\
		&\le 4r^{d+2}\epsilon_2+4\bar{C}_{\textup{BU}}H(u,\bar{R}_0,y_0) (r\rho)^{d+2+\bar{\alpha}}+4|B_{r\rho}|C_L^2|x_0-z_0|^2,
	\end{align*}
	where $C_L$ is the Lipschitz constant of $u$. Since $a\ge \delta$, we get 
	\begin{align*}
		\delta^2(r\rho)^{d+2}&\le C\Big(r^{d+2}\epsilon_2+(r\rho)^{d+2+\bar\alpha}+(r\rho)^{d}|x_0-z_0|^2\Big)\\
		&\le Cr^{d+2}\Big(\epsilon_2+\rho^{d+2+\bar\alpha}+\rho^{d}\rho_2^2\Big),
	\end{align*}
	for some $C>0$ depending only on $d$, $\dist(y_0,\partial D)$, $C_L$, $N$. Choosing $\rho:=\rho_2^{\frac{2}{2+\bar\alpha}}$, we get 
	\begin{align*}
		\delta^2\le C\Big(\epsilon_2{\rho_2^{-\frac{2d+4}{2+\bar\alpha}}}+2\rho_2^{\frac{2\bar\alpha}{2+\bar\alpha}}\Big),
	\end{align*}
	which, by taking $\rho_2$ and $\epsilon_2$ small enough, implies that $\Omega_k\cap B_{\rho_2r}(x_0)=\emptyset$ for $k\ge 3$. The same argument implies that in the interior of $B_{\rho_2r}(x_0)$ all the points have frequency $1$. The regularity and the modulus of continuity of the normal vector on the interface between $\Omega_1$ and $\Omega_2$ follow by the implicit function theorem.
\end{proof}

\begin{figure}
	\centering
			\includegraphics[scale=0.28]{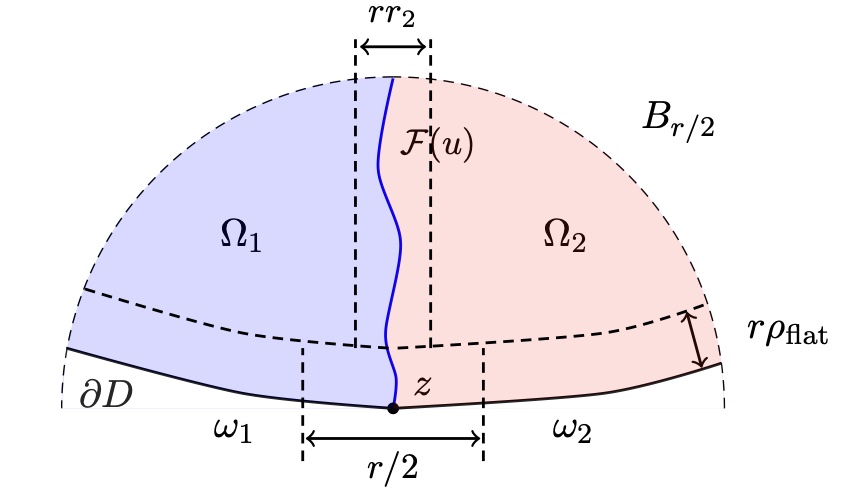}
	\caption{Clean-up at points of frequency two}
	\label{fig:clean-up}
\end{figure}

We are now ready to prove the full clean-up result for boundary points of frequency $2$. Since $\mathcal{Z}_2^{\partial D}(u)$ coincides with $\mathcal{R}_{\partial D}(u)$, the following result immediately implies \Cref{thm:up_to_the_bound}.
\begin{proposition}\label{prop:clean_up_2}
	Let $x_0\in\mathcal{Z}_2^{\partial D}(u)$ and let  $P^{x_0,2}$ be as in \Cref{cor:blowup}, i.e. 
	\begin{equation*}
		\begin{aligned}
			&P_j^{x_0,2}=a_{x_0,2}(x\cdot\bm{e}_{x_0})^-(-x\cdot\nnu(x_0))^+,\\
			&P_k^{x_0,2}=a_{x_0,2}(x\cdot\bm{e}_{x_0})^+(-x\cdot\nnu(x_0))^+,\\
			&		P_i^{x_0}=0\quad\text{for all }i\neq j,k
		\end{aligned}
	\end{equation*}
	for some $\bm e_{x_0}\in\partial B_1$ (such that $\bm e_{x_0}\cdot\nnu(x_0)=0$), some $j,k\in\{1,\dots,N\}$ and some $a_{x_0,2}>0$. Then, there exists $R>0$ depending on $d$, $D$, $N$ and $x_0$ such that:
	\begin{enumerate}
		\item $u_i\equiv 0$ in $D\cap B_R(x_0)$ for all $i\neq j,k$;
		\item in $B_R(x_0)$ we have that 
		\[\mathcal{F}(u)=\mathcal{R}(u)=\partial\Omega_j\cap D= \partial\Omega_k\cap D.\]
		Moreover, $B_R(x_0)\cap \overline{\mathcal{F}(u)}$ is a $(d-1)$-dimensional manifold (with boundary $B_R(x_0)\cap\mathcal{F}_{\partial D}(u)$) of class $C^1$.

	\end{enumerate}
\end{proposition}
\begin{proof}
	The proof follows by applying \Cref{lemma:flatness}, \Cref{prop:reg} and \Cref{lemma:clean_up_interior} at any scale. Let us be more precise. Without loss of generality, in the proof we can assume $x_0=0$, $\nnu(0)=-\bm{e}_d$, $\bm{e}_{x_0}=\bm{e}_{d-1}$, $j=1$ and $k=2$. By \Cref{prop:reg} there exists $R\in(0,R_0)$ such that
	\[
	B_R\cap\partial D=B_R\cap\left(\overline{\omega_1}\cup        \overline{\omega_2}\right)
	\]
	and such that, in $B_R$, $\mathcal F_{\partial D}(u)=\mathcal{Z}_2^{\partial D}(u)$ is a smooth interface (of points of frequency $2$) separating $\omega_1$ and $\omega_2$. We will next extend this clean-up to the interior of $D$; we will use the geometric construction from Figure \ref{fig:clean-up}.
	Let $r\leq R/2$ and $z\in\mathcal{Z}^{\partial D}_2(u)\cap B_{R/2}$ be arbitrarily chosen.   From \Cref{lemma:flatness} (choosing $\rho=1/2$ and $\eta=1/2$) and \Cref{cor:blowup} we know that there exists $\rho_{\textup{flat}}\in(0,1)$ such that in $B_{r/2}(z)$ we have
	\begin{multline*}
		\{x\in D\colon \dist(x,\partial D)<r\rho_{\textup{flat}}\}\cap\{|x_{d-1}|>r/4\} \\ 
		= \{x\in \Omega_1\cup\Omega_2\colon \dist(x,\partial D)<r\rho_{\textup{flat}}\}\cap\{|x_{d-1}|>r/4\}.
	\end{multline*}
	From \Cref{lemma:clean_up_interior} and \Cref{cor:blowup} we know that there exists $r_2<\rho_{\textup{flat}}$ such that for any $y$ in the vertical strip 
	$$\{x:(x-z)\cdot e_{d-1}=0\}\cap\{\dist(x,\partial D)>r\rho_{\textup{flat}}\} \cap B_{R/2}(z)$$ 
	we have that 
	\begin{equation*}
		D=\overline{\Omega_1}\cup\overline{\Omega_2}\quad\text{in}\quad B_{rr_2}(y).
	\end{equation*}
	Thus, for every $k\ge 3$, we get 
	$$\begin{cases}
		\Omega_k\cap\Big(\{x:|x_{d-1}-z_{d-1}|< rr_2\}\cup\{\dist(x,\partial D)>r\rho_{\textup{flat}}\}\Big) \cap B_{r}(z)=\emptyset\\
		\Omega_k\cap\Big(\{x:|x_{d-1}-z_{d-1}|> r/2\}\cup\{\dist(x,\partial D)<r\rho_{\textup{flat}}\}\Big) \cap B_{r}(z)=\emptyset.\\
	\end{cases}$$
	Now, by the rate of convergence of the $2$-homogeneous rescalings $\{u^{t,z,2}\}_{t>0}$, we can find $r_3>0$ (universal for all $z$ in $B_{R/2}$) such that
	$$\Omega_k\cap\Big(\{x:|x_{d-1}-z_{d-1}|\ge rr_2\}\cup\{\dist(x,\partial D)\ge r\rho_{\textup{flat}}\}\Big) \cap B_{r}(z)=\emptyset,$$
	for every $r<r_3$.
	From this we obtain that for every $r<r_3$ and every $k\neq 1,2$,
	$$\Omega_k\cap\Big(B_r(z)\setminus \Big(\{x:|x_{d-1}-z_{d-1}|< r/2\}\cap\{\dist(x,\partial D)< r\rho_{\textup{flat}}\}\Big)\Big) =\emptyset.$$
	Since $r$ and $z$ are arbitrary, we get that $\Omega_k\cap B_{r_3}=\emptyset$ for $k\ge 3$, and that the interface $\partial\Omega_1\cap\partial\Omega_2$ is $C^1$ up to the fixed boundary $\partial D$. 
\end{proof}

\subsection{Structure of the free boundary in dimension two}\label{sub:proof-of-theorem-2D}
In this section we prove Theorem \ref{thm:up_to_the_bound_2D}. Suppose that $x_0$ is a point on $\mathcal F_{\partial D}(u)$ and let $\gamma\ge 1$ be the frequency at $x_0$. Then, the following holds.
\begin{itemize}
	\item Every $\gamma$-homogeneous blow-up $P=(P_1,\dots,P_N)$ of $u$ at $x_0$ belongs to the class $\mathcal S$ (defined in the proof of \Cref{prop:topology}) of the half-plane $\mathcal H=\{x\in\R^2\ :\ x\cdot\nnu(x_0)>0\}$; in particular, the nodal set $\{|P|=0\}\cap\mathcal{H}$ has zero measure.
	\item We claim that $\gamma$ is an integer greater than $1$ and the nodal set $\{|P|=0\}$ splits the half-plane $\mathcal H$ in $\gamma$ equal sectors $S_1,\dots, S_\gamma$ (we notice that two non-touching sectors may belong belong to the same positivity set $\{P_i>0\}$ and that the numeration of the sectors is not corresponding to the one of the components of the blow-up); indeed, the set $\{|P|>0\}$ is the union of disjoint open cones. If $S$ is one of these cones, then there is some $i\in\{1,\dots,N\}$ such that $\Delta P_i=0$ in $S$, $P_i>0$ in $S$, $P_i=0$ on $\partial S$. Since $P_i$ is $\gamma$-homogeneous, the opening of the cone $S$ has to be exactly $\pi/\gamma$. Now, the claim follows since the only way to fit a finite number of disjoint cones with the same opening in the half-plane $\mathcal H$ (without leaving space) is to take $\gamma$ to be an integer.
	\item The blow-up $P$ at $x_0$ is unique. 
	Indeed, suppose that there is another blow-up $Q$ such that $P\neq Q$ on one of the sectors $S_k$ (say, on $S_1$). Then, there are two distinct indices $i\neq j\in\{1,\dots,N\}$ such that $S_1\subset \{P_i>0\}$ and at the same time $S_1\subset\{Q_j>0\}$. If we pick any point $y_0\subset B_1$ in the open sector $S_1$, then we can find two sequences of radii $(r_k^P)_{k\ge 1}$ and $(r_k^Q)_{k\ge 1}$ such that $u_i(x_0+r_k^Py_0)>0$ and $u_j(x_0+r_k^Qy_0)>0$. By the continuity of $u$, we can find a third sequence $r_k\to0$ such that $u_j(x_0+r_ky_0)=0$; thus, we find a $\gamma$-homogeneous blow-up that vanishes in $y_0$, which is impossible.  
\end{itemize}
Finally, the claim follows by the clean-up results for the points of frequency $1$ in the interior (\Cref{lemma:clean_up_interior}) and on the boundary (\cref{l:clean-up-boundary-1}).  \qed

\section*{Acknowledgements}

The authors are supported by the European Research Council (ERC), through the European Union’s
Horizon 2020 project ERC VAREG - Variational approach to the regularity of the free boundaries (grant agreement No. 853404). The authors acknowledge the MIUR Excellence Department Project awarded to the Department of Mathematics, University of Pisa, CUP I57G22000700001. R.O. acknowledges support from the 2023 INdAM-GNAMPA project 2023 no. \texttt{CUP\_E53C22001930001} and from the 2024 INdAM-GNAMPA project no. \texttt{CUP\_E53C23001670001}. B.V. acknowledges support from the projects PRA\_2022\_14 GeoDom (PRA 2022 - Università di Pisa) and MUR-PRIN “NO3” (n.2022R537CS). 

\bibliographystyle{aomalpha} 
\bibliography{biblio.bib}

\end{document}